\documentclass[a4paper, 12pt]{amsart}
\usepackage[margin=1in]{geometry}
\usepackage{amsmath}
\usepackage{amssymb}
\usepackage{enumitem}
\usepackage{mathtools}

\usepackage[pdftex]{hyperref}
\hypersetup{
  colorlinks = true,
  linkcolor =blue,
  anchorcolor = red,
  citecolor = blue,
  urlcolor = blue
}

\usepackage{tikz}
\tikzset{>=stealth}
\tikzset{blacknode/.style={circle, draw, fill=black, text=white, minimum size=0.7cm}}
\tikzset{whitenode/.style={circle, draw, fill=white, text=black, minimum size=0.7cm}}
\tikzset{smallblacknode/.style={circle, draw, fill=black, text=white, minimum size=0.5cm}}
\tikzset{smallwhitenode/.style={circle, draw, fill=white, text=black, minimum size=0.5cm}}
\tikzset{miniblacknode/.style={circle, draw, fill=black, minimum size=0cm}}
\tikzset{miniwhitenode/.style={circle, draw, fill=white, minimum size=0cm}}
\tikzset{graphedge/.style={draw=black,very thick}}

\usepackage{amsthm}
\usepackage{thmtools}
\newtheorem{thm}{Theorem}
\newtheorem{theorem}{Theorem}[section]
\newtheorem{lemma}[theorem]{Lemma}
\newtheorem{prop}[theorem]{Proposition}
\newtheorem{corollary}[theorem]{Corollary}
\newtheorem*{con}{Construction}
\theoremstyle{definition}
\newtheorem{definition}[theorem]{Definition}
\newtheorem{remark}[theorem]{Remark}

\renewcommand{\*}{\mathbf}
\newcommand{\mf}{\mathfrak}
\newcommand{\mc}{\mathcal}
\newcommand{\ul}{\underline}
\newcommand{\ol}{\overline}

\newcommand{\N}{{\mathbb N}}
\newcommand{\Z}{{\mathbb Z}}

\newcommand{\C}{{\mathbb C}}
\newcommand{\V}{{\mathbb V}}
\newcommand{\lmod}{\text{-mod}}
\DeclareMathOperator{\Stab}{Stab}

\begin{document}

\title{Classifying quotients on Coxeter groups by isomorphism in Bruhat order}
\author{Joseph Newton}
\maketitle

\begin{abstract}
We classify all quotients $W/W_J$ up to isomorphism in Bruhat order, with $(W,S)$ a Coxeter system and $W_J$ a parabolic subgroup of $W$. In particular, the non-trivial isomorphisms fall into a small number of cases which are highly restricted; all have $W$ finite and $W_J$ a maximal parabolic. This has the immediate application of classifying dominant and antidominant blocks of category $\mc O$ for Kac-Moody algebras up to equivalence.
\end{abstract}

\section{Introduction}

This paper extends a result from Coulembier in \cite{Co} which classifies all quotients on finite Coxeter groups up to isomorphism in Bruhat order. This is motivated by the study of blocks of the BGG category $\mc O$ for finite semisimple Lie algebras, where it is shown that two blocks are equivalent precisely when their corresponding posets via the Weyl group are isomorphic. Here we extend the classification of Bruhat posets to all Coxeter systems, finite and infinite, and subsequently the classification of blocks of category $\mc O$ extends to more general Kac-Moody algebras.

Our main result is as follows:

\begin{thm}\label{thm:one}
Suppose $W,U$ are irreducible Coxeter groups with parabolic subgroups $W_J,U_K$ respectively. Then $W/W_J$ and $U/U_K$ are isomorphic as posets with the Bruhat order if and only if the pairs $(W,W_J)$ and $(U,U_K)$ are one of the following cases:
\begin{enumerate}
	\item $(I_2(n),A_1)\leftrightarrow(A_{n-1},A_{n-2})$ for $n\geq4$;
	\item $(B_n,B_{n-1})\leftrightarrow(A_{2n-1},A_{2n-2})$ for $n\geq3$;
	\item $(B_n,B_{n-1})\leftrightarrow(I_2(2n),A_1)$ for $n\geq3$;
	\item $(B_n,A_{n-1})\leftrightarrow(D_{n+1},A_n)$ for $n\geq3$;
	\item $(H_3,H_2)\leftrightarrow(D_6,D_5)$ (denoting $H_2=I_2(5)$);
	\item $W=W_J$ and $U=U_K$;
	\item There is a group isomorphism $W\to U$ that restricts to an isomorphism $W_J\to U_K$.
\end{enumerate}
\end{thm}

In short, we have the remarkable result that an irreducible Coxeter group can be recovered from only the Bruhat order on a single non-trivial quotient, with only a few exceptions given by (1)-(5). We prove that these exceptional pairs are isomorphisms in Section~\ref{sec:coxeter}, and prove that there are no other isomorphisms in Sections~\ref{sec:chainlike} and \ref{sec:sudoku}. This also extends to reducible Coxeter systems as follows:

\begin{thm}\label{thm:reducible}
Suppose $(W,S)$, $(U,T)$ are Coxeter systems and $J\subseteq S$, $K\subseteq T$. Then $W/W_J\cong U/U_K$ as posets with the Bruhat order if and only if there are disjoint unions $S=S_1\sqcup\dots\sqcup S_n$ and $T=T_1\sqcup\dots\sqcup T_n$ with $n\in\N$ (and possibly with some $S_i$ or $T_i$ empty) so that each of $(W_{S_i},S_i)$ and $(U_{T_i},T_i)$ are Coxeter systems that are irreducible or trivial and $W_{S_i}/W_{S_i\cap J}\cong U_{T_i}/U_{T_i\cap K}$ as posets for all $1\leq i\leq n$.
\end{thm}

This is proven in \cite[Theorem 3.2.1]{Co}, but we also provide a proof in Section~\ref{subsec:reducible}. Figure \ref{fig:reducible} shows an example of an isomorphism between posets of reducible systems (refer to Section~\ref{sec:coxeter} for the notation). Following the same methods in \cite{Co}, we now have an immediate application:

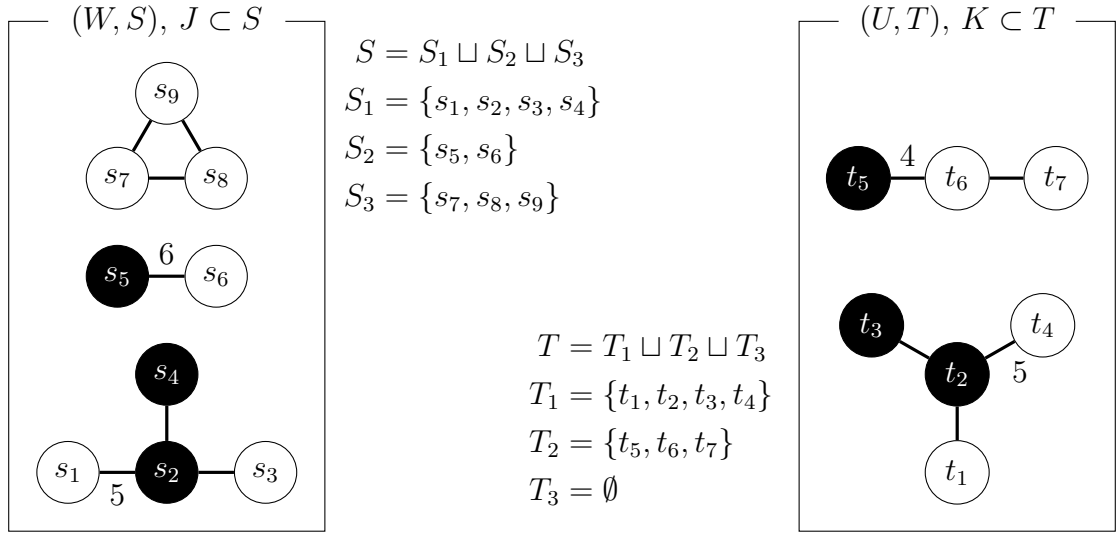
\begin{figure}[h]
\begin{center}
\begin{tikzpicture}[scale=1.3]
\def\marg{0.6}
\def\sep{1.2}
\begin{scope}[shift=({-4,0})]
\draw (\sep,4+\marg) -- (1+\marg,4+\marg) -- (1+\marg,-\marg) -- (-1-\marg,-\marg) -- (-1-\marg,4+\marg) -- (-\sep,4+\marg);
\node at (0,4+\marg) {$(W,S),\,J\subset S$};
\node[whitenode] (s1) at (-1,0) {$s_1$};
\node[blacknode] (s2) at (0,0) {$s_2$};
\node[whitenode] (s3) at (1,0) {$s_3$};
\node[blacknode] (s4) at (0,1) {$s_4$};
\node[blacknode] (s5) at (-0.5,2) {$s_5$};
\node[whitenode] (s6) at (0.5,2) {$s_6$};
\node[whitenode] (s7) at (-0.5,3) {$s_7$};
\node[whitenode] (s8) at (0.5,3) {$s_8$};
\node[whitenode] (s9) at (0,3.866) {$s_9$};
\begin{scope}[every edge/.style=graphedge]
\path (s2) edge node[below] {$5$} (s1) edge (s3) edge (s4);
\path (s5) edge node[above] {$6$} (s6);
\path (s7) edge (s8) edge (s9); \path (s8) edge (s9);
\end{scope}
\end{scope}
\begin{scope}[shift=({4,0})]
\draw (\sep,4+\marg) -- (1+\marg,4+\marg) -- (1+\marg,-\marg) -- (-1-\marg,-\marg) -- (-1-\marg,4+\marg) -- (-\sep,4+\marg);
\node at (0,4+\marg) {$(U,T),\,K\subset T$};
\node[whitenode] (t1) at (0,0) {$t_1$};
\node[blacknode] (t2) at (0,1) {$t_2$};
\node[blacknode] (t3) at (-0.866,1.5) {$t_3$};
\node[whitenode] (t4) at (0.866,1.5) {$t_4$};
\node[blacknode] (t5) at (-1,3) {$t_5$};
\node[whitenode] (t6) at (0,3) {$t_6$};
\node[whitenode] (t7) at (1,3) {$t_7$};
\begin{scope}[every edge/.style=graphedge]
\path (t2) edge (t1) edge (t3) edge node[below right] {$5$} (t4);
\path (t6) edge node[above] {$4$} (t5) edge (t7);
\end{scope}
\end{scope}
\node at (-0.9,3.5) {$\begin{aligned}S&=S_1\sqcup S_2\sqcup S_3\\S_1&=\{s_1,s_2,s_3,s_4\}\\S_2&=\{s_5,s_6\}\\S_3&=\{s_7,s_8,s_9\}\end{aligned}$};
\node at (0.9,0.5) {$\begin{aligned}T&=T_1\sqcup T_2\sqcup T_3\\T_1&=\{t_1,t_2,t_3,t_4\}\\T_2&=\{t_5,t_6,t_7\}\\T_3&=\emptyset\end{aligned}$};
\end{tikzpicture}
\end{center}
\caption{Coxeter systems and subgroups with isomorphic quotient posets $W^J\cong U^K$ by Theorem~\ref{thm:reducible}. Generators in $J,K$ are drawn in white, others are drawn in black. The isomorphisms $W_{S_i}/W_{S_i\cap J}\cong U_{T_i}/U_{T_i\cap K}$ for $i=1,2,3$ are examples of Theorem~\ref{thm:one} cases (7), (3), (6) respectively.}\label{fig:reducible}
\end{figure}

\begin{thm}\label{thm:two}
Let $\mf g\supset\mf b\supset\mf h$ and $\mf g'\supset\mf b'\supset\mf h'$ be symmetrizable Kac-Moody algebras over $\C$ with corresponding Borel and Cartan subalgebras, and $\mc W,\mc W'$ their Weyl groups. Suppose $\mc O_\Lambda$, $\mc O'_{\Lambda'}$ are blocks of the categories $\mc O$, $\mc O'$ for $\mf g$, $\mf g'$ respectively, that are outside the critical hyperplanes and such that there are weights $\lambda\in\Lambda\subset\mf h^*$ and $\lambda'\in\Lambda'\subset(\mf h')^*$ which are either both dominant or both antidominant. Then there is an equivalence of categories $\mc O_\Lambda\simeq\mc O'_{\Lambda'}$ if and only if $\mc W(\Lambda)/\Stab(\lambda)\cong\mc W'(\Lambda')/\Stab(\lambda')$ as posets with the Bruhat order.
\end{thm}

We follow the notation in \cite{Fi} and \cite{BS} for the statement of this theorem, with more details in Section~\ref{subsec:categoryproof}. In the notation of \cite{Co} this may be written more succinctly as
\[\mc O(W,W_J)\simeq\mc O(U,U_K)\iff W/W_J\cong U/U_K\]
where $(W,S),(U,T)$ are any crystallographic Coxeter systems and $J\subseteq S$, $T\subseteq K$. The proof of Theorem~\ref{thm:two} follows the same method as in \cite{Co} and is given in Section~\ref{subsec:categoryproof}.

\subsection*{Acknowledgement} This research was completed during an honours project supervised by Kevin Coulembier; the author gives many thanks for his guidance and discussions.

\section{Background}\label{sec:coxeter}

The study of Coxeter systems presented here is sourced from \cite{BB}, and we will use their notation throughout.

\subsection{Coxeter systems}  Let $S$ be a finite set. A Coxeter matrix is a function $m:S\times S\to\{1,2,\dots,\infty\}$ that is symmetric and has $m(s,t)=1\iff s=t$ for all $s,t\in S$. The Coxeter group $W$ corresponding to $S$ and $m$ has generator presentation
\[W=\langle S\mid (st)^{m(s,t)}=e\text{ for all }s,t\in S\rangle\]
where $e$ is the group identity. We write $\ell(w)$ for the minimal number $n\in\N$ so that $w\in W$ can be written as a product of $n$ generators $s_1\dots s_n$, and such a product is called a reduced expression. The Bruhat order is the partial order on $W$ where $u<v$ means that there are reflections $t_1,\dots,t_n\in\{wsw^{-1}\mid w\in W,s\in S\}$ with $v=ut_1\dots t_n$ and $\ell(ut_1\dots t_i)<\ell(ut_1\dots t_{i+1})$ for all $i$. We denote a covering relation in Bruhat order by $\vartriangleleft$. Denote by $W_J$ the subgroup generated by $J\subseteq S$, called a parabolic subgroup (conjugates of $W_J$ are also called parabolic, but these can also be written in the form $W_J$ for some choice of $S$ and $J$). The left cosets of $W_J$ each have a minimal-length representative, and we denote the set of these representatives $W^J$. The Bruhat order restricted to $W^J$ then induces an order on the quotient $W/W_J$.

Let us now rewrite all of the above in a more combinatorial language which will be used throughout this paper. $W$ is the quotient on the free group of words $S^*$ by the equivalence relation generated by:
\begin{itemize}
	\item $ss=e$ for each $s\in S$ (called \textit{nil-moves} when performed in the direction $ss\to e$);
	\item $\underbrace{ststs\dots}_{m(s,t)\text{ terms}}=\underbrace{tstst\dots}_{m(s,t)\text{ terms}}$ for distinct $s,t\in S$ with $m(s,t)\neq\infty$ (called \textit{braid-moves}).
\end{itemize}
We will sometimes distinguish elements of $S^*$ from elements of $W$ by writing them with an underline, e.g. $\ul w$. We recall a number of results from \cite{BB} which relate the Bruhat order to reduced expressions:
\begin{theorem}
Let $w,u\in W$.
\begin{enumerate}[itemsep=5pt]
	\item Word Property: Any expression for $w$ can be transformed into a reduced expression for $w$ by a sequence of nil-moves and braid-moves, and any two reduced expressions for $w$ are connected via a sequence of braid-moves.
	\item Subword Property: If $u\leq w$ and $w=s_1s_2\dots s_q$ is a reduced expression, then there exists a reduced expression $u=s_{i_1}s_{i_2}\dots s_{i_k}$ with $1\leq i_1<\dots<i_k\leq q$ (called a subword). Conversely, if there is a reduced expression for $u$ that is a subword of a reduced expression for $w$ then $u\leq w$.
	\item $u\vartriangleleft w$ if and only if $u<w$ and $\ell(w)=\ell(u)+1$, that is any reduced expression for $w$ can be turned into a reduced expression for $u$ by removing one generator.
	\item Chain Property: If $u<w$, then there exist elements $x_0,\dots,x_k\in W$, $k\geq1$ with $u=x_0\vartriangleleft x_1\vartriangleleft \dots\vartriangleleft x_k=w$.
	\item $w\in W^J$ if and only if the rightmost generator in every reduced expression for $w$ is not in $J$.
	\item Properties (1), (2), (3) and (4) hold in $W^J$ also.
\end{enumerate}
\end{theorem}

Throughout this paper we will focus on the set $W^J$ as opposed to the quotient $W/W_J$, taking (5) as its definition. Some particular cases of the Word Property which we'll use very frequently are as follows:
\begin{corollary}
Let $w\in W$ and $\ul w$ be an expression for $w$.
\begin{itemize}
	\item If no braid-moves or nil-moves can be performed on $\ul w$, then it must be reduced.
	\item If $\ul w$ is reduced then $s\in S$ appears in $\ul w$ if and only if $s\leq w$, that is all reduced expressions for $w$ are comprised of the same set of generators.
\end{itemize}
\end{corollary}

\subsection{Graphs}

We will often notate the information in $m$ using a \textit{Coxeter graph}, which has a node corresponding to each generator and edges between pairs of generators $s,t$ with $m(s,t)\geq3$. If $m(s,t)\geq4$, we label the corresponding edge with $m(s,t)$. Following the notation in \cite{Co}, we notate a pair $(W,W_J)$ of a Coxeter group and parabolic subgroup by colouring the nodes corresponding to elements of $J$ white and all others black, and call the coloured graph a \textit{bw-Coxeter graph}.

To display a poset $W^J$ we draw a graph with a node for each element and an edge for each covering relation $u\vartriangleleft v$. We will always arrange elements vertically by length with the identity element at the bottom. Note that in the literature a \textit{Bruhat graph} also includes edges between $u$ and $v$ whenever $v=tu$ for some reflection $t$ conjugate to $s\in S$; for simplicity we will omit these additional edges.

\subsection{Terminology}

We call a Coxeter system:
\begin{itemize}
	\item \textit{finite} if $W$ is a finite group;
	\item \textit{crystallographic} if $m(s,t)\in\{2,3,4,6,\infty\}$ for all distinct $s,t\in S$;
	\item \textit{simply-laced} if $m(s,t)\in\{2,3\}$ for all distinct $s,t\in S$ (that is, the Coxeter graph has no labelled edges);
	\item \textit{right-angled} if $m(s,t)\in\{2,\infty\}$ for all distinct $s,t\in S$.
\end{itemize}
The finite Coxeter groups are all named, and provided in the Appendix for reference. The notation $(W,W_J)$ can occasionally be ambiguous, for example $(B_3,A_1)$ could refer to 3 different pairs depending on which generator of $B_3$ is taken as the generator of $A_1$. It is easy to check, however, that the cases in Theorem~\ref{thm:one} are not ambiguous in this sense and for each there is only one choice of parabolic subgroup up to relabelling.

\subsection{Isomorphisms of finite Bruhat posets}\label{subsec:classification}

For irreducible systems $(W,S)$ and $(U,T)$ clearly an isomorphism $W^J\cong U^K$ occurs whenever the bw-Coxeter graphs are identical, or more formally when there is a bijection $\sigma:S\to T$ with $m(s_1,s_2)=m(\sigma(s_1),\sigma(s_2))$ and $s_1\in J\iff\sigma(s_1)\in K$ for all $s_1,s_2\in S$, so in this case we call the pairs $(W,W_J)$ and $(U,U_K)$ \textit{isomorphic}. This gives case (7) in Theorem~\ref{thm:one}, while case (6) is where $W^J$ and $U^K$ are trivial, that is each poset only contains the identity.

Now we show that cases (1)-(5) produce isomorphic posets. The crystallographic cases are considered in \cite{Co} and follow from the application to representation theory, but we also present a purely combinatorial approach for these below. There are two new cases not considered in \cite{Co} since they are non-crystallographic; the pair $(I_2(m),A_1)$ may be seen as a generalisation of the cases $(B_2,A_1)$ and $(G_2,A_1)$, while $(H_3,H_2)$ is exceptional but has a relatively simple poset to describe.

\begin{prop}
The poset $W^J$ where the pair $(W,W_J)$ is $(I_2(m),A_1)$, $(A_n,A_{n-1})$ or $(B_k,B_{k-1})$ is the total order on $|W^J|$ elements, with $|W^J|$ equal to $m$, $n+1$ or $2k-1$ respectively. That is, we have the isomorphisms (1), (2) and (3) in Theorem~\ref{thm:one}.
\end{prop}

\begin{figure}[h]
\begin{minipage}{0.32\textwidth}\centering
\begin{tikzpicture}[scale=1]
\node[smallblacknode] (1) at (0,0) {\small{$1$}};
\node[smallwhitenode] (2) at (0,1.1) {\small{$2$}};
\node[smallwhitenode] (3) at (0,2.2) {\small{$3$}};
\node[smallwhitenode] (4) at (0,3.3) {\small{$4$}};
\node[smallwhitenode] (5) at (0,4.4) {\small{$5$}};
\begin{scope}[every edge/.style=graphedge]
\path (2) edge (1) edge (3);
\path (4) edge (3) edge (5);
\end{scope}
\begin{scope}[shift={(1.5,-0.2)}]
\node (e) at (0,0) {\small{$e$}};
\node (1) at (0.0,1) {\small{$1$}};
\node (21) at (0.0,2) {\small{$21$}};
\node (321) at (0.0,3) {\small{$321$}};
\node (4321) at (0.0,4) {\small{$4321$}};
\node (54321) at (0.0,5) {\small{$54321$}};
\path (1) edge (21) edge (e);
\path (321) edge (21);
\path (4321) edge (321) edge (54321);
\end{scope}
\end{tikzpicture}
\[(A_5,A_4)\]
\end{minipage}\hfill
\begin{minipage}{0.32\textwidth}\centering
\begin{tikzpicture}[scale=1]
\node[smallblacknode] (1) at (0,0) {\small{$1$}};
\node[smallwhitenode] (2) at (0,1.1) {\small{$2$}};
\node[smallwhitenode] (3) at (0,2.2) {\small{$3$}};
\begin{scope}[every edge/.style=graphedge]
\path (2) edge (1) edge node[left] {$4$} (3);
\end{scope}
\begin{scope}[shift={(1.5,-0.2)}]
\node (e) at (0,0) {\small{$e$}};
\node (1) at (0.0,1) {\small{$1$}};
\node (21) at (0.0,2) {\small{$21$}};
\node (321) at (0.0,3) {\small{$321$}};
\node (2321) at (0.0,4) {\small{$2321$}};
\node (12321) at (0.0,5) {\small{$12321$}};
\path (1) edge (21) edge (e);
\path (321) edge (21);
\path (2321) edge (321) edge (12321);
\end{scope}
\end{tikzpicture}
\[(B_3,B_2)\]
\end{minipage}\hfill
\begin{minipage}{0.32\textwidth}\centering
\begin{tikzpicture}[scale=1]
\node[smallblacknode] (1) at (0,0) {\small{$1$}};
\node[smallwhitenode] (2) at (0,1.1) {\small{$2$}};
\begin{scope}[every edge/.style=graphedge]
\path (1) edge node[left] {$6$} (2);
\end{scope}
\begin{scope}[shift={(1.5,-0.2)}]
\node (e) at (0,0) {\small{$e$}};
\node (1) at (0.0,1) {\small{$1$}};
\node (21) at (0.0,2) {\small{$21$}};
\node (121) at (0.0,3) {\small{$121$}};
\node (2121) at (0.0,4) {\small{$2121$}};
\node (12121) at (0.0,5) {\small{$12121$}};
\path (1) edge (21) edge (e);
\path (121) edge (21);
\path (2121) edge (121) edge (12121);
\end{scope}
\end{tikzpicture}
\[(I_2(6),A_1)\]
\end{minipage}
\caption{Examples of posets where the Bruhat order is a total order.}
\end{figure}
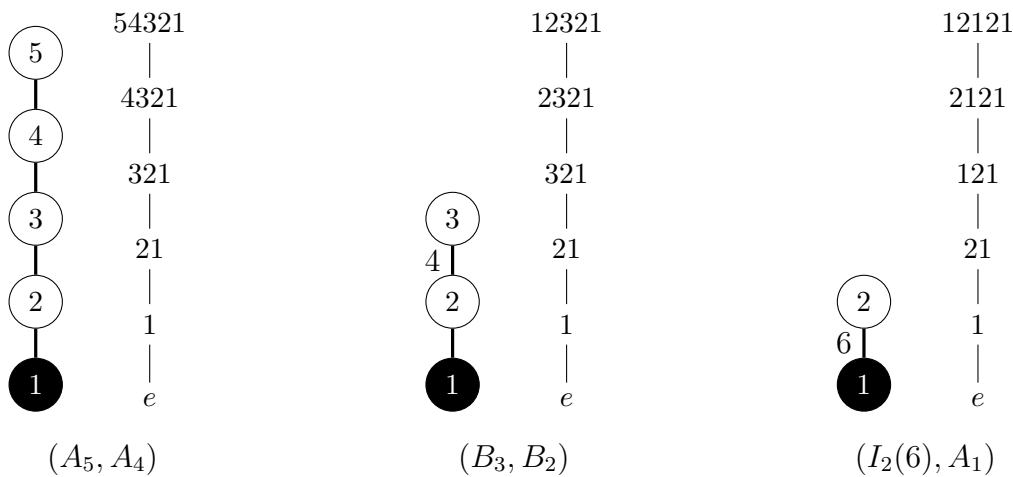

\begin{proof}
$(I_2(m),A_1)$ with $S=\{s,t\}$, $J=\{t\}$ gives $W^J$ consisting of
\[e<s<ts<sts<tsts<\dots\]
with the greatest element having length $m-1$ (or $W^J$ has no greatest element if $m=\infty$). $(A_n,A_{n-1})$ with $S=\{s_1,\dots,s_n\}$, $J=\{s_2,\dots,s_n\}$ gives $W^J$ consisting of
\[e<s_1<s_2s_1<s_3s_2s_1<\dots<s_n\cdots s_1.\]
Finally, $(B_k,B_{k-1})$ with $S=\{s_1,\dots,s_k\}$, $J=\{s_2,\dots,s_k\}$ gives $W^J$ consisting of
\[e<s_1<s_2s_1<s_3s_2s_1<\dots<s_k\cdots s_1<s_{k-1}s_k\cdots s_1<\dots<s_1\cdots s_k\cdots s_1.\]
It is not difficult to check that adding any other generators to these expressions gives elements not in $W^J$. The first three isomorphisms in Theorem~\ref{thm:one} then follow from counting the number of elements.
\end{proof}

\begin{prop}\label{prop:bd}
Suppose $(W,W_J)=(B_n,A_{n-1})$ and $(U,U_K)=(D_{n+1},A_n)$ with $n\geq3$. Then $W^J\cong U^K$ as posets, that is we have the isomorphism (4) in Theorem~\ref{thm:one}.
\end{prop}

\begin{figure}[h]
\begin{minipage}{0.49\textwidth}\centering
\begin{tikzpicture}
\node[smallblacknode] (1) at (0,0) {\scriptsize{$1$}};
\node[smallwhitenode] (2) at (0,1) {\scriptsize{$2$}};
\node[smallwhitenode] (3) at (0,2) {\scriptsize{$3$}};
\node[smallwhitenode] (4) at (0,3) {\scriptsize{$4$}};
\node[smallwhitenode] (5) at (0,4) {\scriptsize{$5$}};
\begin{scope}[every edge/.style=graphedge]
\path (2) edge node[left] {$4$} (1) edge (3);
\path (4) edge (3) edge (5);
\end{scope}
\begin{scope}[shift={(3.2,-0.3)}, xscale=2, yscale=0.6]
\node (e) at (0.0,0) {\tiny{$e$}};
\node (1) at (0.0,1) {\tiny{$1$}};
\node (21) at (0.0,2) {\tiny{$21$}};
\node (321) at (-0.5,3) {\tiny{$321$}};
\node (121) at (0.5,3) {\tiny{$121$}};
\node (1321) at (0.5,4) {\tiny{$1321$}};
\node (4321) at (-0.5,4) {\tiny{$4321$}};
\node (54321) at (-1.0,5) {\tiny{$54321$}};
\node (14321) at (0.0,5) {\tiny{$14321$}};
\node (21321) at (1.0,5) {\tiny{$21321$}};
\node (214321) at (0.0,6) {\tiny{$214321$}};
\node (121321) at (1.0,6) {\tiny{$121321$}};
\node (154321) at (-1.0,6) {\tiny{$154321$}};
\node (2154321) at (-1.0,7) {\tiny{$2154321$}};
\node (1214321) at (1.0,7) {\tiny{$1214321$}};
\node (3214321) at (0.0,7) {\tiny{$3214321$}};
\node (32154321) at (-1.0,8) {\tiny{$32154321$}};
\node (13214321) at (1.0,8) {\tiny{$13214321$}};
\node (12154321) at (0.0,8) {\tiny{$12154321$}};
\node (132154321) at (0.0,9) {\tiny{$132154321$}};
\node (213214321) at (1.0,9) {\tiny{$213214321$}};
\node (432154321) at (-1.0,9) {\tiny{$432154321$}};
\node (1432154321) at (-1.0,10) {\tiny{$1432154321$}};
\node (2132154321) at (0.0,10) {\tiny{$2132154321$}};
\node (1213214321) at (1.0,10) {\tiny{$1213214321$}};
\node (12132154321) at (0.5,11) {\tiny{$12132154321$}};
\node (21432154321) at (-0.5,11) {\tiny{$21432154321$}};
\node (321432154321) at (-0.5,12) {\tiny{$321432154321$}};
\node (121432154321) at (0.5,12) {\tiny{$121432154321$}};
\node (1321432154321) at (0.0,13) {\tiny{$1321432154321$}};
\node (21321432154321) at (0.0,14) {\tiny{$21321432154321$}};
\node (121321432154321) at (0.0,15) {\tiny{$121321432154321$}};
\path (1) edge (21) edge (e);
\path (321) edge (21);
\path (121) edge (21);
\path (1321) edge (321) edge (121);
\path (4321) edge (321);
\path (54321) edge (4321);
\path (14321) edge (4321) edge (1321);
\path (21321) edge (1321);
\path (214321) edge (14321) edge (21321);
\path (121321) edge (21321);
\path (154321) edge (54321) edge (14321);
\path (2154321) edge (154321) edge (214321);
\path (1214321) edge (214321) edge (121321);
\path (3214321) edge (214321);
\path (32154321) edge (2154321) edge (3214321);
\path (13214321) edge (3214321) edge (1214321);
\path (12154321) edge (2154321) edge (1214321);
\path (132154321) edge (32154321) edge (12154321) edge (13214321);
\path (213214321) edge (13214321);
\path (432154321) edge (32154321);
\path (1432154321) edge (432154321) edge (132154321);
\path (2132154321) edge (132154321) edge (213214321);
\path (1213214321) edge (213214321);
\path (12132154321) edge (2132154321) edge (1213214321);
\path (21432154321) edge (1432154321) edge (2132154321);
\path (321432154321) edge (21432154321);
\path (121432154321) edge (21432154321) edge (12132154321);
\path (1321432154321) edge (321432154321) edge (121432154321);
\path (21321432154321) edge (1321432154321);
\path (121321432154321) edge (21321432154321);
\end{scope}
\end{tikzpicture}
\[(B_5,A_4)\]
\end{minipage}\hfill
\begin{minipage}{0.49\textwidth}\centering
\begin{tikzpicture}
\node[smallblacknode] (1) at (0,0) {\scriptsize{$1$}};
\node[smallwhitenode] (2) at (0,1) {\scriptsize{$2$}};
\node[smallwhitenode] (3) at (0,2) {\scriptsize{$3$}};
\node[smallwhitenode] (4) at (0,3) {\scriptsize{$4$}};
\node[smallwhitenode] (5) at (0,4) {\scriptsize{$5$}};
\node[smallwhitenode] (0) at (-1,1) {\scriptsize{$0$}};
\begin{scope}[every edge/.style=graphedge]
\path (2) edge (1) edge (3) edge (0);
\path (4) edge (3) edge (5);
\end{scope}
\begin{scope}[shift={(3.2,-0.3)}, xscale=2, yscale=0.6]
\node (e) at (0.0,0) {\tiny{$e$}};
\node (1) at (0.0,1) {\tiny{$0$}};
\node (21) at (0.0,2) {\tiny{$21$}};
\node (021) at (0.5,3) {\tiny{$021$}};
\node (321) at (-0.5,3) {\tiny{$321$}};
\node (0321) at (0.5,4) {\tiny{$0321$}};
\node (4321) at (-0.5,4) {\tiny{$4321$}};
\node (04321) at (0.0,5) {\tiny{$04321$}};
\node (54321) at (-1.0,5) {\tiny{$54321$}};
\node (20321) at (1.0,5) {\tiny{$20321$}};
\node (204321) at (0.0,6) {\tiny{$204321$}};
\node (120321) at (1.0,6) {\tiny{$120321$}};
\node (054321) at (-1.0,6) {\tiny{$054321$}};
\node (2054321) at (-1.0,7) {\tiny{$2054321$}};
\node (1204321) at (1.0,7) {\tiny{$1204321$}};
\node (3204321) at (0.0,7) {\tiny{$3204321$}};
\node (32054321) at (-1.0,8) {\tiny{$32054321$}};
\node (13204321) at (1.0,8) {\tiny{$13204321$}};
\node (12054321) at (0.0,8) {\tiny{$12054321$}};
\node (132054321) at (0.0,9) {\tiny{$132054321$}};
\node (213204321) at (1.0,9) {\tiny{$213204321$}};
\node (432054321) at (-1.0,9) {\tiny{$432054321$}};
\node (1432054321) at (-1.0,10) {\tiny{$1432054321$}};
\node (0213204321) at (1.0,10) {\tiny{$0213204321$}};
\node (2132054321) at (0.0,10) {\tiny{$2132054321$}};
\node (02132054321) at (0.5,11) {\tiny{$02132054321$}};
\node (21432054321) at (-0.5,11) {\tiny{$21432054321$}};
\node (021432054321) at (0.5,12) {\tiny{$021432054321$}};
\node (321432054321) at (-0.5,12) {\tiny{$321432054321$}};
\node (0321432054321) at (0.0,13) {\tiny{$0321432054321$}};
\node (20321432054321) at (0.0,14) {\tiny{$20321432054321$}};
\node (120321432054321) at (0.0,15) {\tiny{$120321432054321$}};
\path (1) edge (21) edge (e);
\path (021) edge (21);
\path (321) edge (21);
\path (0321) edge (321) edge (021);
\path (4321) edge (321);
\path (04321) edge (4321) edge (0321);
\path (54321) edge (4321);
\path (20321) edge (0321);
\path (204321) edge (04321) edge (20321);
\path (120321) edge (20321);
\path (054321) edge (54321) edge (04321);
\path (2054321) edge (054321) edge (204321);
\path (1204321) edge (204321) edge (120321);
\path (3204321) edge (204321);
\path (32054321) edge (2054321) edge (3204321);
\path (13204321) edge (3204321) edge (1204321);
\path (12054321) edge (2054321) edge (1204321);
\path (132054321) edge (32054321) edge (12054321) edge (13204321);
\path (213204321) edge (13204321);
\path (432054321) edge (32054321);
\path (1432054321) edge (432054321) edge (132054321);
\path (0213204321) edge (213204321);
\path (2132054321) edge (132054321) edge (213204321);
\path (02132054321) edge (2132054321) edge (0213204321);
\path (21432054321) edge (1432054321) edge (2132054321);
\path (021432054321) edge (21432054321) edge (02132054321);
\path (321432054321) edge (21432054321);
\path (0321432054321) edge (321432054321) edge (021432054321);
\path (20321432054321) edge (0321432054321);
\path (120321432054321) edge (20321432054321);
\end{scope}
\end{tikzpicture}
\[(D_6,A_5)\]
\end{minipage}
\caption{Example of the isomorphism $B_n/A_{n-1}\cong D_{n+1}/A_n$ with $n=5$.}
\end{figure}
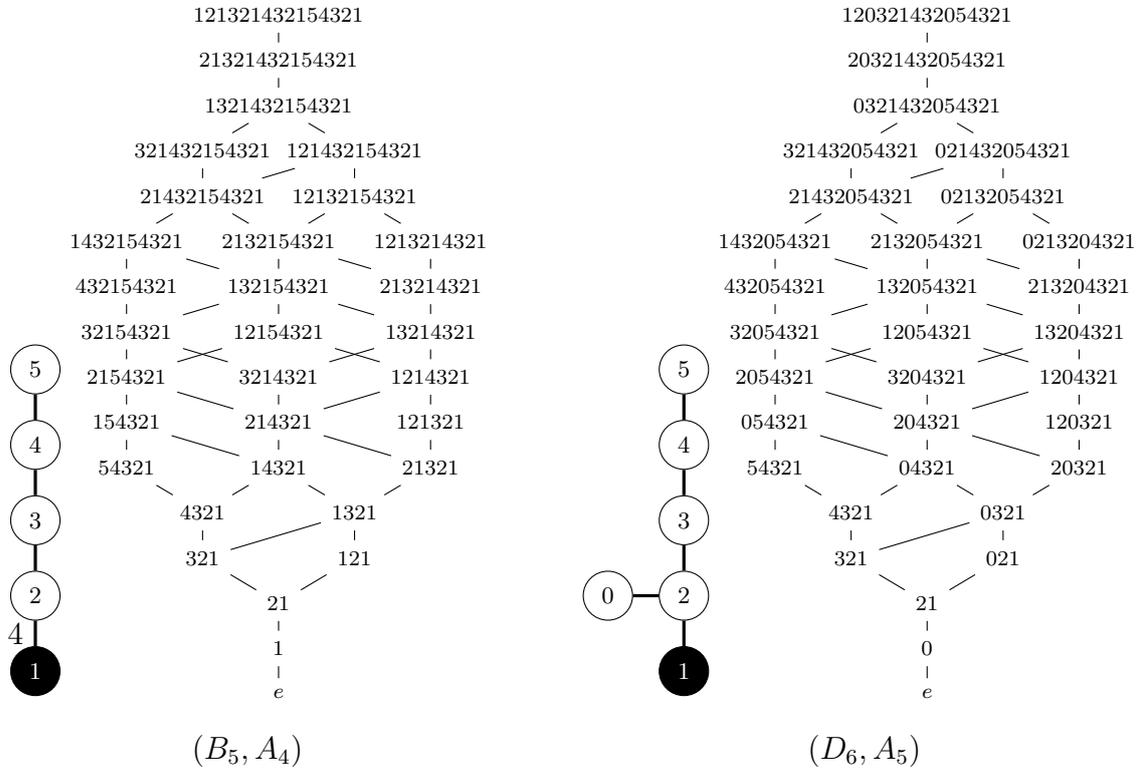

\begin{proof}
Denote the generators of $(W,S)$ and $(U,T)$ as in the diagram below.
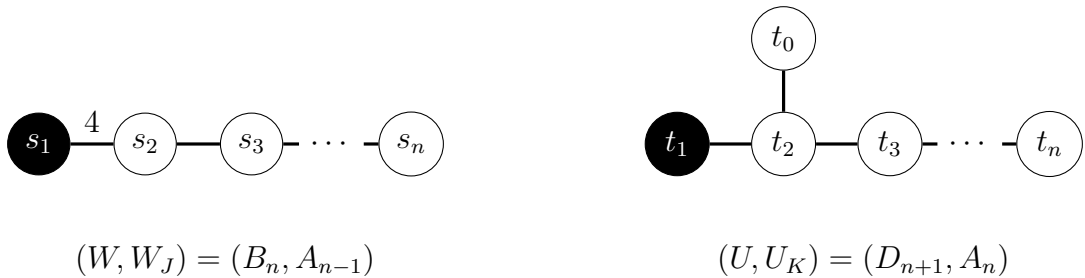
\begin{center}
\begin{tikzpicture}[scale=0.7]
\node[blacknode] (1) at (0,0) {$s_1$};
\node[whitenode] (2) at (2,0) {$s_2$};
\node[whitenode] (3) at (4,0) {$s_3$};
\node (d) at (5.5,0) {$\cdots$};
\node[whitenode] (n) at (7,0) {$s_n$};
\begin{scope}[every edge/.style=graphedge]
	\path (2) edge node[above] {$4$} (1) edge (3);
	\path (d) edge (3) edge (n);
\end{scope}
\node at (3.5,-2.2) {$(W,W_J)=(B_n,A_{n-1})$};
\begin{scope}[shift={(12,0)}]
\node[blacknode] (1) at (0,0) {$t_1$};
\node[whitenode] (2) at (2,0) {$t_2$};
\node[whitenode] (3) at (4,0) {$t_3$};
\node (d) at (5.5,0) {$\cdots$};
\node[whitenode] (n) at (7,0) {$t_n$};
\node[whitenode] (0) at (2,2) {$t_0$};
\begin{scope}[every edge/.style=graphedge]
	\path (2) edge (1) edge (0) edge (3);
	\path (d) edge (3) edge (n);
\end{scope}
\node at (3.5,-2.2) {$(U,U_K)=(D_{n+1},A_n)$};
\end{scope}
\end{tikzpicture}
\end{center}
We can in fact explicitly describe elements of $B_n/A_{n-1}$ as follows. Fix $w\in W^J$; any reduced expression $\ul w$ is of the form $\ul w_1s_{k_1}s_{k_1-1}\dots s_1$ for some product of generators $\ul w_1$ and $1\leq k_1\leq n$. Choose $\ul w$ so that $k_1$ is maximal, and for $1\leq i\leq n$ let $s^{(i)}$ be the furthest-right $s_i$ term in $\ul w_1$ if it exists. Fix $i>1$ and assume for a contradiction that $s^{(i)}$ commutes with every generator to the right of it in $\ul w_1$.
\begin{itemize}
\item If $i<k_1$, then shifting $s^{(i)}$ into $s_{k_1}\dots s_1$ with braid moves shows $w\not\in W^J$, since
\[s_is_{i+1}s_i\dots s_1=s_{i+1}s_is_{i+1}s_{i-1}\dots s_1=s_{i+1}\dots s_1s_{i+1}.\]
\item If $i=k_1$, then we can perform a nil-move on $s^{(i)}$ and $s_{k_1}$ contradicting $\ul w$ reduced.
\item If $i=k_1+1$, then we have a reduced expression ending in $s_{k_1+1}\dots s_1$ contradicting maximality of $k_1$.
\item If $i>k_1+1$, then we can shift $s^{(i)}$ to the right end of $\ul w$ contradicting $w\in W^J$.
\end{itemize}
Thus there must be some `blocking' generator to the right of $s^{(i)}$ in $\ul w_1$. Assume this generator is $s_{i+1}$, so $s^{(i+1)}$ is to the right of $s^{(i)}$. By the same reasoning there must be a blocking generator to the right of $s^{(i+1)}$, which must be $s_{i+2}$ since $s^{(i)}$ is left of $s^{(i+1)}$. Repeating this we reach a contradiction at $s^{(n)}$, thus our assumption was false and the blocking generator for $s^{(i)}$ is $s_{i-1}$. Repeating on $s^{(i-1)},s^{(i-2)},\dots$ we see that $\ul w_1$ is either trivial or of the form $\ul w_2s_{k_2}\dots s_1$. Again choose $\ul w$ so that $k_2$ is maximal, and it is not hard to check that $k_2<k_1$ (otherwise move the generators $s_{k_2}\dots s_1$ as far right as possible and apply a series of braid moves). We can then repeat the entire process above, showing that $w=(s_{k_m}\dots s_1)(s_{k_{m-1}}\dots s_1)\dots(s_{k_1}\dots s_1)$ for some integers $1\leq k_1<\dots<k_m\leq n$.

The same process works for $D_{n+1}/A_n$ as well, except that when we consider $t^{(2)}$ we can choose either $t_1$ or $t_0$ to be the blocking generator. We must choose whichever we didn't take for the previous sequence $t_{k_j}\dots t_2t_{(1\text{ or }0)}$, or else we can perform a braid move on $t_1t_2t_1$ or $t_0t_2t_0$ and obtain a contradiction. Thus we see that there is a bijection $W^J\to U^K$ in which given $w=(s_{k_m}\dots s_1)(s_{k_{m-1}}\dots s_1)\dots(s_{k_1}\dots s_1)\in W^J$ we replace each $s_i$ with $t_i$ for $i>1$, and replace each $s_1$ with either $t_1$ or $t_0$ alternating from right to left. By the Subword Property, this is an isomorphism of Bruhat order.
\end{proof}

\begin{prop}
$H_3/H_2$ and $D_6/D_5$ are isomorphic as posets, that is we have the isomorphism (5) in Theorem~\ref{thm:one}.
\end{prop}

\begin{proof}
Direct computation gives the posets shown in Figure~\ref{fig:h3}.
\end{proof}

\begin{figure}[h]
\begin{center}
\begin{tikzpicture}[scale=0.65]
\begin{scope}
    \node[whitenode] (3) at (0,2) {$3$};
    \node[whitenode] (2) at (0,0) {$2$};
    \node[blacknode] (1) at (0,-2) {$1$};
\end{scope}
\begin{scope}[every edge/.style=graphedge]
    \path (3) edge node[left] {$5$} (2);
    \path (2) edge (1);
\end{scope}
\end{tikzpicture}\hspace{50pt}
\begin{tikzpicture}[scale=0.6]
\begin{scope}
	\node (e) at (0,-3) {\small{$e$}};
	\node (1) at (0,-2) {\small{$1$}};
	\node (21) at (0,-1) {\small{$21$}};
	\node (321) at (0,0) {\small{$321$}};
	\node (2321) at (0,1) {\small{$2321$}};
	\node (12321) at (-1,2) {\small{$12321$}};
	\node (32321) at (1,2) {\small{$32321$}};
	\node (132321) at (0,3) {\small{$132321$}};
	\node (2132321) at (0,4) {\small{$2132321$}};
	\node (32132321) at (0,5) {\small{$32132321$}};
	\node (232132321) at (0,6) {\small{$232132321$}};
	\node (1232132321) at (0,7) {\small{$1232132321$}};
\end{scope}
\begin{scope}[every edge/.style={draw=black}]
	\path (e) edge (1);
	\path (1) edge (21);
	\path (21) edge (321);
	\path (321) edge (2321);
	\path (2321) edge (12321) edge (32321);
	\path (132321) edge (12321) edge (32321);
	\path (132321) edge (2132321);
	\path (2132321) edge (32132321);
	\path (32132321) edge (232132321);
	\path (232132321) edge (1232132321);
\end{scope}
\end{tikzpicture}\hspace{80pt}
\begin{tikzpicture}[scale=0.65]
\begin{scope}
    \node[whitenode] (5) at (0,6) {$5$};
    \node[whitenode] (4) at (0,4) {$4$};
    \node[whitenode] (3) at (0,2) {$3$};
    \node[whitenode] (2) at (0,0) {$2$};
    \node[blacknode] (1) at (0,-2) {$1$};
    \node[whitenode] (6) at (2,4) {$6$};
\end{scope}
\begin{scope}[every edge/.style=graphedge]
    \path (5) edge (4);
    \path (4) edge (3) edge (6);
    \path (3) edge (2);
    \path (2) edge (1);
\end{scope}
\end{tikzpicture}\hspace{30pt}
\begin{tikzpicture}[scale=0.6]
\begin{scope}
	\node (e) at (0,-3) {\small{$e$}};
	\node (1) at (0,-2) {\small{$1$}};
	\node (21) at (0,-1) {\small{$21$}};
	\node (321) at (0,0) {\small{$321$}};
	\node (4321) at (0,1) {\small{$4321$}};
	\node (54321) at (-1,2) {\small{$54321$}};
	\node (64321) at (1,2) {\small{$64321$}};
	\node (564321) at (0,3) {\small{$564321$}};
	\node (4564321) at (0,4) {\small{$4564321$}};
	\node (34564321) at (0,5) {\small{$34564321$}};
	\node (234564321) at (0,6) {\small{$234564321$}};
	\node (1234564321) at (0,7) {\small{$1234564321$}};
\end{scope}
\begin{scope}[every edge/.style={draw=black}]
	\path (e) edge (1);
	\path (1) edge (21);
	\path (21) edge (321);
	\path (321) edge (4321);
	\path (4321) edge (54321) edge (64321);
	\path (564321) edge (54321) edge (64321);
	\path (564321) edge (4564321);
	\path (4564321) edge (34564321);
	\path (34564321) edge (234564321);
	\path (234564321) edge (1234564321);
\end{scope}
\end{tikzpicture}
\end{center}
\caption{Coxeter graph and poset for $(H_3,H_2)$ (left) and $(D_6,D_5)$ (right).}\label{fig:h3}
\end{figure}
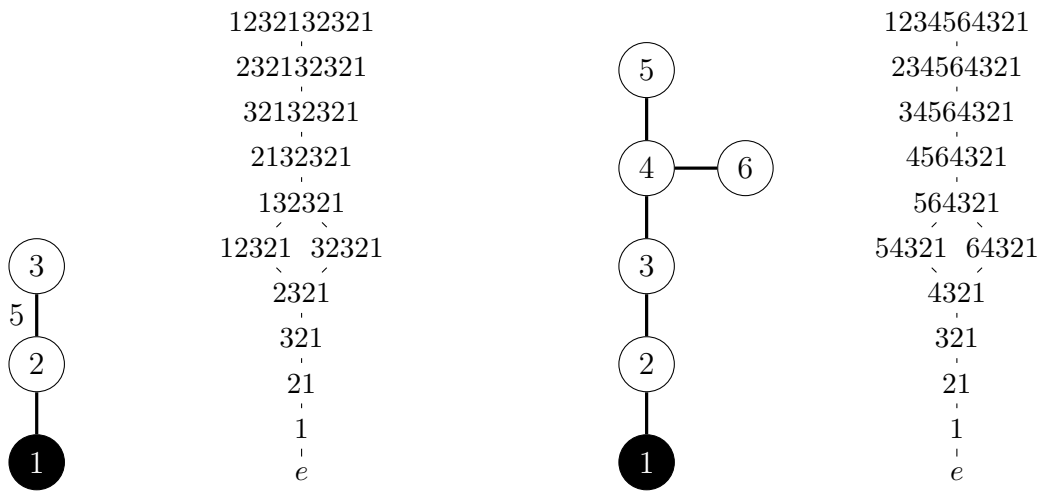

\section{Machinery for the classification}\label{sec:chainlike}

The goal is now to reconstruct the Coxeter graph given only the Bruhat order $\leq$ on $W^J$. If we can achieve this then two Coxeter pairs can only produce isomorphic posets if they are isomorphic as Coxeter pairs (barring the exceptional cases (1)-(6) in Theorem~\ref{thm:one}). There are certain features of the quotient $W^J$ which we can derive entirely using the structure of the poset, the length function $\ell$ being a simple example:

\begin{con}
Denote the unique least element of $W^J$ by $e$ and denote $\ell(e)=0$. Define the length $\ell(w)$ of each element inductively by setting $\ell(w)=n$ if all elements less than $w$ have length less than $n$.
\end{con}

In this section we define and derive the properties of several tools which will help in separating non-isomorphic posets.

\subsection{Chainlike elements}

The main tool we'll use is what we will call \textit{chainlike elements}, which form a subset that closely resembles the structure of the Coxeter graph.

\begin{definition}
Call an element $w\in W^J$ \textit{semi-chainlike} if it covers exactly one element in $W^J$. We denote this element $w'\vartriangleleft w$. Call $w$ \textit{chainlike} if $w$ is semi-chainlike and $w'$ is either chainlike or $e$, inductively. In other words, $w$ is chainlike if and only if there is a unique chain of coverings $e\vartriangleleft\dots\vartriangleleft w$. We denote the set of chainlikes $C(W^J)$.
\end{definition}

\begin{figure}[h]\centering
\begin{center}
\begin{tikzpicture}[scale=0.9]
\begin{scope}
    \node[blacknode] (0) at (0,0.2) {$0$};
    \node[whitenode] (1) at (-1,1) {$1$};
    \node[whitenode] (2) at (-1,2) {$2$};
    \node[whitenode] (3) at (-1,3) {$3$};
    \node[whitenode] (4) at (1,1) {$4$};
    \node[whitenode] (5) at (1,2) {$5$};
    \node[whitenode] (6) at (0.3,3) {$6$};
    \node[whitenode] (7) at (1.7,3) {$7$};
\end{scope}
\begin{scope}[every edge/.style=graphedge]
    \path (0) edge (1) edge (4);
    \path (1) edge (2);
    \path (2) edge (3);
    \path (4) edge (5);
    \path (5) edge (6) edge (7);
\end{scope}
\end{tikzpicture}\hspace{2em}
\begin{tikzpicture}[scale=0.9]
\begin{scope}
    \node (0) at (0,0.2) {$0$};
    \node (1) at (-1,1) {$10$};
    \node (2) at (-1,2) {$210$};
    \node (3) at (-1,3) {$3210$};
    \node (4) at (1,1) {$40$};
    \node (5) at (1,2) {$540$};
    \node (6) at (0.3,3) {$6540$};
    \node (7) at (1.7,3) {$7540$};
\end{scope}
\begin{scope}[every edge/.style={draw=black}]
    \path (0) edge (1) edge (4);
    \path (1) edge (2);
    \path (2) edge (3);
    \path (4) edge (5);
    \path (5) edge (6) edge (7);
\end{scope}
\end{tikzpicture}
\end{center}
\caption{A simply-laced Coxeter quotient (left) and all of its chainlike elements (right).}\label{fig:tree}
\end{figure}
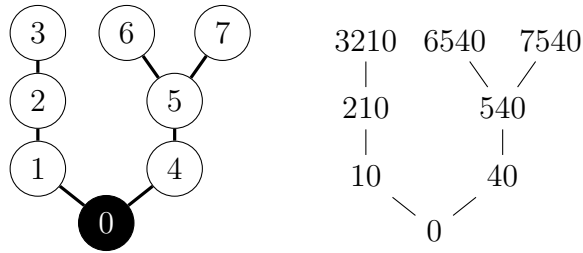

\begin{remark}
The Bruhat graph of $C(W^J)\cup\{e\}$ naturally has a tree structure. Figure~\ref{fig:tree} demonstrates why we are focusing on chainlike elements; in certain special cases (when $|J|=|S|-1$ and the Coxeter graph has no labelled edges or cycles) the graph of $C(W^J)$ has exactly the same shape as the Coxeter graph. This will be proven via the study of simple elements in the next section.
\end{remark}

\begin{prop}\label{prop:unique}
Suppose $w\in W^J$ is semi-chainlike.
\begin{enumerate}[label=(\alph*)]
\item We have $w=sw'$ for some $s\in S$, and every reduced expression for $w$ begins with $s$ (or equivalently, the leftmost generator $s$ cannot be moved with braid-moves).
\item If $w$ is chainlike then it has a unique reduced expression.
\end{enumerate}
\end{prop}

\begin{proof}
(a) Any reduced expression for $w$ must contain a reduced expression for $w'$, which we denote $w'=s_1\dots s_n$. Then $w=s_1\dots s_iss_{i+1}\dots s_n$ for some $0\leq i\leq n$ and $s\in S$. If $i=0$ then $w=sw'$. Otherwise, let $u=s_2\dots s_iss_{i+1}\dots s_n$. We have $u\in W^J$ and $u$ is reduced, since otherwise there would be a sequence of braid- and nil-moves resulting in $u$ being shorter or ending in an element of $J$, and the same moves could be performed for $w=s_1u$ contradicting $w\in W^J$ reduced. Thus $u\vartriangleleft w$, so for $w$ to be semi-chainlike we require $u=w'$ and so $w=s_1u=s_1w'<w'$, a contradiction. Thus the only possibility is $w=sw'$, and every reduced expression for $w$ must be of this form.

(b) All $w\in W^J$ with $\ell(w)=1$ are chainlike and have a unique reduced expression. The result then follows from (a) and induction on $\ell(w)$.
\end{proof}

\subsection{Simple elements}

First we characterise chainlike elements in simply-laced Coxeter groups and outline a construction which takes the Bruhat poset and determines the Coxeter graph. This will be the foundation for the construction for general Coxeter groups, so note that it is \textit{not} assumed that $W$ is simply-laced throughout this section.

\begin{definition}
We call an element $w\in W^J$ \textit{simple} if it has an expression of the form $w=s_ks_{k-1}\dots s_1s_0$ with $k\geq0$ satisfying the following conditions:
\begin{itemize}
	\item $s_0,\dots,s_k\in S$ are all distinct;
	\item $s_0\in S\setminus J$ and $s_1,\dots,s_k\in J$;
	\item $m(s_i,s_j)\geq3$ if and only if $|i-j|=1$ for all $i,j\in\{0,\dots,k\}$.
\end{itemize}
By the Word Property this is the unique reduced expression for $w$.
\end{definition}

\begin{theorem}\label{thm:simple}
We have:
\begin{enumerate}[label=(\alph*)]
	\item If $w\in W^J$ is simple, then $w$ is chainlike.
	\item If $W$ is simply-laced then $w\in W^J$ is simple if and only if $w$ is chainlike.
\end{enumerate}
\end{theorem}

\begin{proof}
All elements in $S\setminus J$ are chainlike elements in $W^J$. For $\ell(w)=2$ we see that chainlike elements must be of the form $s_1s_0$ with $s_0\in S\setminus J$, $s_1\in J$ and $s_1s_0\neq s_0s_1$. We cannot have $s_1\not\in J$ since we would have $s_0,s_1\vartriangleleft s_1s_0$ and $s_1s_0$ not semi-chainlike. With these base cases complete we now continue by induction on $\ell(w)$.

For (a), let's suppose $w=s_k\dots s_0$ is simple and show that $w$ is chainlike. Let $u\vartriangleleft w$ be obtained by removing a generator $s_i$. If $s_i\neq s_k$ then $s_{i+1}$ commutes with every generator to the right of it in $u$, hence a reduced word for $u$ is $s_k\dots s_{i+2}s_{i-1}\dots s_0s_{i+1}$ which ends in a generator in $J$ and so $u\not\in W^J$. If instead $s_i=s_k$ then $u=s_{k-1}\dots s_0$ is chainlike by induction, so $w$ covers only one element $w'=u$ and $w$ is chainlike.

Now let $W$ be simply-laced. For (b) suppose $w$ is chainlike so that $w'=s_k\dots s_0\in W^J$ is simple by induction with $k\geq1$. Proposition~\ref{prop:unique} gives $w=sw'$ with $m(s,s_k)\geq3$.
\begin{itemize}
\item Suppose $s<w'$, that is $s=s_i$ for some $i$. Since $m(s,s_k)\geq3$ we must have $s=s_{k-1}$, so $w=s_{k-1}w'=s_{k-1}s_ks_{k-1}\dots s_0$. Since $W$ is simply-laced we have the braid-move $s_{k-1}s_ks_{k-1}\to s_ks_{k-1}s_k$, violating Proposition~\ref{prop:unique}.
\item Suppose $s\not<w'$, and either $s\not\in J$ or $m(s_i,s)\geq3$ for some $i\in\{0,\dots,k-1\}$. Then removing $s_k$ from $w$ gives a reduced expression for $sw''$, which is in $W^J$ as either $s\not\in J$ or $s$ cannot be moved past $s_i$ with braid-moves. But then $sw''\vartriangleleft w$ with $w'\neq sw''$, so $w$ is not semi-chainlike. 
\item The only case left is $s\not<w'$, $s\in J$, and $m(s_i,s)\geq3\iff s_i=s_k$ for all $s_i<w'$. Then $w$ is simple.
\end{itemize}
This covers all cases, so we have shown inductively that any chainlike $w$ is simple.
\end{proof}

Intuitively, each simple element with length greater than 1 corresponds to a path starting at a black node and travelling through a sequence of white nodes such that no selection of nodes in the path forms a cycle. This means that when there are cycles or multiple black nodes there may be multiple paths to each white node. To account for this we attempt to collect simple elements into equivalence classes based on their leftmost generator, as exemplified in Figure~\ref{fig:equivalence}. The next definition achieves this.

\begin{definition}
For $u,v\in W^J$, we denote
\[L(u,v)=\min\{\ell(w)\mid w\in W^J,\,u\leq w,\,v\leq w\}\]
which is well defined since $W^J$ is a directed poset (see Proposition 2.2.9 in \cite{BB}). We then define the relation $\sim$ on $C(W^J)$ by
\[u\sim v\ \iff\ L(u,v)=L(u',v)=L(u,v').\]
\end{definition}

\begin{figure}[h]\centering
\begin{center}
\begin{tikzpicture}[scale=0.8]
\begin{scope}
	\node[blacknode] (0) at (0,0) {$0$};
	\node[blacknode] (1) at (-1.73,1) {$1$};
	\node[whitenode] (2) at (-1.73,-1) {$2$};
	\node[whitenode] (3) at (-3.73,-1) {$3$};
	\node[whitenode] (4) at (-3.73,1) {$4$};
	\node[whitenode] (5) at (-5.46,0) {$5$};
	\node[whitenode] (6) at (2,0) {$6$};
	\node[whitenode] (7) at (3.73,-1) {$7$};
	\node[whitenode] (8) at (3.73,1) {$8$};
	\node[whitenode] (9) at (5.46,0) {$9$};
\end{scope}
\begin{scope}[every edge/.style=graphedge]
	\path (0) edge (1) edge (2) edge (6);
	\path (1) edge (2) edge (4);
	\path (3) edge (2) edge (4);
	\path (5) edge (3) edge (4);
	\path (7) edge (6) edge (8) edge (9);
	\path (8) edge (6) edge (9);
\end{scope}
\end{tikzpicture}\vspace{2em}

\begin{tikzpicture}[scale=1.1]
\def\r{0.8}
\begin{scope}
	\draw (0,0) circle (\r);
	\draw (-1.73,1) circle (\r);
	\draw (-1.73,-1) circle (\r);
	\draw (-3.73,-1) circle (\r);
	\draw (-3.73,1) circle (\r);
	\draw (-5.46,0) circle (\r);
	\draw (2,0) circle (\r);
	\draw (3.73,-1) circle (\r);
	\draw (3.73,1) circle (\r);
	\draw (5.46,0) circle (\r);
\end{scope}
\begin{scope}
	\node (0) at (0,0) {$0$};
	\node (1) at (-1.73,1) {$1$};
	\node (20) at (-1.73+\r*0.35,-1+\r*0.35) {$20$};
	\node (21) at (-1.73-\r*0.35,-1-\r*0.35) {$21$};
	\node (321) at (-3.73+\r*0.1,-1-\r*0.5) {$321$};
	\node (320) at (-3.73-\r*0.3,-1+\r*0.1) {$320$};
	\node (341) at (-3.73+\r*0.25,-1+\r*0.5) {$341$};
	\node (41) at (-3.73+\r*0.4,1+\r*0.4) {$41$};
	\node (4320) at (-3.73-\r*0.2,1-\r*0.4) {$4320$};
	\node (541) at (-5.46-\r*0.2,\r*0.5) {$541$};
	\node (5320) at (-5.46+\r*0.4,0) {$5320$};
	\node (5321) at (-5.46-\r*0.2,-\r*0.5) {$5321$};
	\node (60) at (2,0) {$60$};
	\node (760) at (3.73,-1) {$760$};
	\node (860) at (3.73,1) {$860$};
	\node (9760) at (5.46,-\r*0.4) {$9760$};
	\node (9860) at (5.46,\r*0.4) {$9860$};
\end{scope}
\begin{scope}[every edge/.style={draw=black}]
	\path (0) edge (20) edge (60);
	\path (1) edge (21);
	\path (321) edge (21) edge (5321);
	\path (320) edge (20) edge (4320) edge (5320);
	\path (41) edge (1) edge (341) edge (541);
	\path (860) edge (60) edge (9860);
	\path (760) edge (60) edge (9760);
\end{scope}
\end{tikzpicture}
\end{center}
\caption{A simply-laced Coxeter pair (top), and all of its chainlike elements collected into equivalence classes by their leftmost generator (bottom). Edges between chainlike elements denote covering relations. Note that the edges 1-0 and 7-8 are not represented in the covering relations; these cases are considered in step 2 (b) of Theorem~\ref{thm:construction}.}\label{fig:equivalence}
\end{figure}
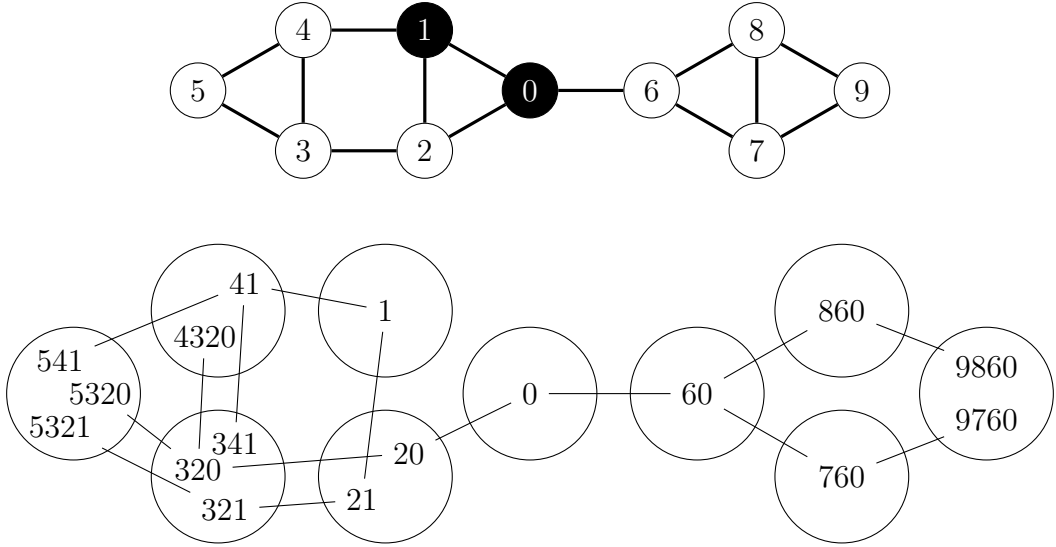

\begin{prop}
Let $u=s_1\dots s_n$ and $v=t_1\dots t_k$ be the unique reduced expressions for $u,v\in C(W^J)$. Suppose $\ul r=r_1\dots r_m$ is a (possibly degenerate) maximal-length subword of both $u$ and $v$, that is $m\geq0$ is maximal with $r_1,\dots,r_m\in S$ satisfying
\begin{alignat*}{9}
	r_1&={}&s_{i_1},\ r_2&={}&s_{i_2},\ &\dots{}&,\ r_m&={}&s_{i_m}\text{ with }1&\leq{}&i_1&<{}&i_2<&\dots{}&<i_m&\leq{}&n;\\
	r_1&={}&t_{j_1},\ r_2&={}&t_{j_2},\ &\dots{}&,\ r_m&={}&t_{j_m}\text{ with }1&\leq{}&j_1&<{}&j_2<&\dots{}&<j_m&\leq{}&k.
\end{alignat*}
Then we have $L(u,v)=\ell(u)+\ell(v)-\ell(\ul r)=\ell(u)+\ell(v)-m$.
\end{prop}

\begin{proof}
Denote $i_0=j_0=0$, $i_{m+1}=n+1$, $j_{m+1}=k+1$. Let $S^*(u,v)$ be the set of words of the form $\ul x_0r_1\ul x_1r_2\dots r_k\ul x_k$ where $\ul r=r_1\dots r_m=s_{i_1}\dots s_{i_n}=t_{j_1}\dots t_{j_n}$ is any maximal-length subword of both $u$ and $v$, and each $\ul x_p$ is a permutation of the elements $s_i$ for $i_p<i<i_{p+1}$ and $t_j$ for $j_p<j<j_{p+1}$ such that the order of the $s_i$ terms is preserved and the order of the $t_j$ terms is also preserved. Since $m$ is maximal, these are all the words of minimal length containing both $u$ and $v$ as subwords, and their length is $\ell(\ul w)=\ell(u)+\ell(v)-m$. These words also contain no pairs $s_it_j$ or $t_js_i$ with $s_i=t_j$, and every $\ul w\in S^*(u,v)$ ends with either $s_n$ or $t_k$ which are both in $S\setminus J$. To show these are reduced expressions of elements of $W^J$ it then suffices show $S^*(u,v)$ is closed under the braid-moves of $(W,S)$.

We can see that any braid-moves we can apply to $\ul w\in S^*(u,v)$ cannot include an $s$ and $t$ from the same $\ul x_p$ term; if they did we would either have a nil-move in $u$ or $v$, or $s_i=t_j$ for some $i,j$ in the same $\ul x_p$ term contradicting maximality of $m$. Suppose now that we can apply a braid-move to a sequence containing the consecutive terms $t_{j_p-2}t_{j_p-1}r_p s_{i_p+1}$. This means $r_ps_{i_p+1}=t_{j_p-2}t_{j_p-1}$, so $r_1\dots r_p s_{i_p+1}r_{p+1}\dots r_m$ is a subword of both $u$ and $v$, contradicting maximality of $m$. The same holds if there are 2 terms to the right of $r_p$ or if $s$ and $t$ are swapped in the expression above.

Now suppose we have $t_{j_p-1}=s_{i_p+1}$ and can apply a length 3 braid-move to $t_{j_p-1}r_p s_{i_p+1}$ to obtain $\ul w^\dagger$. Then set $\ul r^\dagger$ to be equal to $\ul r$ but with $r_p$ replaced with $t_{j_p-1}=s_{i_p+1}$, and we see $\ul r^\dagger$ is also a maximal-length subword of $u,v$ and $\ul w^\dagger\in S^*(u,v)$. The same holds for $s_{i_p-1}r_p t_{j_p+1}$, and so $S^*(u,v)$ is closed under braid-moves.
\end{proof}

\begin{prop}\label{prop:link}
Suppose $u,v\in C(W^J)$ and $s,t\in S$ with $u=su'$ and $v=tv'$.
\begin{enumerate}[label=(\alph*)]
	\item $u\sim v$ implies $s=t$.
	\item If $u,v$ are simple, then $u\sim v$ if and only if $s=t$.
\end{enumerate}
\end{prop}

\begin{proof}
(a) Construct a maximal length subword $\ul r=r_1\dots r_m$ of $u$ and $v$, and suppose $L(u,v)=L(u',v)=L(u,v')$. Assume $r_1\neq s$, so that $\ul r$ is a subword of $u'$ also. But then
\[L(u',v)\leq \ell(u')+\ell(v)-m=\ell(u)-1+\ell(v)-m=L(u,v)-1\]
and we have a contradiction. Thus $r_1=s$, and similarly $r_1=t$, so $t=s$.

(b) Suppose $s=t$ and construct a maximal length subword $\ul r=r_1\dots r_m$ of $u$ and $v$. If $r_1\neq s$ then $s\ul r$ is a longer subword of $u$ and $v$, a contradiction, so $r_1=s=t$. But then $r_2\dots r_m$ is a maximal length subword of $u'$ and $v$; if there were a longer subword $\ul r^\dagger$, then $s\ul r^\dagger$ would be longer than $\ul r$ and a subword of $u$ and $v$, a contradiction. Thus
\[L(u',v)=\ell(u')+\ell(v)-(m-1)=\ell(u)+\ell(v)-m=L(u,v)\]
By the same reasoning we also have $L(u,v')=L(u,v)$, so $u\sim v$.
\end{proof}

Thus $\sim$ is an equivalence relation on simple elements, which comprise $C(W^J)$ if $W$ is simply-laced. We can then reconstruct the Coxeter graph by assigning a node to each equivalence class and drawing edges wherever there are coverings. First, we make a general statement that applies to non-simply-laced groups as well:

\begin{lemma}\label{lem:leftmost}
If $s_n,\dots,s_0$ satisfies $s_n,\dots,s_1\in J$, $s_0\in S\setminus J$ and $s_i,s_{i+1}$ are connected by an edge for $0\leq i<n$, then there is a simple chainlike in $C(W^J)$ starting with $s_n$ and ending with $s_0$. In particular, if $s\in S$ is in a connected component of the graph of $(W,W_J)$ that contains at least one element of $S\setminus J$, then $s$ is the leftmost generator in the unique reduced expression of some simple chainlike in $C(W^J)$.
\end{lemma}

\begin{proof}
Take $s_n,\dots,s_0\in S,n\geq0$ to be of minimal length among all sequences starting with $s_n$ and ending with $s_0$ satisfying the above conditions. If $m(s_i,s_j)=2$ for some $i<j$ with $|i-j|\geq2$, then $s_n,\dots,s_j,s_i,\dots,s_0$ is a shorter sequence, a contradiction. Hence $s_n\dots s_0$ is simple and thus chainlike. Then if $s$ is in a connected component of the graph containing a black node, take $s_n=s$ and $s_0$ a closest black node to $s$.
\end{proof}

\begin{theorem}\label{thm:construction}
If $W$ is simply-laced and irreducible and $J\neq S$, then the following construction on $W^J$ produces the bw-Coxeter graph of the pair $(W,W_J)$:
\begin{enumerate}
	\item Draw a node for every equivalence class of $\sim$, colouring any nodes corresponding to length 1 elements black.
	\item For each pair of distinct equivalence classes $U,V\subseteq C(W^J)$, connect the corresponding nodes with an edge if there are $u\in U,v\in V$ with either:
	\begin{itemize}
		\item[(a)] $u=v'$ or $v=u'$, or
		\item[(b)] $u'=v'$ and there are exactly 2 elements in $W^J$ that cover both $u$ and $v$.
	\end{itemize}
\end{enumerate}
\end{theorem}

\begin{proof}
We need to check that if $s,t\in S$ have an edge in the bw-Coxeter graph then an edge is drawn by the construction. If $s,t\in S\setminus J$, then an edge is drawn by step 2 (b) since $s'=t'=e$. If $s\in J$ and $t\in S\setminus J$ have an edge, then $t\vartriangleleft st$ and step 2 (a) draws an edge. Now suppose $s,t\in J$, and let $ss_n\dots s_0=sw$ be a simple chainlike as given by Lemma~\ref{lem:leftmost}. If $t\leq w$, we must have $t=s_n$ since $m(s,t)\geq3$. Then step 2 (a) draws an edge and we are done, so assume $t\not\leq w$:
\begin{itemize}
	\item If $m(t,s_i)=3$ for some $i\leq n-2$, then assuming $i$ minimal we have $sts_i\dots s_0$ chainlike and we are done.
	\item If $m(t,s_i)=2$ for $i\leq n-2$ and $m(t,s_{n-1})=3$, then $stw'$ is chainlike instead.
	\item If $m(t,s_i)=2$ for $i\leq n-1$ and $m(t,s_n)=3$ (so we have a triangle between $s$, $t$ and $s_n$ in the graph), then $sw$ and $tw$ satisfy the requirements of step 2 (b) in the construction, since the 2 elements of length $\ell(w)+1$ are $stw\neq tsw$.
	\item If $m(t,s_i)=2$ for all $i$, then $tsw$ is chainlike and we are done.
\end{itemize}
This covers all cases. Finally, we check that no erroneous edges are added in step 2 (b). Suppose $u'=v'$, $u=su'$, $v=tu'$, and $s,t$ are not connected by an edge. Then $m(s,t)=2$, so there is only 1 element of length $\ell(u)+1$ greater than both $u,v$, which is $stu'=tsu'$. Thus step 2 (b) correctly does not add an edge between $s$ and $t$.
\end{proof}

\subsection{Classifying all chainlikes}

Now we extend our theory of chainlike elements to all Coxeter groups. This introduces some non-simple chainlike elements which are impostors for the construction in Theorem~\ref{thm:construction}, as exemplified in Figure~\ref{fig:impostor}, which we will classify.

\begin{figure}[h]
\begin{center}
\begin{tikzpicture}
\node at (0,-1) {\phantom{p}Coxeter graph\phantom{d}};
\node[whitenode] (t) at (0,0) {$t$};
\node[blacknode] (s) at (0,2) {$s$};
\node[whitenode] (u) at (0,4) {$u$};
\begin{scope}[every edge/.style=graphedge]
	\path (t) edge (s);
	\path (s) edge node[left] {$4$} (u);
\end{scope}
\node at (4,-1) {\phantom{p}Bruhat poset\phantom{d}};
\begin{scope}[shift={(4,0)}, scale=0.7]
\node (e) at (0,0) {$e$};
\node (s) at (0,1) {$s$};
\node (ts) at (-1,2) {$ts$};
\node (us) at (1,2) {$us$};
\node (tus) at (-1,3) {$tus$};
\node (sus) at (1,3) {$sus$};
\node (tsus) at (-1,4) {$tsus$};
\node (stus) at (1,4) {$stus$};
\node (stsus) at (-1,5) {$stsus$};
\node (ustus) at (1,5) {$ustus$};
\node (ustsus) at (0,6) {$ustsus$};
\node (sustsus) at (0,7) {$sustsus$};
\path (e) edge (s);
\path (s) edge (ts) edge (us);
\path (tus) edge (ts) edge (us);
\path (sus) edge (us);
\path (tsus) edge (tus) edge (sus);
\path (stus) edge (tus) edge (sus);
\path (stsus) edge (tsus) edge (stus);
\path (ustus) edge (stus);
\path (ustsus) edge (stsus) edge (ustus) edge (sustsus);
\end{scope}
\node at (8,-1) {\phantom{p}Chainlike elements\phantom{d}};
\begin{scope}[shift={(8,0)}, scale=1]
\node (s) at (0,1) {$s$};
\node (ts) at (-1,2) {$ts$};
\node (us) at (1,2) {$us$};
\node (sus) at (1,3) {$sus$};
\path (s) edge (ts) edge (us);
\path (sus) edge (us);
\end{scope}
\node at (12,-1) {\phantom{p}Construction~\ref{thm:construction}\phantom{d}};
\begin{scope}[shift={(12,0)}, scale=0.8]
\node[whitenode] (t) at (0,0) {$ts$};
\node[blacknode] (s) at (0,2) {$s$};
\node[whitenode] (u) at (0,4) {$us$};
\node[whitenode] (v) at (0,6) {$sus$};
\begin{scope}[every edge/.style=graphedge]
	\path (t) edge (s);
	\path (u) edge (s) edge (v);
\end{scope}
\end{scope}
\end{tikzpicture}
\end{center}
\caption{Since Proposition~\ref{prop:link} (b) does not extend to non-simple elements, the construction in Theorem~\ref{thm:construction} may produce additional impostor nodes when there are labelled edges. In this example the impostor is $sus$, a form (II) element as in Theorem~\ref{thm:forms}.}\label{fig:impostor}
\end{figure}
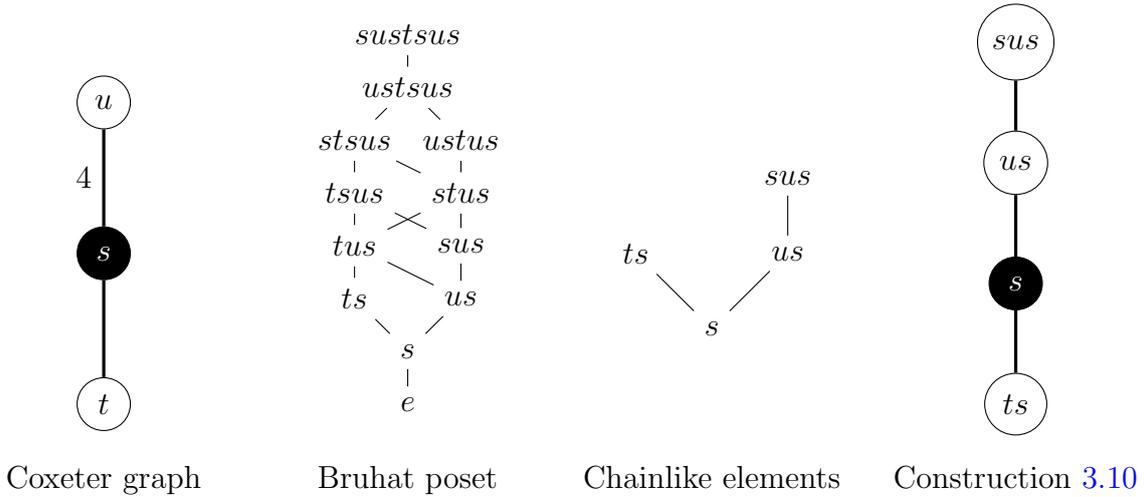

\begin{theorem}\label{thm:forms}
$w\in W^J$ is chainlike if and only if there are $s_0,\dots,s_k\in S$ such that $s_k\dots s_0$ is a simple chainlike and $w$ is in one of the following forms:
\begin{enumerate}[label=(\Roman*)]
	\item $w=s_k\dots s_0$, that is $w$ is simple.
	\item $w=s_ls_{l+1}\dots s_k\dots s_0$ with $0\leq l<k$, $m(s_k,s_{k-1})\geq4$ and $m(s_i,s_{i+1})=3$ for $l\leq i<k-1$.
	\item $w=\underbrace{\dots s_{k-1}s_ks_{k-1}s_ks_{k-1}}_{m\text{ terms}}s_{k-2}\dots s_0$ with $k\geq2$ and $4\leq m<m(s_k,s_{k-1})$.
\end{enumerate}
By the Word Property this is the unique reduced expression for $w$.
\end{theorem}

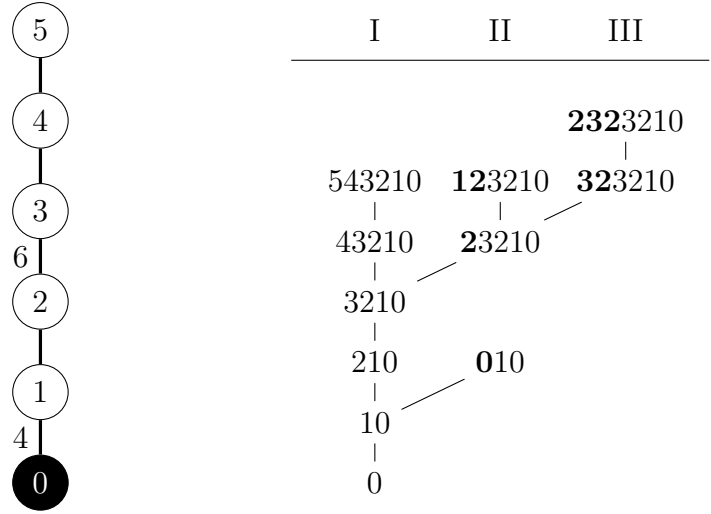
\begin{figure}[h]
\begin{center}
\begin{tikzpicture}[xscale=1.1,yscale=0.8]
\begin{scope}
    \node[blacknode] (0) at (0,0) {$0$};
    \node[whitenode] (1) at (0,1.5) {$1$};
    \node[whitenode] (2) at (0,3) {$2$};
    \node[whitenode] (3) at (0,4.5) {$3$};
    \node[whitenode] (4) at (0,6) {$4$};
    \node[whitenode] (5) at (0,7.5) {$5$};
\end{scope}
\begin{scope}[every edge/.style=graphedge]
    \path (0) edge node[left] {$4$} (1);
    \path (2) edge (1) edge node[left] {$6$} (3);
    \path (4) edge (3) edge (5);
\end{scope}
\begin{scope}[shift={(4,0)}]
\begin{scope}
    \node (0) at (0,0) {$0$};
    \node (10) at (0,1) {$10$};
    \node (210) at (0,2) {$210$};
    \node (3210) at (0,3) {$3210$};
    \node (43210) at (0,4) {$43210$};
    \node (543210) at (0,5) {$543210$};
    \node (010) at (1.5,2) {$\*{0}10$};
    \node (23210) at (1.5,4) {$\*{2}3210$};
    \node (123210) at (1.5,5) {$\*{12}3210$};
    \node (323210) at (3,5) {$\*{32}3210$};
    \node (2323210) at (3,6) {$\*{232}3210$};
\end{scope}
\begin{scope}[every edge/.style={draw=black}]
    \path (0) edge (10);
    \path (10) edge (210) edge (010);
    \path (210) edge (3210);
    \path (3210) edge (43210) edge (23210);
    \path (43210) edge (543210);
    \path (23210) edge (123210) edge (323210);
    \path (323210) edge (2323210);
\end{scope}
\draw (-1,7) -- (4,7);
\node at (0,7.5) {I};
\node at (1.5,7.5) {II};
\node at (3,7.5) {III};
\end{scope}
\end{tikzpicture}
\end{center}
\caption{A Coxeter quotient (left) and all of its chainlike elements (right). Elements of forms (II) and (III) have the part left of $s_k\dots s_0$ highlighted in bold.}\label{fig:forms}
\end{figure}

\begin{remark}
Figure~\ref{fig:forms} shows an example illustrating the three forms. The branching points where $u,v\in C(W^J)$ have $u'=v'$ are very limited and comprise three cases:
\begin{itemize}
	\item $u=s_ks_{k-1}\dots s_0$, $v=ts_{k-1}\dots s_0$ with $t\in S$, that is $u,v$ are both simple;
	\item $u=s_ks_{k-1}\dots s_0$, $v=s_{k-2}s_{k-1}\dots s_0$, that is $u$ is simple and $v$ is of form (II);
	\item $u=s_ks_{k-1}s_k\dots s_0$, $v=s_{k-2}s_{k-1}s_k\dots s_0$, that is $u,v$ are of forms (III) and (II).
\end{itemize}
Also note a particular oddity of form (II) elements, which is that the edges between generators in the part left of $s_k\dots s_0$ must be unlabelled. So in the example in Figure~\ref{fig:forms}, $0123210$ is not chainlike because of the label between 0 and 1. In particular, it covers the two distinct elements $123210$ and $023210=232010$ and so is not semi-chainlike.
\end{remark}

Many proofs later on will require looking for semi-chainlike elements, so before proving the theorem above let's state a lemma which will be useful for identifying these:

\begin{lemma}\label{lem:semi}
Suppose that $w\in W^J$ is semi-chainlike and $w=sw'=sty$ for some $s,t\in S$ and $y\in W^J$ with $y\vartriangleleft w'$. Then $t\leq y$ implies $s\leq y$.
\end{lemma}

\begin{proof}
Assume $t\leq y$ and $s\nleq y$. If $s\not\in J$ then $sy\in W^J$, but then $sy\vartriangleleft w$ so $sy=w'$ as $w$ is semi-chainlike. Then $s=t$ and $w=sty=y$, a contradiction. If instead $s\in J$ then in order to have $w'\in W^J$ we must have $m(s,a)\geq3$ for some generator $a\leq w'$. But then since $w'=ty$ either $a\leq y$, or $a=t$ in which case $a\leq y$ also, so $sy\in W^J$ as the single $s$ cannot be moved past $a$ in any expression. Then $sy=w'$, again a contradiction.
\end{proof}

\begin{proof}[Proof of Theorem~\ref{thm:forms}]
The base case $\ell(w)\leq2$ is the same as in the proof of Theorem~\ref{thm:simple}. We now continue by induction on $\ell(w)$.

Let's suppose $w$ is of one of the forms above.
\begin{itemize}
	\item If $w$ is of form (I), it is simple and therefore chainlike by Theorem~\ref{thm:simple}.
	\item Suppose $w=s_l\dots s_k\dots s_0$ is of form (II) and let $u\vartriangleleft w$ be obtained by removing a generator $s$. If $s$ is right of $s_k$ then $u\not\in W^J$ since $s_k\dots s_0$ is simple. If $s=s_k$, then we have a nil-move $s_{k-1}s_{k-1}$ and so $\ell(u)<\ell(w)-1$, a contradiction. If $s=s_i$ is between the first $s_l$ and $s_k$, then using $m(s_{i-1},s_i)=3$ we have
	\begin{align*}
		&&u=s_l\dots s_{i-1}s_{i+1}\dots s_k\dots s_0&=s_l\dots s_{i-2}s_{i+1}\dots s_k\dots s_{i+1}s_{i-1}s_is_{i-1}s_{i-2}\dots s_0\\
		&&&=s_l\dots s_{i-2}s_{i+1}\dots s_k\dots s_{i+1}s_is_{i-1}s_is_{i-2}\dots s_0\\
		&&&=s_l\dots s_{i-2}s_{i+1}\dots s_k\dots s_0s_i\not\in W^J
	\end{align*}
	The only remaining option is the leftmost generator $s=s_l$ giving $u$ chainlike by induction, and so $w$ is chainlike with $w'=u$.
	\item Suppose $w=\dots s_{k-1}s_ks_{k-1}s_k\dots s_0$ is of form (III) and let $u\vartriangleleft w$ be obtained by removing a generator $s$. If $s$ is right of the rightmost $s_k$ then $u\not\in W^J$ since $s_k\dots s_0$ is simple. If $s$ is any other generator except the first then we have a nil-move $s_ks_k$ or $s_{k-1}s_{k-1}$ and so $\ell(u)<\ell(w)-1$, a contradiction. The only remaining option is the leftmost generator giving $u$ chainlike by induction, so $w$ is chainlike with $w'=u$.
\end{itemize}

Now suppose $w$ is chainlike so that $w'$ is in one of the required forms by induction. Proposition~\ref{prop:unique} gives us that $w=sw'$ with $m(s,s_j)\geq3$, where $s_j$ is the leftmost generator in the expression for $w'$.
\begin{itemize}
	\item Suppose $w'$ is of form (I). If $s\nleq w$ then as in the proof of Theorem~\ref{thm:simple}, the only possibility for $w$ to be chainlike is if $w$ is simple, giving form (I). The only remaining possibility is $s=s_{k-1}$, for which if $m(s_k,s_{k-1})=3$ we can swap $s_k$ and $s_{k-1}$ with a braid move, violating Proposition~\ref{prop:unique}. Hence $m(s_k,s_{k-1})\geq4$, and we have form (II).
	\item Suppose $w'=s_l\dots s_k\dots s_0$ is of form (II).
	\begin{itemize}
		\item If $s=s_{l-1}$, then suppose for a contradiction that $m(s_{l-1},s_l)\geq4$. Consider removing $s_l$ to get $sw''=s_{l-1}s_{l+1}\dots s_k\dots s_0$. This is in $W^J$ since $s_{l-1}$ cannot move past $s_l$ or be replaced with a braid move, so $sw''\vartriangleleft w$. But $sw''\neq w'$, so $w$ is not semi-chainlike. Hence $m(s_{l-1},s_l)=3$ and $w$ is of form (II).
		\item If $s=s_{l+1}$ and $l<k-1$, then $w=s_{l+1}s_ls_{l+1}s_{l+2}\dots s_k\dots s_0$ and we have the braid move $s_{l+1}s_ls_{l+1}\to s_ls_{l+1}s_l$, violating Proposition~\ref{prop:unique}.
		\item If $s=s_{l+1}$ and $l=k-1$, then $w=s_ks_{k-1}s_ks_{k-1}\dots s_0$. If $m(s_k,s_{k-1})=4$, then we have the braid-move $s_ks_{k-1}s_ks_{k-1}\to s_{k-1}s_ks_{k-1}s_k$, violating Proposition~\ref{prop:unique}. Otherwise, $w$ is of form (III).
		\item If $s\not<w'$, then $w$ is not semi-chainlike by Lemma~\ref{lem:semi}.
	\end{itemize}
	\item Suppose $w'=s_j\dots s_{k-1}s_ks_{k-1}s_k\dots s_0$ is of form (III), with $s_j\in\{s_k,s_{k-1}\}$.
	\begin{itemize}
		\item Let $s\in\{s_k,s_{k-1}\}$ with $s\neq s_j$. If the number of $s_k,s_{k-1}$ terms in $w$ is at least $m(s_k,s_{k-1})$, then we have the braid move between $s_k$ and $s_{k-1}$, and so $w$ is not chainlike by Proposition~\ref{prop:unique}. Otherwise, $w$ is of form (III).
		\item If $s_j=s_{k-1}$ and $s=s_{k-2}$, then $s_{k-2}w''=s_{k-2}s_ks_{k-1}\dots s_k\dots s_0$ is in $W^J$ since the only possible braid-move is swapping $s_k$ and $s_{k-2}$. But $sw''\neq w'$, so $w$ is not semi-chainlike.
		\item If $s\not<w'$, then $w$ is not semi-chainlike by Lemma~\ref{lem:semi}.
	\end{itemize}
\end{itemize}
This covers all cases, so we have shown inductively that any chainlike $w$ is in one of the required forms.
\end{proof}

The problem of determining the Coxeter graph from the Bruhat poset can now be reduced to identifying the type of each chainlike element, since given this information we can apply the construction in Theorem~\ref{thm:construction} to the form (I) elements and then use the form (II) and (III) elements to determine the edge labels. We formalise this below. First let us generalise step 2 (b) in Theorem~\ref{thm:construction} to allow for labelled edges by introducing a function $M(u,v)$ which matches the value of $m(s,t)$, where $s,t$ are the leftmost generators in $u,v$ with $u'=v'$ (later we will generalise $M$ to apply to more pairs $u,v$).

\begin{definition}\label{def:mstart}
Suppose $u,v\in C(W^J)$ with $u'=v'$. We denote:
\begin{align*}
	M_1(u,v)&=\{u,v\};\\
	M_i(u,v)&=\{w\in W^J\mid z\vartriangleleft w\text{ for all }z\in M_{i-1}(u,v)\}\text{ for integers }i\geq2;\\
	M(u,v)&=\min\{i\geq2\mid |M_i(u,v)|=1\}
\end{align*}
with $M(u,v)=\infty$ if $|M_i(u,v)|\neq1$ for all $i\geq2$.
\end{definition}

\begin{prop}\label{prop:M}
Suppose $u,v\in C(W^J)$ are simple with $u'=v'$, and $s,t\in S$ with $u=su'$ and $v=tu'$. Then $m(s,t)=M(u,v)$.
\end{prop}

\begin{proof}
The reduced expressions of elements of $M_2(u,v)$ must contain $su'$ as a subword, or more precisely all the generators comprising the unique reduced word for $su'$ in order, and similarly for $tu'$. Thus $M_2(u,v)=\{tsu',stu'\}$. If $m(t,s)=2$ then $tsu'=stu'$ and $M(u,v)=2$, and otherwise $tsu',stu'$ have unique reduced expressions. Thus by the same argument inductively, we have
\[M_i(u,v)=\{\underbrace{\dots tstst}_{i\text{ terms}}u',\ \underbrace{\dots ststs}_{i\text{ terms}}u'\}\ \text{ for }i<m(s,t)\]
If $m(s,t)=\infty$ then $|M_i(u,v)|=2$ for all $i\geq2$. If instead $m(s,t)$ is finite then at $i=m(s,t)$ we have $\underbrace{ststs\dots}_{i\text{ terms}}=\underbrace{tstst\dots}_{i\text{ terms}}$, and so $|M_{m(s,t)}|=1$.
\end{proof}

\begin{theorem}\label{thm:graph}
Suppose that $W$ is irreducible and for each chainlike $w\in C(W^J)$ it is known whether $w$ is of form (I), (II) or (III). Then the following construction on $W^J$ produces the bw-Coxeter graph of the pair $(W,W_J)$:
\begin{enumerate}
	\item Restrict $\sim$ to the form (I) elements and draw a node for each equivalence class, colouring any nodes corresponding to length 1 elements black.
	\item For each pair of distinct equivalence classes $U,V\subseteq C(W^J)$, connect the corresponding nodes with an edge if there are $u\in U,v\in V$ with either:
	\begin{itemize}
		\item[(a)] $u=v'$ or $v=u'$, or
		\item[(b)] $u'=v'$ and $M(u,v)\geq3$.
	\end{itemize}
	If (b) holds and $M(u,v)\geq4$, then label the edge with $M(u,v)$.
	\item For each form (II) element $u$ such that $u'$ is of form (I), let $q\geq0$ be the number of form (III) elements greater than $u$, and label the edge between the nodes corresponding to the classes containing $u'$ and $u''$ with $q+4$.
\end{enumerate}
\end{theorem}

An example of this construction is demonstrated in Figure~\ref{fig:const}.

\begin{figure}[h]
\begin{center}
\begin{tikzpicture}[xscale=1.6,yscale=1]
\begin{scope}
	\node[circle,draw] (e) at (0,0) {};
	\node[draw] (0) at (0,1) {$u''$ (I)};
	\node[draw] (10) at (0,2) {$u'$ (I)};
	\node[draw] (210) at (-1,3) {$v$ (I)};
	\node[draw] (010) at (1,3) {$u$ (II)};
	\node[circle,draw] (0210) at (-1,4) {};
	\node[draw] (1010) at (1,4) {(III)};
	\node[circle,draw] (21010) at (-2,5) {};
	\node[circle,draw] (10210) at (0,5) {};
	\node[draw] (01010) at (2,5) {(III)};
	\node[circle,draw] (121010) at (-2,6) {};
	\node[circle,draw,label=above:{$\vdots$}] (021010) at (0,6) {};
	\node[circle,draw] (010210) at (2,6) {};
\end{scope}
\begin{scope}[every edge/.style={draw=black}]
	\path (0) edge (e) edge (10);
	\path (210) edge (10) edge (0210);
	\path (010) edge (10) edge (0210) edge (1010);
	\path (0210) edge (10210) edge (21010);
	\path (1010) edge (10210) edge (21010) edge (01010);
	\path (10210) edge (121010) edge (010210);
	\path (21010) edge (121010) edge (021010);
	\path (01010) edge (010210) edge (021010);
\end{scope}
\end{tikzpicture}\hspace{1em}
\begin{tikzpicture}[scale=1]
\begin{scope}
    \node[blacknode] (0) at (0,0) {$u''$};
    \node[whitenode] (1) at (0,2) {$u'$};
    \node[whitenode] (2) at (0,4) {$v$};
    \node at (-3.5,2) {\huge{$\rightarrow$}};
    \node at (-3.5,3) {Construction~\ref{thm:graph}};
\end{scope}
\begin{scope}[every edge/.style=graphedge]
    \path (1) edge node[left] {$6$} (0) edge (2);
\end{scope}
\end{tikzpicture}
\end{center}
\caption{Part of a Bruhat poset $W^J$ (left) with chainlike elements and forms identified. The construction produces a graph (right) with a node for each equivalence class of form (I) elements, and the labelled edge 6 is added by step 3.}\label{fig:const}
\end{figure}
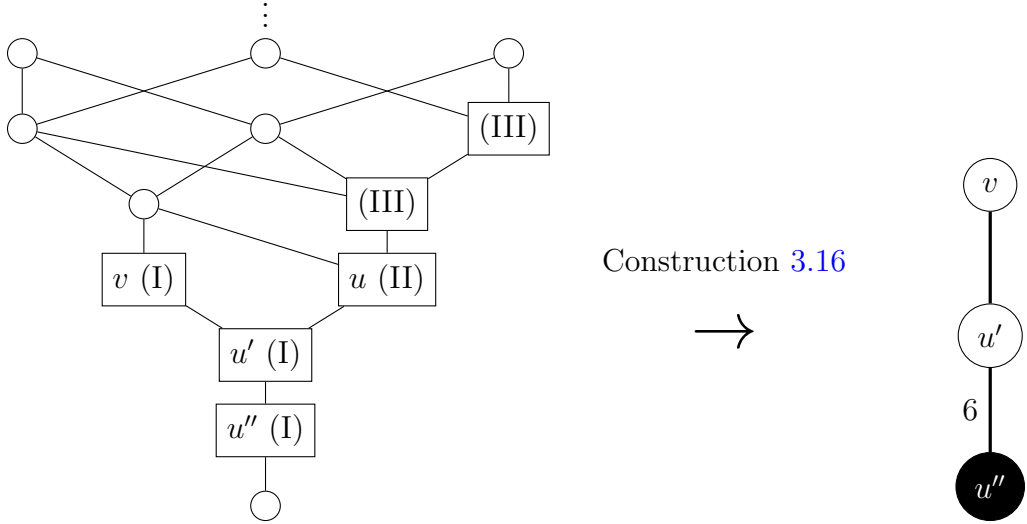

\begin{proof}
As in Theorem~\ref{thm:construction}, steps 1 and 2 in the construction correctly produce all nodes and edges in the bw-Coxeter graph of $W$, and we now check edge labels. For every pair $s,t\in S$ with $m(s,t)\geq3$, by Theorem~\ref{thm:construction} we have at least one of the following:
\begin{itemize}
	\item There are $u,v\in C(W^J)$ with leftmost generators $s,t$ respectively and $u'=v'$. Then by Proposition~\ref{prop:M} step 2 correctly labels the edge.
	\item There is $w\in C(W^J)$ with $w=stw''$. Then by Theorem~\ref{thm:forms} $u=tstw''$ is chainlike of form (II) if and only if $m(s,t)\geq4$, and if this holds then there are exactly $m(s,t)-4$ form (III) chainlikes above $u$, so step 3 correctly labels the edge.
	\item There is $w\in C(W^J)$ with $w=tsw''$, so we have the same result as above but with $s$ and $t$ reversed.
\end{itemize}
Hence all labelled edges match those in the bw-Coxeter graph.
\end{proof}

\subsection{Reducing to irreducibles}\label{subsec:reducible}

Let us now prove Theorem~\ref{thm:reducible} so that Theorem~\ref{thm:graph} can be extended to reducible Coxeter systems as well.

\begin{prop} We have the following results:
\begin{enumerate}[label=(\alph*)]
\item For distinct $s_1,s_2\in S\setminus J$, write $s_1\leftrightarrow s_2$ to mean either $m(s_1,s_2)\geq3$ or there exist chainlikes $w_1\geq s_1,w_2\geq s_2$ with $w_1\sim w_2$ (or both). Then $s_a,s_b\in S\setminus J$ are in the same connected component of the graph of $(W,W_J)$ if and only if there is a sequence $s_a=s_1\leftrightarrow s_2\leftrightarrow\dots\leftrightarrow s_n=s_b$ for some $n\in\N$.
\item Let $S=S_1\sqcup\dots\sqcup S_n$ be the connected components of the graph of $(W,W_J)$. For each $1\leq i\leq n$, the poset $(W_{S_i})^{S_i\cap J}$ is equal to the subposet
\[W^J_i=\{w\in W^J\mid s\leq w\implies s\in S_i\text{ for all }s\in S\setminus J\}.\]
\end{enumerate}
\end{prop}

\begin{proof}
For (a), suppose $s_a,s_b$ are in the same connected component. Then there is a sequence $s_a=s_1,s_2,\dots,s_n=s_b$ with each pair $s_i,s_{i+1}$ connected by either an edge, in which case $m(s_i,s_{i+1})\geq3$, or by a sequence of nodes in $J$, in which case selecting one of these nodes and using Lemma~\ref{lem:leftmost} and Proposition~\ref{prop:link} gives chainlikes $w_i\sim w_{i+1}$ with $w_i>s_i$, $w_{i+1}>s_{i+1}$. Thus $s_1\leftrightarrow\dots\leftrightarrow s_n$. Conversely if $s_a=s_1\leftrightarrow\dots\leftrightarrow s_n=s_b$ then each pair $s_i,s_{i+1}$ either has an edge or there are chainlikes $w_i\sim w_{i+1}$ with $w_i\geq s_i$, $w_{i+1}\geq s_{i+1}$. By Theorem~\ref{thm:forms} adjacent generators in the reduced expression for a chainlike element have $m\geq3$, and the leftmost generators in $w_i,w_{i+1}$ are the same by Proposition~\ref{prop:link}, so the expressions for $w_i$ and $w_{i+1}$ give a sequence of nodes connecting $s_i$ and $s_{i+1}$. Thus each pair $s_i,s_{i+1}$ is in the same connected component, and so are $s_a,s_b$.

For (b), we clearly have $(W_{S_i})^{S_i\cap J}\subseteq W^J_i$ since $s\leq w\in W_{S_i}$ only if $s\in S_i$. Now suppose a reduced expression for $w\in W^J_i$ contains a generator $s\in S\setminus S_i$ (and so necessarily $s\in J$), and assume $s$ is the rightmost such generator. Then $s$ can be moved to the right end of the word with braid-moves as it commutes with elements of $S_i$, a contradiction. Thus $W^J_i\subseteq(W_{S_i})^{S_i\cap J}$ and we are done.
\end{proof}

\begin{proof}[Proof of Theorem~\ref{thm:reducible}]
In the isomorphism $W^J\cong U^K$ each non-trivial subposet $W^i$ must map to a unique $U^j$. Moreover, $W^i$ is trivial if and only if the connected component corresponding to $W^i$ only contains generators in $J$, and similarly for each $U^j$. Thus for each trivial $W^i$ we can add an empty set to the disjoint union of connected components $T=T_1\sqcup\dots\sqcup T_{n_2}$ and conversely for each trivial $U^j$ add an empty set to $S=S_1\sqcup\dots\sqcup S_{n_1}$, bringing the unions to the same number of components $n$ and satisfying $(W_{S_i})^{S_i\cap J}\cong (U_{T_i})^{T_i\cap K}$ for each $1\leq i\leq n$.
\end{proof}

\subsection{Detectors and recovering $m(s,t)$}

The next step is to identify the form of $w\in C(W^J)$ using only its position in the poset. Since certain quotients have isomorphic posets there will be cases where this cannot be achieved, and Section~\ref{sec:sudoku} is devoted to identifying these cases. Here we set up key pieces of machinery which we use throughout Section~\ref{sec:sudoku}.

Firstly, we can extend the domain of the function $M$ to include all pairs of chainlikes $u,x$ such that $x$ branches off from the chain $\dots\to u''\to u'\to u$, that is $x'<u'$ and $x\not<u$ (the case $x'=u'$ is already covered by Definition~\ref{def:mstart}). Below we show that for these pairs $(u,x)$ we can define $M$ in terms of the poset so that $M(u,x)=m(s,t)$, where $s,t$ are the leftmost generators in $u,x$ respectively. Since we are most interested in cases where $m(s,t)\geq3$, we give this a special name and say that $x$ detects $u$ (defined more rigorously below). For example, in Figure~\ref{fig:forms} the element $x=010$ detects $u=123210$, since the leftmost generators in each (0 and 1) do not commute.

In most cases (in particular when there are no cycles in the graph) detected elements are always form (II) or (III), hence this method provides a way of revealing `impostor' chainlikes; in Figure~\ref{fig:impostor} for example, we can use the fact that $ts$ detects $sus$ to show that $sus$ is non-simple. Making this rigorous however will take the entirety of Section~\ref{sec:sudoku}.

\begin{definition}
For $u,x\in W^J$ we denote
\[X(u,x)=\{w\in W^J\mid u\leq w,\,x\leq w,\,\ell(w)=L(u,x)\}\]
Suppose $u,x\in C(W^J)$ with $x'<u'$ and $x\not<u$. If $|X(u,x)|=|X(u',x)|+1$, then we say $u$ is \textit{detected by $x$}.
\end{definition}

\begin{prop}\label{prop:detector}
Suppose $u,x\in C(W^J)$ with $x'<u'$ and $x\not<u$. If $u=su'$ and $x=tx'$ with $s,t\in S$, then $u$ is detected by $x$ if and only if $m(s,t)\geq3$. Moreover, if this holds then there is exactly 1 element in $X(u',x)$ covered by 2 distinct elements in $X(u,x)$, and these elements are $stu'\vartriangleright tu'\vartriangleleft tsu'$.
\end{prop}

\begin{proof}
Since $x'<u'$, let $u$ and $x$ have reduced expressions $u=ss_k\dots s_0$ (not necessarily simple) and $x=ts_l\dots s_0$ with $-1\leq l<k$ (where $l=-1$ means $x=t\in S$). We have $t\neq s$ and also $t\neq s_i$ for $l<i\leq k$ since otherwise $x<u$. We then have $L(u',x)=\ell(u')+1$ since $\ul r=x'$ is a maximal length subword of $u'$ and $x$. Then if $w\in W^J$ with $\ell(w)=\ell(u')+1$, $u'\leq w$ and $x\leq w$, by the Subword Property $w$ can be obtained by inserting $t$ into $s_k\dots s_0$ to the left of $s_l$ (or anywhere if $x=t$), so
\[X(u',x)=\{w_i=s_k\dots s_{i+1}ts_i\dots s_0\mid l\leq i\leq k\}\]
We have $w_i=w_{i-1}$ if and only if $m(t,s_i)\geq3$, so starting at $i=l$ we have one element $w_l$, and then as we increment $i$ we encounter a new element each time $m(t,s_i)\geq3$. Hence,
\[|X(u',x)|=1+\big|\{i\in\N\mid l<i\leq k,\,m(t,s_i)\geq3\}\big|.\]
Repeating the above but exchanging $u'$ for $u$, we obtain 
\[|X(u,x)|=1+\big|\{i\in\N\mid l<i\leq k+1,\,m(t,s_i)\geq3\}\big|\text{ where }s_{k+1}=s\]
and thus $|X(u,x)|=|X(u',x)|+1$ if and only if $m(s,t)\geq3$. For the second claim, consider $w_i$ as above. We have $sw_i\in X(u,x)$, and $sw_i\vartriangleright w_j\implies w_i=w_j$ for any $i,j$, since otherwise $w_j$ would have an expression beginning with $s$, a contradiction. The only element of $X(u,x)$ not of the form $sw_i$ is $tu=tsu'$, which covers $w_k=tu'$. 
\end{proof}

\begin{definition}\label{def:mend}
Suppose $u,x\in C(W^J)$ with $x'<u'$ and $x\not<u$. If $u$ is not detected by $x$, we set $M(u,x)=2$. Otherwise, let $w_1,w_2$ be the unique pair of elements of $X(u,x)$ with $w_1\vartriangleright w\vartriangleleft w_2$ for some $w\in X(u',x)$ as given by Proposition~\ref{prop:detector}, and we denote:
\begin{align*}
	M_2(u,x)&=\{w_1,w_2\};\\
	M_i(u,x)&=\{w\in W^J\mid z\vartriangleleft w\text{ for all }z\in M_{i-1}(u,x)\}\text{ for integers }i\geq3;\\
	M(u,x)&=\min\{i\geq2\mid |M_i(u,x)|=1\}
\end{align*}
with $M(u,x)=\infty$ if $|M_i(u,x)|\neq1$ for all $i\geq2$.
\end{definition}

\begin{thm}
Suppose $u,x\in C(W^J)$ with $x'\leq u'$ and $x\not<u$, and let $s,t$ be the leftmost generators in the unique reduced expressions for $u,x$ respectively. Then
\[M(u,x)=m(s,t)\]
where $M(u,x)$ is given by Definition~\ref{def:mstart} if $x'=u'$ or Definition~\ref{def:mend} if $x'<u'$.
\end{thm}

\begin{proof}
If $u=ss_k\dots s_0$ is detected by $x=ts_l\dots s_0$ (both not necessarily simple), then by the same reasoning as in the proof of Proposition~\ref{prop:M} we have for $i\leq m(t,s)$,
\[M_i(u,x)=\{\underbrace{\dots tstst}_{i\text{ terms}}s_k\dots s_0,\ \underbrace{\dots ststs}_{i\text{ terms}}s_k\dots s_0\}\]
and $M(u,x)=m(t,s)$. If $u$ is not detected by $x$ then $m(t,s)=2$ and $M(u,x)=2$ by definition. The only case not checked is $u'=x'$ with either $u$ or $x$ not simple, for which we have $\{u,x\}$ either $\{s_k\dots s_0,s_{k-2}s_{k-1}\dots s_0\}$ or $\{s_ks_{k-1}s_k\dots s_0,s_{k-2}s_{k-1}s_k\dots s_0\}$, and $M(u,x)=2=m(s_k,s_{k-2})=m(s,t)$ in both cases.
\end{proof}

\begin{remark}\label{rem:detectors}
Suppose $u,x\in C(W^J)$ with $x'<u'$ and $x\not<u$. Using the classification in Theorem~\ref{thm:forms} we can list all cases where $u$ is detected by $x$; most of Section~\ref{sec:sudoku} will be spent looking for other features in the posets to distinguish these. The table below shows all cases where $u$ is detected by $x$ with $\ell(x)\geq2$. All differently labelled generators (including the unlabelled $s$) are distinct.

\begin{center}\begin{tabular}{l|l|l}
$u$ & $x$ & Conditions \\\hline\hline
$s_k\dots s_0$ (I) & $ss_j\dots s_0$ (I) & $m(s,s_k)\geq3$, $0\leq j\leq k-2$\\\hline
$s_l\dots s_k\dots s_0$ (II) & $ss_j\dots s_0$ (I) & $m(s,s_l)\geq3$, $0\leq j\leq l$ \\
$s_l\dots s_k\dots s_0$ (II) & $s_{l-1}s_l\dots s_0$ (II) & $m(s_l,s_{l-1})\geq4$, $l\neq0$ \\\hline
$s_k\dots s_ks_{k-1}s_k\dots s_0$ (III) & $ss_j\dots s_0$ (I) & $m(s,s_k)\geq3$, $0\leq j\leq k$ \\\hline
$s_{k-1}\dots s_ks_{k-1}s_k\dots s_0$ (III) & $ss_j\dots s_0$ (I) & $m(s,s_{k-1})\geq3$, $0\leq j\leq k-1$ \\
$s_{k-1}\dots s_ks_{k-1}s_k\dots s_0$ (III) & $s_{k-2}s_{k-1}\dots s_0$ (II) & $m(s_{k-1},s_{k-2})\geq4$, $k\geq2$ \\
$s_{k-1}\dots s_ks_{k-1}s_k\dots s_0$ (III) & $s_{k-2}s_{k-1}s_k\dots s_0$ (II) & $m(s_{k-1},s_{k-2})=3$, $k\geq2$ \\\hline
\end{tabular}\end{center}\vspace{5pt}

Notice that there is only one case for which $u$ is simple, and for it to occur there must be a cycle in the graph, $s-s_j-s_{j+1}-\dots-s_k-s$. We will see later in Lemma~\ref{lem:A} that this case can frequently be separated from the others.
\end{remark}

If $x=s\in S\setminus J$ then $u$ is detected by $x$ if and only if $m(t,x)\geq3$ where $u=tu'$. We can say a little more about this case:

\begin{prop}\label{prop:black}
Suppose $u$ is detected by $s\in S\setminus J$. If there is $v\in C(W^J)$ of length 2 with $v'=s$ and $u\sim v$, then $u$ is simple.
\end{prop}

\begin{proof}
We have $v=ts$ where $t$ is the leftmost generator of $u$. If $u$ is not simple, then $t$ appears at least twice in the expression for $u$, so $\ul r=t$ is a maximal subword of both pairs $\{u,v\}$ and $\{u',v\}$. Thus $L(u,v)=L(u',v)+1$, contradicting $u\sim v$.
\end{proof}

We introduce one more general piece of machinery in this section for distinguishing posets: in the case where $u'=v'$ and $M(u,v)=2$ we define sets $N_i$ of semi-chainlike elements above $u$ and $v$, whose generator expressions have similarities with chainlike elements. The sets $N_i$ turn out to each only have one element, but we denote them as sets in advance of proving this fact.

\begin{definition}
For $u,v\in C(W^J)$ with $u'=v'$ and $M(u,v)=2$, we denote:
\begin{align*}
	N_0(u,v)&=M_2(u,v)=\{w\in W^J\mid u\vartriangleleft w,\,v\vartriangleleft w\}\quad\text{(which only has 1 element)};\\
	N_i(u,v)&=\{w\in W^J\mid w\text{ semi-chainlike with }w'\in N_{i-1}(u,x)\}\text{ for integers }i\geq1;\\
	N(u,v)&=\min\{i\geq1\mid N_i(u,v)=\emptyset\}
\end{align*}
where we take $N(u,v)=\infty$ if $N_i(u,v)$ is non-empty for all $i\geq1$.
\end{definition}

\begin{prop}\label{prop:N}
Suppose $u,v\in C(W^J)$ with $u'=v'$.
\begin{enumerate}[label=(\alph*)]
	\item If $u=ss_k\dots s_0$ and $v=ts_k\dots s_0$ are both simple with $m(s,t)=2$, then suppose $l\in\{0,\dots,k\}$ is minimal with $m(s_i,s_{i+1})=3$ for $l\leq i<k$. Then
	\begin{align*}
		N(u,v)&=\begin{cases}k-l+2&\text{if }m(s,s_k)=m(t,s_k)=3\\1&\text{otherwise}\end{cases}
	\shortintertext{\item If $u=s_{k-1}s_k\dots s_0$ is of form (II) and $v=ss_k\dots s_0$ is simple, then}
		N(u,v)&=\begin{cases}3&\text{if }m(s,s_k)=3\text{ and }m(s_k,s_{k-1})=4\\1&\text{otherwise}\end{cases}
	\shortintertext{\item If $u=s_ks_{k-1}s_k\dots s_0$ is of form (III) and $v=s_{k-2}s_{k-1}s_k\dots s_0$ is of form (II), then}
		N(u,v)&=\begin{cases}5&\text{if }m(s_k,s_{k-1})=5\\1&\text{otherwise}\end{cases}
	\end{align*}
\end{enumerate}
\end{prop}

Examples of all of these cases are given in Figure~\ref{fig:N} at the end of this section. This $N$ function will be used in the last part of Section~\ref{sec:sudoku}. For example, in Figure~\ref{fig:h3} we can see that in the $H_3$ poset we have $N(12321,32321)=5$, which is case (b) in Proposition~\ref{prop:N}, while in the $D_6$ poset we have $N(54321,64321)=k-l+2=3-0+2=5$, which is case (a). This provides an explanation for why the exceptional pair $(H_3,H_2)\leftrightarrow(D_6,D_5)$ does not extend to an infinite family of isomorphisms; for other Coxeter pairs $(D_n,D_{n-1})$ we obtain a different value than 5 for $k-l+2$. This particular case is formalised in Proposition~\ref{prop:n5}.

\begin{figure}\centering
\begin{tikzpicture}
\begin{scope}[shift={(-3,0)}]
\node[smallblacknode] (0) at (0,0) {\scriptsize{$0$}};
\node[smallwhitenode] (1) at (0,1) {\scriptsize{$1$}};
\node[smallwhitenode] (2) at (0,2) {\scriptsize{$2$}};
\node[smallwhitenode] (3) at (-0.5,2.866) {\scriptsize{$3$}};
\node[smallwhitenode] (4) at (0.5,2.866) {\scriptsize{$4$}};
\begin{scope}[every edge/.style=graphedge]
\path (0) edge node[left] {$4$} (1);
\path (2) edge (1) edge (3) edge (4);
\end{scope}
\end{scope}
\begin{scope}[shift={(0,0)}, xscale=1.3, yscale=0.7, every node/.style={align=center}]
\node (e) at (0,0) {\tiny{$e$}};
\node (0) at (0.0,1) {\tiny{$0$}};
\node (10) at (0.0,2) {\tiny{$10$}};
\node (210) at (-0.5,3) {\tiny{$210$}};
\node (010) at (0.5,3) {\tiny{$010$}};
\node (0210) at (1.0,4) {\tiny{$0210$}};
\node (4210) at (-1.0,4) {\tiny{$\*{4210}$}};
\node (3210) at (0.0,4) {\tiny{$\*{3210}$}};
\node (34210) at (-1.5,5) {\tiny{$\*{34210}$}};
\node (03210) at (0.5,5) {\tiny{$03210$}};
\node (04210) at (-0.5,5) {\tiny{$04210$}};
\node (10210) at (1.5,5) {\tiny{$10210$}};
\node (104210) at (0.0,6) {\tiny{$104210$}};
\node (103210) at (1.0,6) {\tiny{$103210$}};
\node (010210) at (2.0,6) {\tiny{$010210$}};
\node (034210) at (-1.0,6) {\tiny{$034210$}};
\node (234210) at (-2.0,6) {\tiny{$\*{234210}$}};
\node (1234210) at (-2.5,7) {\tiny{$\*{1234210}$}};
\node (0234210) at (-1.5,7) {\tiny{$0234210$}};
\node (1034210) at (-0.5,7) {\tiny{$1034210$}};
\node (0104210) at (1.5,7) {\scriptsize{$\cdots$}}; 
\node (2104210) at (0.5,7) {\tiny{$2104210$}};
\node (32104210) at (-1,8) {\scriptsize{$\cdots$}}; 
\node (10234210) at (-2,8) {\tiny{$10234210$}};
\node (01234210) at (-3,8) {\tiny{$01234210$}};
\path (0) edge (10) edge (e);
\path (210) edge (10);
\path (010) edge (10);
\path (0210) edge (210) edge (010);
\path (4210) edge (210);
\path (3210) edge (210);
\path (34210) edge (4210) edge (3210);
\path (03210) edge (3210) edge (0210);
\path (04210) edge (4210) edge (0210);
\path (10210) edge (0210);
\path (104210) edge (04210) edge (10210);
\path (103210) edge (03210) edge (10210);
\path (010210) edge (10210);
\path (034210) edge (34210) edge (04210) edge (03210);
\path (234210) edge (34210);
\path (1234210) edge (234210);
\path (0234210) edge (234210) edge (034210);
\path (1034210) edge (034210) edge (104210) edge (103210);
\path (2104210) edge (104210);
\path (10234210) edge (0234210) edge (1234210) edge (1034210);
\path (01234210) edge (1234210) edge (0234210);
\end{scope}
\begin{scope}[shift={(5,0)}]
\node[smallblacknode] (0) at (0,0) {\scriptsize{$0$}};
\node[smallwhitenode] (1) at (0,1) {\scriptsize{$1$}};
\node[smallwhitenode] (2) at (0,2) {\scriptsize{$2$}};
\node[smallwhitenode] (3) at (-0.5,2.866) {\scriptsize{$3$}};
\node[smallwhitenode] (4) at (0.5,2.866) {\scriptsize{$4$}};
\begin{scope}[every edge/.style=graphedge]
\path (0) edge node[left] {$4$} (1);
\path (2) edge (1) edge node[below left] {$4$} (3) edge (4);
\end{scope}
\end{scope}
\begin{scope}[shift={(7.5,0)}, xscale=1.1, yscale=0.85]
\node (e) at (0,0) {\tiny{$e$}};
\node (0) at (0.0,1) {\tiny{$0$}};
\node (10) at (0.0,2) {\tiny{$10$}};
\node (210) at (-0.5,3) {\tiny{$210$}};
\node (010) at (0.5,3) {\tiny{$010$}};
\node (0210) at (1.0,4) {\tiny{$0210$}};
\node (4210) at (-1.0,4) {\tiny{$\*{4210}$}};
\node (3210) at (0.0,4) {\tiny{$\*{3210}$}};
\node (34210) at (-2.0,5) {\tiny{$\*{34210}$}};
\node (23210) at (0.0,5) {\tiny{$23210$}};
\node (03210) at (1.0,5) {\tiny{$03210$}};
\node (04210) at (-1.0,5) {\tiny{$04210$}};
\node (10210) at (2.0,5) {\tiny{$10210$}};
\node (104210) at (0.5,6) {\tiny{$104210$}};
\node (010210) at (1.5,6) {\scriptsize{$\cdots$}}; 
\node (034210) at (-1.5,6) {\tiny{$034210$}};
\node (423210) at (-0.5,6) {\tiny{$423210$}};
\node (234210) at (-2.5,6) {\tiny{$234210$}};
\path (0) edge (10) edge (e);
\path (210) edge (10);
\path (010) edge (10);
\path (0210) edge (210) edge (010);
\path (4210) edge (210);
\path (3210) edge (210);
\path (34210) edge (4210) edge (3210);
\path (23210) edge (3210);
\path (03210) edge (3210) edge (0210);
\path (04210) edge (4210) edge (0210);
\path (10210) edge (0210);
\path (104210) edge (04210) edge (10210);
\path (034210) edge (34210) edge (04210) edge (03210);
\path (423210) edge (23210) edge (34210);
\path (234210) edge (34210) edge (23210);
\end{scope}
\begin{scope}[shift={(-3,-7)}]
\node[smallblacknode] (0) at (0,0) {\scriptsize{$0$}};
\node[smallwhitenode] (1) at (0,1) {\scriptsize{$1$}};
\node[smallwhitenode] (2) at (0,2) {\scriptsize{$2$}};
\node[smallwhitenode] (3) at (0,3) {\scriptsize{$3$}};
\node[smallwhitenode] (4) at (0,4) {\scriptsize{$4$}};
\begin{scope}[every edge/.style=graphedge]
\path (1) edge (0) edge (2);
\path (3) edge node[left] {$4$} (2) edge (4);
\end{scope}
\end{scope}
\begin{scope}[shift={(0,-7)}, xscale=1.8, yscale=0.6]
\node (e) at (0,0) {\tiny{$e$}};
\node (0) at (0.0,1) {\tiny{$0$}};
\node (10) at (0.0,2) {\tiny{$10$}};
\node (210) at (0.0,3) {\tiny{$210$}};
\node (3210) at (0.0,4) {\tiny{$3210$}};
\node (43210) at (-0.5,5) {\tiny{$\*{43210}$}};
\node (23210) at (0.5,5) {\tiny{$\*{23210}$}};
\node (243210) at (-0.5,6) {\tiny{$\*{243210}$}};
\node (123210) at (0.5,6) {\tiny{$123210$}};
\node (1243210) at (0.0,7) {\tiny{$1243210$}};
\node (0123210) at (1.0,7) {\tiny{$0123210$}};
\node (3243210) at (-1.0,7) {\tiny{$\*{3243210}$}};
\node (23243210) at (-1.0,8) {\tiny{$\*{23243210}$}};
\node (13243210) at (0.0,8) {\tiny{$13243210$}};
\node (01243210) at (1.0,8) {\tiny{$01243210$}};
\node (013243210) at (1.0,9) {\scriptsize{$\dots\qquad$}}; 
\node (213243210) at (-1.0,9) {\tiny{$213243210$}};
\node (123243210) at (0.0,9) {\tiny{$123243210$}};
\path (0) edge (10) edge (e);
\path (210) edge (10);
\path (3210) edge (210);
\path (43210) edge (3210);
\path (23210) edge (3210);
\path (243210) edge (43210) edge (23210);
\path (123210) edge (23210);
\path (1243210) edge (243210) edge (123210);
\path (0123210) edge (123210);
\path (3243210) edge (243210);
\path (23243210) edge (3243210);
\path (13243210) edge (3243210) edge (1243210);
\path (01243210) edge (1243210) edge (0123210);
\path (213243210) edge (13243210) edge (23243210);
\path (123243210) edge (23243210) edge (13243210);
\end{scope}
\begin{scope}[shift={(5,-7)}]
\node[smallblacknode] (0) at (0,0) {\scriptsize{$0$}};
\node[smallwhitenode] (1) at (0,1) {\scriptsize{$1$}};
\node[smallwhitenode] (2) at (0,2) {\scriptsize{$2$}};
\node[smallwhitenode] (3) at (0,3) {\scriptsize{$3$}};
\node[smallwhitenode] (4) at (0,4) {\scriptsize{$4$}};
\begin{scope}[every edge/.style=graphedge]
\path (1) edge (0) edge (2);
\path (3) edge node[left] {$4$} (2) edge node[left] {$4$} (4);
\end{scope}
\end{scope}
\begin{scope}[shift={(7.5,-7)}, xscale=1.4, yscale=0.75]
\node (e) at (0,0) {\tiny{$e$}};
\node (0) at (0.0,1) {\tiny{$0$}};
\node (10) at (0.0,2) {\tiny{$10$}};
\node (210) at (0.0,3) {\tiny{$210$}};
\node (3210) at (0.0,4) {\tiny{$3210$}};
\node (43210) at (-0.5,5) {\tiny{$\*{43210}$}};
\node (23210) at (0.5,5) {\tiny{$\*{23210}$}};
\node (243210) at (0.0,6) {\tiny{$\*{243210}$}};
\node (123210) at (1.0,6) {\tiny{$123210$}};
\node (343210) at (-1.0,6) {\tiny{$343210$}};
\node (2343210) at (-0.5,7) {\tiny{$2343210$}};
\node (1243210) at (0.5,7) {\tiny{$1243210$}};
\node (0123210) at (1.5,7) {\scriptsize{$\dots$}}; 
\node (3243210) at (-1.5,7) {\tiny{$3243210$}};
\path (0) edge (e);
\path (10) edge (0);
\path (210) edge (10);
\path (3210) edge (210);
\path (43210) edge (3210);
\path (23210) edge (3210);
\path (243210) edge (43210) edge (23210);
\path (123210) edge (23210);
\path (343210) edge (43210);
\path (2343210) edge (343210) edge (243210);
\path (1243210) edge (243210) edge (123210);
\path (3243210) edge (243210) edge (343210);
\end{scope}
\begin{scope}[shift={(-3,-15)}]
\node[smallblacknode] (0) at (0,0) {\scriptsize{$0$}};
\node[smallwhitenode] (1) at (0,1) {\scriptsize{$1$}};
\node[smallwhitenode] (2) at (0,2) {\scriptsize{$2$}};
\node[smallwhitenode] (3) at (0,3) {\scriptsize{$3$}};
\begin{scope}[every edge/.style=graphedge]
\path (1) edge (0) edge (2);
\path (3) edge node[left] {$5$} (2);
\end{scope}
\end{scope}
\begin{scope}[shift={(0,-15)}, xscale=2.2, yscale=0.55]
\node (e) at (0,0) {\tiny{$e$}};
\node (0) at (0.0,1) {\tiny{$0$}};
\node (10) at (0.0,2) {\tiny{$10$}};
\node (210) at (0.0,3) {\tiny{$210$}};
\node (3210) at (0.0,4) {\tiny{$3210$}};
\node (23210) at (0.0,5) {\tiny{$23210$}};
\node (323210) at (-0.5,6) {\tiny{$\*{323210}$}};
\node (123210) at (0.5,6) {\tiny{$\*{123210}$}};
\node (1323210) at (-0.5,7) {\tiny{$\*{1323210}$}};
\node (0123210) at (0.5,7) {\tiny{$0123210$}};
\node (01323210) at (0.5,8) {\tiny{$01323210$}};
\node (21323210) at (-0.5,8) {\tiny{$\*{21323210}$}};
\node (321323210) at (-0.5,9) {\tiny{$\*{321323210}$}};
\node (021323210) at (0.5,9) {\tiny{$021323210$}};
\node (0321323210) at (0.0,10) {\tiny{$0321323210$}};
\node (1021323210) at (1.0,10) {\tiny{$1021323210$}};
\node (2321323210) at (-1.0,10) {\tiny{$\*{2321323210}$}};
\node (12321323210) at (-1.0,11) {\tiny{$\*{12321323210}$}};
\node (02321323210) at (0.0,11) {\tiny{$02321323210$}};
\node (10321323210) at (1.0,11) {\tiny{$10321323210$}};
\node (210321323210) at (1.0,12) {\scriptsize{$\dots\qquad$}}; 
\node (102321323210) at (0.0,12) {\tiny{$102321323210$}};
\node (012321323210) at (-1.0,12) {\tiny{$012321323210$}};
\path (0) edge (10) edge (e);
\path (210) edge (10);
\path (3210) edge (210);
\path (23210) edge (3210);
\path (323210) edge (23210);
\path (123210) edge (23210);
\path (1323210) edge (323210) edge (123210);
\path (0123210) edge (123210);
\path (01323210) edge (1323210) edge (0123210);
\path (21323210) edge (1323210);
\path (321323210) edge (21323210);
\path (021323210) edge (21323210) edge (01323210);
\path (0321323210) edge (321323210) edge (021323210);
\path (1021323210) edge (021323210);
\path (2321323210) edge (321323210);
\path (12321323210) edge (2321323210);
\path (02321323210) edge (2321323210) edge (0321323210);
\path (10321323210) edge (0321323210) edge (1021323210);
\path (102321323210) edge (02321323210) edge (12321323210) edge (10321323210);
\path (012321323210) edge (12321323210) edge (02321323210);
\end{scope}
\begin{scope}[shift={(5,-15)}]
\node[smallblacknode] (0) at (0,0) {\scriptsize{$0$}};
\node[smallwhitenode] (1) at (0,1) {\scriptsize{$1$}};
\node[smallwhitenode] (2) at (0,2) {\scriptsize{$2$}};
\node[smallwhitenode] (3) at (0,3) {\scriptsize{$3$}};
\begin{scope}[every edge/.style=graphedge]
\path (1) edge (0) edge (2);
\path (3) edge node[left] {$6$} (2);
\end{scope}
\end{scope}
\begin{scope}[shift={(7.5,-15)}, xscale=1.5, yscale=0.72]
\node (e) at (0,0) {\tiny{$e$}};
\node (0) at (0.0,1) {\tiny{$0$}};
\node (10) at (0.0,2) {\tiny{$10$}};
\node (210) at (0.0,3) {\tiny{$210$}};
\node (3210) at (0.0,4) {\tiny{$3210$}};
\node (23210) at (0.0,5) {\tiny{$23210$}};
\node (323210) at (-0.5,6) {\tiny{$\*{323210}$}};
\node (123210) at (0.5,6) {\tiny{$\*{123210}$}};
\node (1323210) at (0.0,7) {\tiny{$\*{1323210}$}};
\node (0123210) at (1.0,7) {\tiny{$0123210$}};
\node (2323210) at (-1.0,7) {\tiny{$2323210$}};
\node (12323210) at (0.0,8) {\tiny{$12323210$}};
\node (01323210) at (1.0,8) {\tiny{$01323210$}};
\node (21323210) at (-1.0,8) {\tiny{$21323210$}};
\node (321323210) at (-1.5,9) {\tiny{$321323210$}};
\node (121323210) at (-0.5,9) {\tiny{$121323210$}};
\node (021323210) at (0.5,9) {\scriptsize{$\dots$}}; 
\path (0) edge (10) edge (e);
\path (210) edge (10);
\path (3210) edge (210);
\path (23210) edge (3210);
\path (323210) edge (23210);
\path (123210) edge (23210);
\path (1323210) edge (323210) edge (123210);
\path (0123210) edge (123210);
\path (2323210) edge (323210);
\path (12323210) edge (2323210) edge (1323210);
\path (01323210) edge (1323210) edge (0123210);
\path (21323210) edge (1323210) edge (2323210);
\path (321323210) edge (21323210);
\path (121323210) edge (12323210) edge (21323210);
\end{scope}
\end{tikzpicture}
\caption{Examples of Proposition~\ref{prop:N} cases (a) with $N=k-l+2$ (top left) and $N=1$ (top right), (b) with $N=3$ (middle left) and $N=1$ (middle right), and (c) with $N=5$ (bottom left) and $N=1$ (bottom right). In each example, $u$, $v$ and the elements of the sets $N_i(u,v)$ for $i\geq0$ are highlighted in bold.}\label{fig:N}
\end{figure}
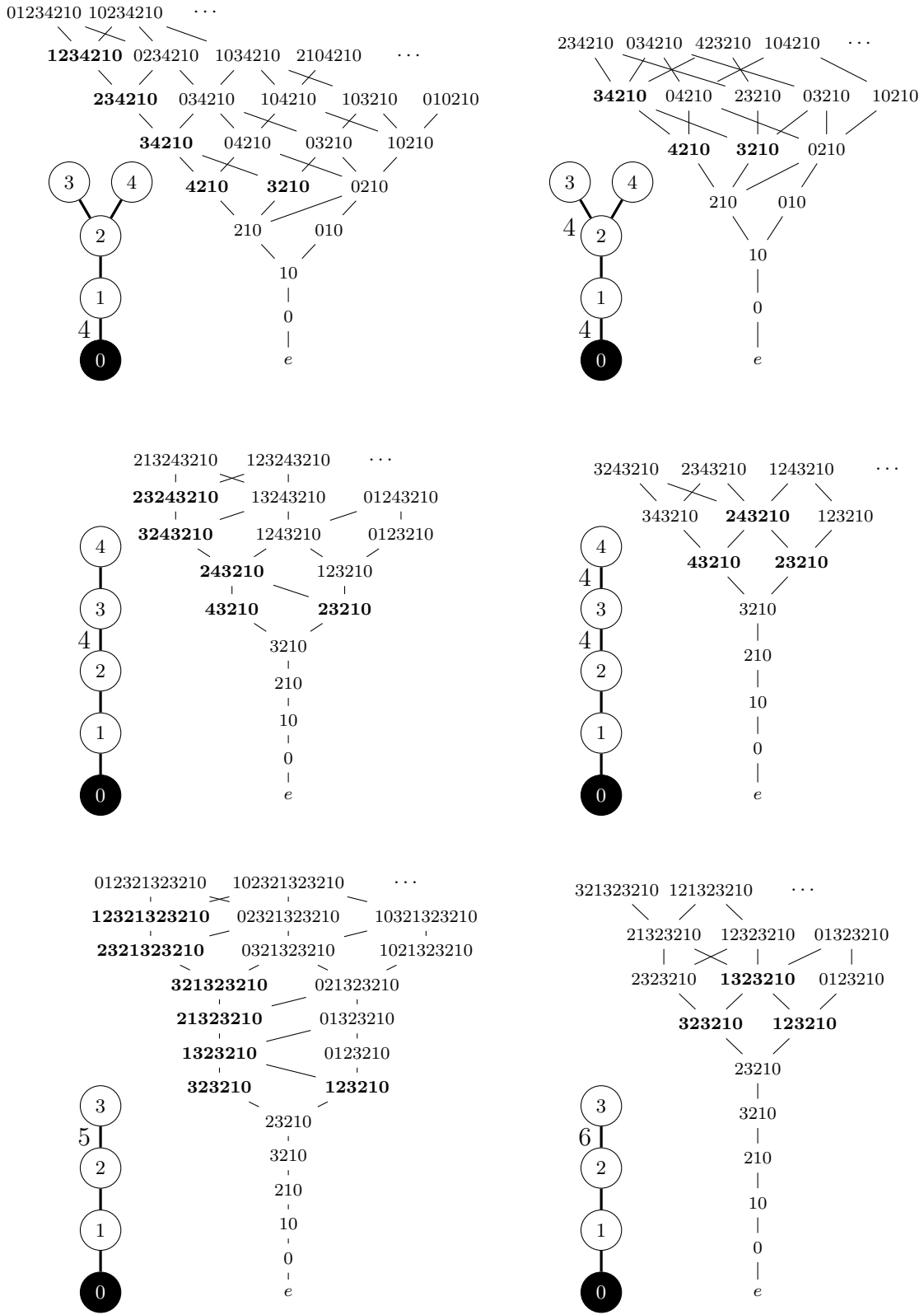

\begin{proof}[Proof of Proposition~\ref{prop:N}] We abbreviate Proposition~\ref{prop:unique} and Lemma~\ref{lem:semi} as $(*)$.

For (a), assume $m(s,s_k)=m(t,s_k)=3$. We have $N_0(u,v)=\{sts_k\dots s_0\}=\{tss_k\dots s_0\}$ and show by induction that $N_i(u,v)=\{s_{k-i+1}\dots s_ksts_k\dots s_0\}$ for $1\leq i\leq k-l+1$. If $w$ is semi-chainlike and covers $w_j=s_j\dots s_ksts_k\dots s_0$, then by $(*)$ the only possibility is $w=s_{j-1}\dots s_ksts_k\dots s_0$. However, if $m(s_j,s_{j-1})\geq4$ then $w$ also covers $s_{j-1}s_{j+1}\dots s_ksts_k\dots s_0\in W^J$ and $w$ is not semi-chainlike. This completes the induction and shows $N_{k-l+2}(u,v)$ is empty. If instead $m(s,s_k)\geq4$ then $w_k$ also covers $s_kss_k\dots s_0$, so $N_1(u,v)$ is empty and $N=1$, and the same holds if $m(t,s_k)\geq4$.

For (b), we have $N_0(u,v)=\{s_{k-1}ss_k\dots s_0\}=\{ss_{k-1}s_k\dots s_0\}$. Then by $(*)$, the only possible element of $N_1(u,v)$ is $s_ks_{k-1}ss_k\dots s_0$. This is not semi-chainlike and $N=1$ if and only if either $s_kss_k\dots s_0$ or $s_ks_{k-1}s_k\dots s_0$ are in $W^J$, which correspond to $m(s,s_k)\geq4$ and $m(s_k,s_{k-1})\geq5$ respectively. Otherwise, we have $N_2(u,v)=\{s_{k-1}s_ks_{k-1}ss_k\dots s_0\}$ and then $N_3(u,v)$ is empty as $s_{k-2}s_{k-1}s_ks_{k-1}ss_k\dots s_0$ covers $s_{k-2}s_ks_{k-1}ss_k\dots s_0\in W^J$ and is not semi-chainlike, so $N=3$.

For (c), we have $N_0(u,v)=\{s_ks_{k-2}s_{k-1}s_k\dots s_0\}=\{s_{k-2}s_ks_{k-1}s_k\dots s_0\}$. Then by $(*)$, the only possible element of $N_1(u,v)$ is $s_{k-1}s_ks_{k-2}s_{k-1}s_k\dots s_0$. This is not semi-chainlike and $N=1$ if and only if $s_{k-1}s_ks_{k-1}s_k\dots s_0\in W^J$, which corresponds to $m(s_k,s_{k-1})\geq 6$. Otherwise if $m(s_k,s_{k-1})=5$ we can continue the sequence, noting that $m(s_{k-1},s_{k-2})=3$ as $v$ is chainlike:
\begin{align*}
	N_1(u,v)&=\{s_{k-1}s_ks_{k-2}s_{k-1}s_k\dots s_0\}\\
	N_2(u,v)&=\{s_ks_{k-1}s_ks_{k-2}s_{k-1}s_k\dots s_0\}\\
	N_3(u,v)&=\{s_{k-1}s_ks_{k-1}s_ks_{k-2}s_{k-1}s_k\dots s_0\}\\
	N_4(u,v)&=\{s_{k-2}s_{k-1}s_ks_{k-1}s_ks_{k-2}s_{k-1}s_k\dots s_0\}
\end{align*}
Then $N_5(u,v)$ is empty since $s_{k-3}s_{k-2}s_{k-1}s_ks_{k-1}s_ks_{k-2}s_{k-1}s_k\dots s_0$ covers the element obtained by removing the leftmost $s_{k-2}$, so $N=5$.
\end{proof}

\section{Classifying isomorphic posets}\label{sec:sudoku}

Throughout this section we will denote by $(W,S)$ and $(\ol W,\ol S)$ two irreducible Coxeter systems with $J\subseteq S$ and $\ol J\subseteq\ol S$ such that there is a bijective map $W^J\to\ol W^{\ol J}$, $w\mapsto\ol w$ which is an isomorphism of the Bruhat poset, i.e. $u\leq v$ if and only if $\ol u\leq\ol v$. Since our definitions of chainlike elements, detectors, the $L$, $M$ and $N$ functions and the relation $\sim$ rely only on the structure of the poset, these are identical for $W^J$ and $\ol W^{\ol J}$ (in particular, $\overline{(u')}=\ol u'$ for all semi-chainlike $u\in W^J$). If we can show that for all $w\in C(W^J)$ the form of $w$ is the same as the form of $\ol w$, then Construction~\ref{thm:graph} implies the graphs of $(W,S)$ and $(\ol W,\ol S)$ are the same. Consequently, we aim to find all scenarios in which $w$ does \textit{not} have the same form as $\ol w$. To simplify the notation, we will use $s$ to denote generators in $S$ and $t$ to denote generators in $\ol S$, and we write $C=C(W^J)$, $\ol C=C(\ol W^{\ol J})$. 

\subsection{The basket case}

In addition to the known finite isomorphisms, there is one other collection of cases where $w\in C$ and $\ol w\in\ol C$ have different forms. This isn't an isomorphism between distinct posets, but rather an automorphism with $W^J\cong\ol W^{\ol J}$. Figure~\ref{fig:small_auto} shows the smallest case of this automorphism.

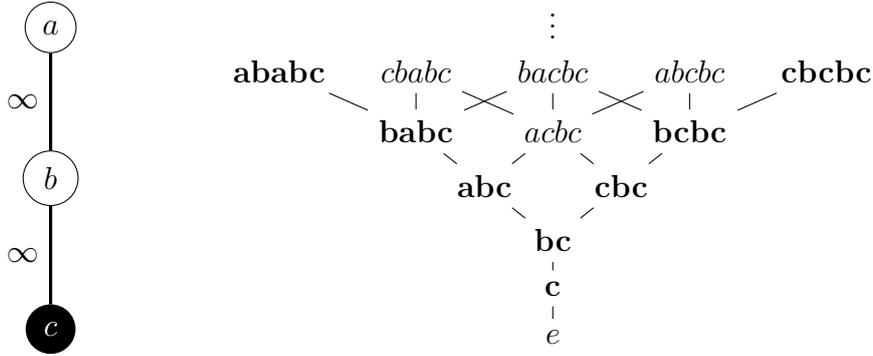
\begin{figure}[h]
\begin{center}
\begin{tikzpicture}[scale=1]
\begin{scope}[every node/.style={circle,draw}]
    \node (A) at (0,2) {$a$};
    \node (B) at (0,0) {$b$};
    \node [fill=black, text=white] (C) at (0,-2) {$c$};
\end{scope}
\begin{scope}[every edge/.style={draw=black,very thick}]
    \path (A) edge node[left] {$\infty$} (B);
    \path (B) edge node[left] {$\infty$} (C);
\end{scope}
\end{tikzpicture}\hspace{50pt}
\begin{tikzpicture}[xscale=0.9, yscale=0.8]
\begin{scope}
	\node (e) at (0,-0.4) {$e$};
	\node (c) at (0,0.4) {$\*c$};
	\node (bc) at (0,1.2) {$\*{bc}$};
	\node (abc) at (-1,2.1) {$\*{abc}$};
	\node (cbc) at (1,2.1) {$\*{cbc}$};
	\node (babc) at (-2,3) {$\*{babc}$};
	\node (acbc) at (0,3) {$acbc$};
	\node (bcbc) at (2,3) {$\*{bcbc}$};
	\node (ababc) at (-4,4) {$\*{ababc}$};
	\node (cbabc) at (-2,4) {$cbabc$};
	\node[label=above:{$\vdots$}] (bacbc) at (0,4) {$bacbc$};
	\node (abcbc) at (2,4) {$abcbc$};
	\node (cbcbc) at (4,4) {$\*{cbcbc}$};
\end{scope}
\begin{scope}[every edge/.style={draw=black}]
	\path (e) edge (c);
	\path (c) edge (bc);
	\path (bc) edge (abc) (bc) edge (cbc);
	\path (abc) edge (babc) (abc) edge (acbc) (cbc) edge (acbc) (cbc) edge (bcbc);
	\path (babc) edge (ababc) (babc) edge (cbabc) (babc) edge (bacbc);
	\path (acbc) edge (cbabc) (acbc) edge (bacbc) (acbc) edge (abcbc);
	\path (bcbc) edge (abcbc) (bcbc) edge (cbcbc) (bcbc) edge (bacbc);
\end{scope}
\end{tikzpicture}
\end{center}
\caption{The depicted poset has an automorphism $\phi$ which swaps all occurrences of $a$ with $c$ and vice versa, excluding the last $c$ in each reduced expression, for example $\phi(cbabc)=abcbc$. Chainlike elements are indicated in bold.}\label{fig:small_auto}
\end{figure}

Fortunately, the graphs that produce posets with this automorphism have a special structure that is very restricted, and they all contain a feature in the poset which we'll call a basket.

\begin{definition}
We say $u,v\in C(W^J)$ form a \textit{basket} if $u''=v''$, $u'$ detects $v$ and $v'$ detects $u$. It follows from Remark~\ref{rem:detectors} that a basket can occur in 3 ways:
\begin{itemize}
	\item I/I: $u$ and $v$ are simple and share the same first generator (square of nodes).
	\item II/II: $u$ and $v$ are of form (II) and $u',v'$ are simple (two labelled branches).
	\item III/II: $u$ is of form (III), $v,u'$ are of form (II) and $v'$ is simple (a label $\geq4$ above a label $\geq5$). If the forms of $u$ and $v$ are the other way round, we write II/III.
\end{itemize}
These are illustrated in Figure~\ref{fig:baskets}. If the forms of $u,u',v,v'$ are the same as the forms of $\ol u,\ol u',\ol v,\ol v'$ respectively, we say the basket is \textit{form-preserving}. For example, in Figure~\ref{fig:small_auto} $babc$ and $bcbc$ form a II/III basket, and with the given automorphism this basket is non-form-preserving as it is mapped to a III/II basket.
\end{definition}

\begin{figure}[h]\centering
\begin{tikzpicture}[scale=1]
\node at (0,4.2) {I/I};
\begin{scope}
	\node[blacknode] (a0) at (0,0) {$0$};
	\node[whitenode] (a1) at (0,1) {$1$};
	\node[whitenode] (a2) at (0,2) {$2$};
	\node[whitenode] (a3) at (-0.7,2.7) {$3$};
	\node[whitenode] (a4) at (0.7,2.7) {$4$};
	\node[whitenode] (a5) at (0,3.4) {$5$};
\end{scope}
\begin{scope}[every edge/.style=graphedge]
    \path (a1) edge (a0) edge (a2);
    \path (a3) edge (a2) edge (a5);
    \path (a4) edge (a2) edge (a5);
\end{scope}
\begin{scope}[shift={(0,-5)}]
	\node (a0) at (0,0.4) {$0$};
	\node (a10) at (0,1.2) {$10$};
	\node (a210) at (0,2) {$210$};
	\node (a3210) at (-0.8,2.7) {$3210$};
	\node (a4210) at (0.8,2.7) {$4210$};
	\node (a53210) at (-0.8,3.7) {$53210$};
	\node (a54210) at (0.8,3.7) {$54210$};
	\node at (0,3.7) {$\sim$};
\end{scope}
\begin{scope}[every edge/.style={draw=black}]
	\path (a10) edge (a0) edge (a210);
	\path (a3210) edge (a53210) edge (a210);
	\path (a4210) edge (a54210) edge (a210);
	\path[->]  (a3210) edge (a54210);
	\path[->]  (a4210) edge (a53210);
\end{scope}

\node at (5,4.2) {II/II};
\begin{scope}[shift={(5,0)}]
	\node[blacknode] (b0) at (0,0) {$0$};
	\node[whitenode] (b1) at (0,1) {$1$};
	\node[whitenode] (b2) at (0,2) {$2$};
	\node[whitenode] (b3) at (-0.9,2.9) {$3$};
	\node[whitenode] (b4) at (0.9,2.9) {$4$};
\end{scope}
\begin{scope}[every edge/.style=graphedge]
    \path (b1) edge (b0) edge (b2);
    \path (b3) edge node[below left] {$\geq4$} (b2);
    \path (b4) edge node[below right] {$\geq4$} (b2);
\end{scope}
\begin{scope}[shift={(5,-5)}]
	\node (b0) at (0,0.4) {$0$};
	\node (b10) at (0,1.2) {$10$};
	\node (b210) at (0,2) {$210$};
	\node (b3210) at (-0.8,2.7) {$3210$};
	\node (b4210) at (0.8,2.7) {$4210$};
	\node (b23210) at (-0.8,3.7) {$23210$};
	\node (b24210) at (0.8,3.7) {$24210$};
	\node at (0,3.7) {$\sim$};
\end{scope}
\begin{scope}[every edge/.style={draw=black}]
	\path (b10) edge (b0) edge (b210);
	\path (b3210) edge (b23210) edge (b210);
	\path (b4210) edge (b24210) edge (b210);
	\path[->]  (b3210) edge (b24210);
	\path[->]  (b4210) edge (b23210);
\end{scope}

\node at (10,4.2) {III/II};
\begin{scope}[shift={(10,0)}]
	\node[blacknode] (b0) at (0,0) {$0$};
	\node[whitenode] (b1) at (0,1) {$1$};
	\node[whitenode] (b2) at (0,2.15) {$2$};
	\node[whitenode] (b3) at (0,3.3) {$3$};
\end{scope}
\begin{scope}[every edge/.style=graphedge]
    \path (b1) edge (b0) edge node[left] {$\geq5$} (b2);
    \path (b3) edge node[left] {$\geq4$} (b2);
\end{scope}
\begin{scope}[shift={(10,-5)}]
	\node (c0) at (0,0.4) {$0$};
	\node (c10) at (0,1.2) {$10$};
	\node (c210) at (0,2) {$210$};
	\node (c3210) at (0.8,2.7) {$3210$};
	\node (c1210) at (-0.8,2.7) {$1210$};
	\node (c23210) at (0.8,3.7) {$23210$};
	\node (c21210) at (-0.8,3.7) {$21210$};
	\node at (0,3.7) {$\sim$};
\end{scope}
\begin{scope}[every edge/.style={draw=black}]
	\path (c10) edge (c0) edge (c210);
	\path (c3210) edge (c23210) edge (c210);
	\path (c1210) edge (c21210) edge (c210);
	\path[->]  (c3210) edge (c21210);
	\path[->]  (c1210) edge (c23210);
\end{scope}
\end{tikzpicture}
\caption{Examples of the 3 different ways a basket can occur. Arrows denote detections in the chainlike graphs (which look a bit like a weaving, hence the name `basket').}\label{fig:baskets}
\end{figure}
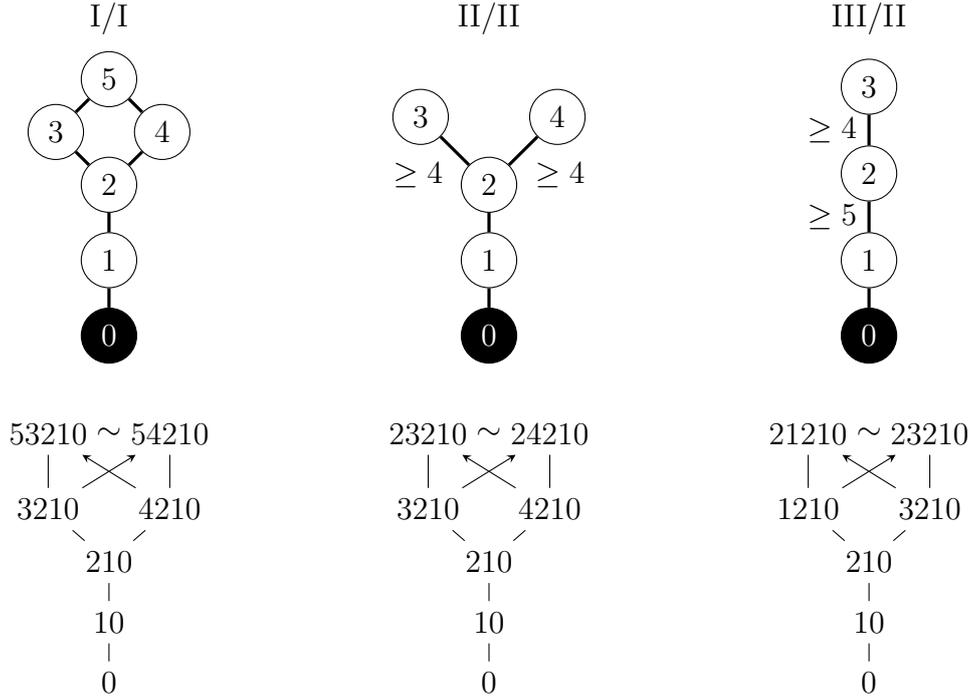

We will show that whenever $W^J$ contains a non-form-preserving basket, there is an automorphism similar to the one in Figure~\ref{fig:small_auto} which `corrects' the forms:

\begin{theorem}\label{thm:basket_complete}
If $W^J$ contains a non-form-preserving basket, then $\phi(w)$ has the same form as $\ol w$ for all $w\in C$ where $\phi$ is the automorphism in Proposition~\ref{prop:automorphism}, and so by Theorem~\ref{thm:graph} the pairs $(W,W_J)$ and $(\ol W,\ol W_{\ol J})$ are isomorphic.
\end{theorem}

We will see that this covers almost all automorphisms which do not preserve chainlike forms. In particular, baskets appear as edge cases in many situations below and so proving this theorem covers all of these edge cases in advance. Here is a breakdown of the steps to proving Theorem~\ref{thm:basket_complete}:
\begin{itemize}
	\item Lemma~\ref{lem:A} is a technical result that will help in several places.
	\item Propositions~\ref{prop:basket_structure} and \ref{prop:basket_constraints} look at the case where a basket has form III/II in $W^J$ and II/III in $\ol W^{\ol J}$, showing the graph has a specific structure depicted in Figure~\ref{fig:basket_case}.
	\item Proposition~\ref{prop:automorphism} shows that the automorphism in Figure~\ref{fig:small_auto} above extends to any Coxeter graph with this structure.
	\item Proposition~\ref{prop:basket_case} shows all other non-form-preserving baskets reduce to this case.
	\item Proposition~\ref{prop:form2} is a powerful result for form (II) elements. A lot of the case-checking following this section is encapsulated in the proof of this statement.
	\item Corollary~\ref{cor:double_label} immediately follows and shows the graphs of $W^J$ and $\ol W^{\ol J}$ are identical up to relabelling nodes, completing the proof of Theorem~\ref{thm:basket_complete}.
\end{itemize}

\begin{lemma}\label{lem:A}
Suppose $u=s_l\dots s_k\dots s_0$ is detected by $x\vartriangleright s_l\dots s_0$, and $\ol u$ is simple. Then $l=k-1$, and $v=s_{k-1}x\in C$ so that $u,v$ form a basket.
\end{lemma}

\begin{proof}
With $\ol u=t_{2k-l}\dots t_0$, we require $\ol x=tt_l\dots t_0$ simple and there is a chainlike element $\ol v=t_{2k-l}tt_l\dots t_0\in\ol C$ with $\ol u\sim\ol v$. Thus the corresponding element of $C$ is $v=s_lx$. If $x=s_{k-2}s_{k-1}\dots s_0$, then $v$ is only chainlike if $m(s_{k-1},s_{k-2})\geq5$ and $u,v$ form a II/III basket. If instead $x=ss_l\dots s_0$ is simple, then $v$ is only chainlike if $m(s,s_l)\geq4$. Then $s_{l+1}\dots s_0$ detects $v$ and so $t_{l+1}\dots t_0$ detects $\ol v$, thus $m(t_{2k-l},t_{l+1})\geq3$ which is only possible if $l=k-1$ in which case $u,v$ form a II/II basket.
\end{proof}

\begin{prop}\label{prop:basket_structure}
Suppose $u,v$ form a III/II basket while $\ol u,\ol v$ form a II/III basket, that is $u,v,\ol u,\ol v$ have reduced expressions as follows:
\begin{align*}
	u&=s_ks_{k-1}s_k\dots s_0 & v&=s_ks_{k+1}s_k\dots s_0\\
	\ol u&=t_kt_{k+1}s_k\dots t_0 & \ol v&=t_kt_{k-1}t_k\dots t_0
\end{align*}
Then $m(s_k,s_{k-1})=m(s_k,s_{k+1})=m(t_k,t_{k-1})=m(t_k,t_{k+1})=\infty$. Furthermore if $k\geq2$ then there exist more generators $s_{k+2},\dots,s_{2k}\in S$ and $t_{k+2},\dots,t_{2k}\in\ol S$ so that
\[\overline{s_l\dots s_k\dots s_0}=t_{2k-l}\dots t_0,\qquad\overline{s_{2k-l}\dots s_0}=t_l\dots t_k\dots t_0\]
for $0\leq l\leq k$.
\end{prop}

\begin{proof}
We cannot have $s_{k-1}v\in C$ since $m(s_k,s_{k-1})\geq5$, so all chainlikes above $u,v$ are of form (III). If $m(s_k,s_{k-1})$ is finite then the maximal chainlikes above $u,\ol u$ are
\[u_{\max}=\underbrace{\dots s_{k-1}s_ks_{k-1}s_ks_{k-1}}_{m(s_k,s_{k-1})-1\text{ terms}}s_{k-2}\dots s_0\qquad\qquad
\ol u_{\max}=\underbrace{\dots t_kt_{k+1}t_kt_{k+1}t_k}_{m(t_k,t_{k+1})-1\text{ terms}}t_{k-1}\dots t_0\]
Comparing lengths we have $m(t_k,t_{k+1})=m(s_k,s_{k-1})-1$, but this is a contradiction since $M(v,u')=m(s_k,s_{k-1})$ while $M(\ol v,\ol u')=m(t_k,t_{k+1})$. Thus we must have $m(s_k,s_{k-1})=m(t_k,t_{k+1})=\infty$, and $m(s_k,s_{k+1})=m(t_k,t_{k-1})=\infty$ as well by repeating with $\ol v$ and $v$.

Now, if $k\geq2$ then we cannot have $m(s_{k-1},s_{k-2})\geq4$, as $x=s_{k-2}s_{k-1}\dots s_0$ would not detect $v'$ but $\ol x\vartriangleright t_{k-1}\dots t_0$ would detect $\ol v'$. Then $s_{k-2}s_{k-1}s_k\dots s_0$ is chainlike and maps to $t_{k+2}t_{k+1}t_k\dots t_0$ for some $t_{k+2}\in\ol S$, and similarly $t_{k-2}t_{k-1}t_k\dots t_0$ is chainlike and maps to $s_{k+2}\dots s_0$ for some $s_{k+2}\in S$. Then inductively $m(s_i,s_{i-1})=m(t_i,t_{i-1})=3$ for $1\leq i\leq k-1$, since otherwise $x_i=s_{i-1}s_i\dots s_0$ does not detect $w_i=s_{k+i}\dots s_0$ but $\ol x_i\vartriangleright t_i\dots t_0$ detects $\ol w_i=t_i\dots t_k\dots t_0$, and similarly for $t_{i-1}t_i\dots t_0$. Thus we have $s_l\dots s_k\dots s_0$ maps to $t_{2k-l}\dots t_0$ and $s_{2k-l}\dots s_0$ maps to $t_l\dots t_k\dots t_0$ for $0\leq l\leq k$.
\end{proof}

\begin{prop}\label{prop:basket_constraints}
Suppose $u,v$ form a III/II basket while $\ol u,\ol v$ form a II/III basket, and let $s_0,\dots,s_{2k}$ be as in Proposition~\ref{prop:basket_structure}. For all $s\in S\setminus\{s_0,\dots,s_{2k}\}$ we have:
\begin{enumerate}[label=(\alph*)]
	\item $s\in J$;
	\item $m(s,s_i)=2$ for all $s_i\in\{s_0,\dots,s_{k-2},s_{k+2},\dots,s_{2k}\}$;
	\item If $m(s,s_{k-1})\geq3$ or $m(s,s_{k+1})\geq3$ then $m(s,s_{k-1})=m(s,s_{k+1})=\infty$.
\end{enumerate}
That is, the structure of the Coxeter graph is as given in Figure~\ref{fig:basket_case}.
\end{prop}

\begin{figure}[h]
\begin{center}
\begin{tikzpicture}[yscale=0.85]
\begin{scope}[every node/.style={minimum size=0.5cm}]
	\node[circle,draw,fill=black] (0) at (-6.5,0) {}; \node at (-6.5,-0.5) {$s_0$};
	\node[circle,draw] (1) at (-5,0) {}; \node at (-5,-0.5) {$s_1$};
	\node (d-) at (-4,0) {$\cdots$};
	\node[circle,draw] (k-2) at (-3,0) {}; \node at (-3,-0.5) {$s_{k-2}$};
	\node[circle,draw] (k-1) at (-1.5,0) {}; \node at (-1.5,-0.5) {$s_{k-1}$};
	\node[circle,draw] (k) at (0,0) {}; \node at (0,-0.5) {$s_k$};
	\node[circle,draw] (k+1) at (1.5,0) {}; \node at (1.5,-0.5) {$s_{k+1}$};
	\node[circle,draw] (k+2) at (3,0) {}; \node at (3,-0.5) {$s_{k+2}$};
	\node (d+) at (4,0) {$\cdots$};
	\node[circle,draw] (2k-1) at (5,0) {}; \node at (5,-0.5) {$s_{2k-1}$};
	\node[circle,draw] (2k) at (6.5,0) {}; \node at (6.5,-0.5) {$s_{2k}$};
	\node[circle,draw] (s1) at (0,1.5) {};
	\node[circle,draw] (s2) at (0,3) {};
	\node (du) at (0,4) {$\vdots$};
	\node[circle,draw,align=center] (other) at (4,2.5) {any connected\\graph of\\white nodes};
\end{scope}
\begin{scope}[every edge/.style=graphedge]
	\path (1) edge (0) edge (d-);
	\path (k-2) edge (d-) edge (k-1);
	\path (k) edge node[above] {$\infty$} (k-1) edge node[above] {$\infty$} (k+1);
	\path (k+2) edge (d+) edge (k+1);
	\path (2k-1) edge (2k) edge (d+);
	\path[dotted] (s1) edge (k) edge (s2);
	\path[dotted] (s2) edge (du);
	\path[bend right=20] (s1) edge node[pos=0.2, above left] {$\infty$} (k-1);
	\path[bend left=20] (s1) edge node[pos=0.2, above right] {$\infty$} (k+1);
	\path[bend right=40] (s2) edge node[pos=0.25, above left] {$\infty$} (k-1);
	\path[bend left=40] (s2) edge node[pos=0.25, above right] {$\infty$} (k+1);
	\path[dotted] (k) edge[bend left=10] (other);
	\path[dotted] (s1) edge[bend left=10] (other);
	\path[dotted] (s2) edge[bend left=10] (other);
\end{scope}
\end{tikzpicture}
\end{center}
\caption{The `basket case'. Any number of nodes with $\infty$-labelled edges to both $s_{k-1}$ and $s_{k+1}$ are permitted, and these can have edges to each other, $s_k$, and other white nodes represented by `any graph of white nodes'. The nodes $s_0,\dots,s_{k-2},s_{k+2},\dots,s_{2k}$ have no other edges.}\label{fig:basket_case}
\end{figure}
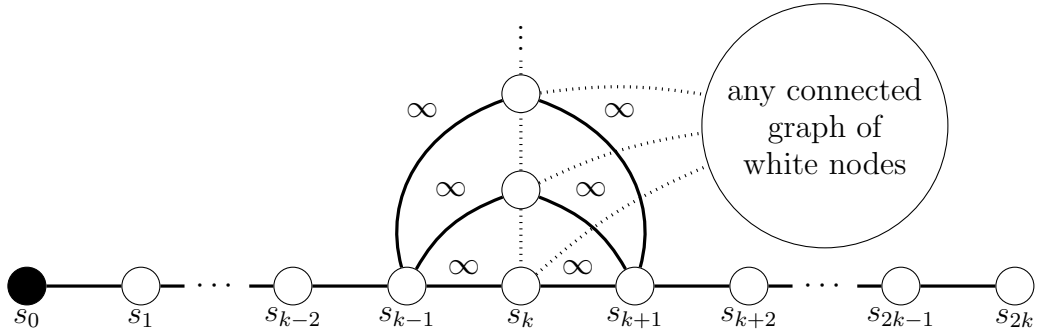

\begin{proof}
Let $s\in S\setminus\{s_0,\dots,s_{2k}\}$.
\begin{itemize}
\item Suppose $s\in J$ and $s$ has an edge to $s_l$ for some $0\leq l<k$. Assuming $l$ is minimal with this property, $x=ss_l\dots s_0$ is chainlike and detects $s_l\dots s_k\dots s_0$, so by Lemma~\ref{lem:A} we have $l=k-1$ and $u',s_{k-1}x$ form a basket. We must have $x$ simple (since $m(s_{k-1},s_{k-2})=3$ for $k\geq2$) so this is a II/II basket and $m(s,s_{k-1})\geq4$. Similarly $\ol x=tt_{k-1}\dots t_0$ is simple, and by Lemma~\ref{lem:A} again since $\ol x$ detects $\ol v'$ we have $m(t,t_{k-1})\geq4$. Hence,
\begin{align*}
	m(t,t_{k+1})=M(\ol u',\ol x)&=M(u',x)=m(s,s_{k-1})\geq4,\\
	m(s,s_{k+1})=M(v',x)&=M(\ol v',\ol x)=m(t,t_{k-1})\geq4.
\end{align*}
Then $v_1=ss_{k+1}x$ is chainlike and is detected by $y=s_{k-1}x$. Since $y\vartriangleright x$ and $y\sim u'$ we must have $\ol y=t_{k+1}\ol x$, and then since $v_1'\sim v'$ and $y$ detects $v$ we must have $\ol v_1=tt_{k-1}\ol x$. Symmetrically, $\ol u_1=tt_{k+1}x$ is chainlike with $u_1=ss_{k-1}x$. Then $u_1$ and $v_1$ form a III/II basket while $\ol u_1,\ol v_1$ are II/III, and so by Proposition~\ref{prop:basket_structure} $m(s,s_{k-1})=m(s,s_{k+1})=m(t,t_{k-1})=m(t,t_{k+1})=\infty$.
\item Suppose $s\in J$ has an edge to $s_{2k-l}$ for some $0\leq l<k$. The element $ss_{2k-l}\dots s_0$ cannot be chainlike since there is no other element covering $t_l\dots t_k\dots t_0$ it can map to (note that $t_kt_{k-1}t_k\dots t_0$ is already mapped to by $s_ks_{k+1}\dots s_0$), so $m(s,s_i)\geq3$ for some $i\leq k$. Then with $x=ss_i\dots s_0$ and $\ol x=tt_i\dots t_0$, we can see by the $M$ functions that $m(t,t_l)\geq3$ and this reduces to the above case.
\item Suppose $s\not\in J$ has an edge to $s_l$ for some $0\leq l\leq k$. Then $s$ detects $s_l\dots s_k\dots s_0$, so $t=\ol s$ detects $t_{2k-l}\dots t_0$ and there is an edge between $t_{2k-l}$ and $t$. Thus $t_{2k}\dots t_{2k-l}t$ is chainlike and $t_{2k}\dots t_{2k-l}t\sim t_{2k}\dots t_0$, but $t_{2k}\dots t_0$ corresponds to $s_0\dots s_k\dots s_0$ which cannot satisfy $\sim$ with any chainlike ending in $s$, so we have a contradiction. Similarly, if there is $s\not\in J$ connected by a chain of generators in $J$ to $s_k$ then we can find a simple $w\geq s$ so that $s_kw\sim s_k\dots s_0$ is simple, and so $t_k\ol w\sim t_k\dots t_0$ is also simple and the same argument using $t_{2k}$ holds. If $s$ connects to $s_{2k-l}$ instead, then the same argument applies to $t$ connecting to $t_l$.
\end{itemize}
This shows the Coxeter graph is as depicted in Figure~\ref{fig:basket_case}.
\end{proof}

\begin{prop}\label{prop:automorphism}
With $s_0,\dots,s_{2k}$ as above, the map $\phi:W^J\to W^J$ defined as follows is an automorphism:
\begin{itemize}
\item If $w\not>u''$, then $\phi(w)=w$.
\item If $w>u''$, choose any reduced expression for $w$ (which necessarily ends in the reduced expression for $u''$) and obtain $\phi(w)$ by replacing all occurrences of $s_{k+j}$ left of the rightmost $k$ generators in $w$ with $s_{k-j}$ for $-k\leq j\leq k$.
\end{itemize}
\end{prop}

\begin{proof}
Let $w\in W^J$. Since $S\setminus J=\{s_0\}$, either $w=s_l\dots s_0$ for some $l\leq k$, or every reduced expression for $w$ ends in the reduced expression for $u''$. Write $\phi_S(s_i)=s_{2k-i}$ and $\phi_S(s)=s$ for all $s\in S\setminus\{s_0,\dots,s_{2k}\}$. By Proposition~\ref{prop:basket_constraints} we have $m(\phi_S(a),\phi_S(b))=m(a,b)$ for all $a,b\in S$, hence any reduced expression for $w$ and its corresponding expression for $\phi(w)$ are altered by braid-moves in the same locations. Thus $\phi$ preserves Bruhat order, and so $\phi$ is an automorphism.
\end{proof}

\begin{prop}\label{prop:basket_case}
Suppose $u,v$ form a basket which is non-form-preserving under $W^J\to\ol W^{\ol J}$. Then $W^J$ has the structure depicted in Figure~\ref{fig:basket_case} and $\phi(w)$ has the same form as $\ol w$ for $w\in\{u,u',v,v'\}$.
\end{prop}

\begin{figure}[h]
\centering
\begin{tikzpicture}[xscale=2.2, yscale=2.1]
	\draw (0,0) node {$u''$}; \draw[thick] (0,-0.2) -- (0,-0.8) node[pos=0.5, left] {\small $s_k,t_k$};
	\draw (0,1) node {$u'$}; \draw[thick] (0,0.2) -- (0,0.8) node[pos=0.9, left] {\small $s_{k-1},t$};
	\draw (0,2) node {$u$}; \draw[thick] (0,1.2) -- (0,1.8) node[pos=0.5, left] {\small $s_k,t_k$}; 
	\draw (1,1) node {$v'$}; \draw[thick] (0.2,0.2) -- (0.8,0.8) node[pos=0.5, below right] {\small $s_{k+1},t_{k+1}$}; 
	\draw (1,2) node {$v$}; \draw[thick] (1,1.2) -- (1,1.8) node[pos=0.5, right] {\small $s_k,t_k$};
	\draw (-1,1) node {$x$}; \draw[thick, dashed] (-0.2,0.2) -- (-0.8,0.8) node[pos=0.5, below left] {\small $s,t_{k-1}$};
	\draw (-1,2) node {$v_1$}; \draw[thick, dashed] (-1,1.2) -- (-1,1.8) node[pos=0.5, left] {\small $s_k,t_k$};
\end{tikzpicture}\hfill
\begin{tikzpicture}[xscale=2.2, yscale=1.6]
	\draw (0,0) node {$u''$}; \draw[thick] (0,-0.2) -- (0,-0.8) node[pos=0.5, left] {\small $s_{k-1},t_{k-1}$};
	\draw (0,1) node {$u'$}; \draw[thick] (0,0.2) -- (0,0.8) node[pos=0.5, left] {\small $s_k,t_k$};
	\draw (0,2) node {$u$}; \draw[thick] (0,1.2) -- (0,1.8) node[pos=0.6, right] {\small $s_{k-1},t_{k+1}$}; 
	\draw (1,1) node {$v'$}; \draw[thick] (0.2,0.2) -- (0.8,0.8) node[pos=0.5, below right] {\small $s,t$}; 
	\draw (1,2) node {$v$}; \draw[thick] (1,1.2) -- (1,1.8) node[pos=0.5, right] {\small $s_{k-1},t_{k+1}$};
	\draw (-1,2) node {$x$}; \draw[thick, dashed] (-0.2,1.2) -- (-0.8,1.8) node[pos=0.5, below left] {\small $s_{k+1},t_{k-1}$};
	\draw (-1,3) node {$v_1$}; \draw[thick, dashed] (-1,2.2) -- (-1,2.8) node[pos=0.5, left] {\small $s_k,t_k$};
	\draw (0,3) node {$u_1$}; \draw[thick, dashed] (0,2.2) -- (0,2.8) node[pos=0.5, right] {\small $s_k,t_k$};
\end{tikzpicture}\hfill
\caption{Illustration of parts (a) (left) and (c) (right) of the proof of Proposition~\ref{prop:basket_case}. In (a), $u$ and $v$ imply the existence of $v_1$ so that $u$ and $v_1$ form a III/II-II/III basket. In (b), $u$ and $v$ imply the existence of $u_1$ and $v_1$ which form a III/II-II/III basket.}\label{fig:combos}
\end{figure}
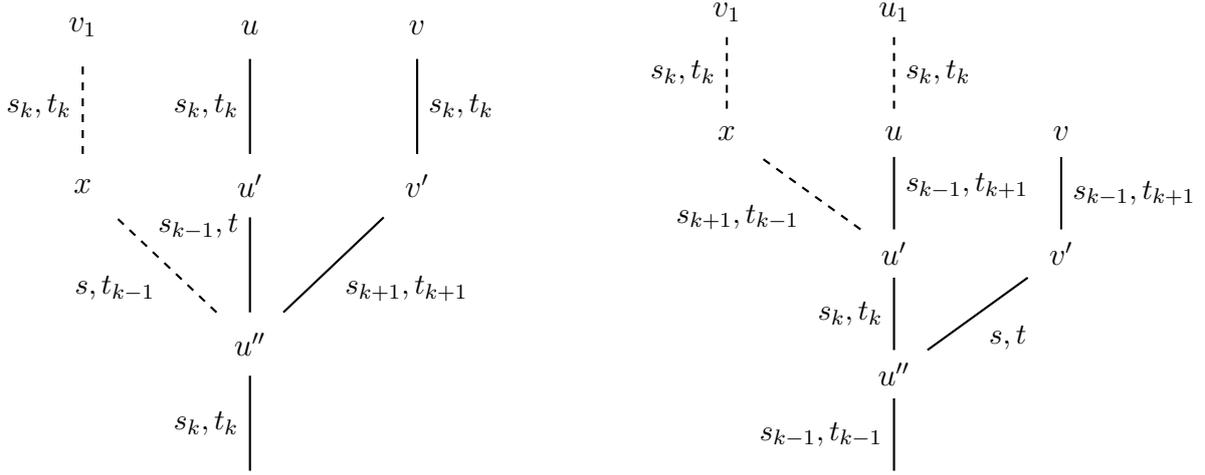

\begin{proof}
We consider all possible combinations of baskets $u,v,\ol u,\ol v$ and for each show that there are $u_1,v_1$ that form a III/II-II/III basket. Two of these are illustrated in Figure~\ref{fig:combos} for better visualisation.
\begin{enumerate}[label=(\alph*)]
\item Suppose $u,v$ form a III/II basket and $\ol u,\ol v$ form a II/II basket. We have
\[\qquad\quad u=s_ks_{k-1}s_k\dots s_0\qquad v=s_ks_{k+1}s_k\dots s_0\qquad \ol u=t_ktt_k\dots t_0\qquad\ol v=t_kt_{k+1}t_k\dots t_0\]
with $m(t_k,t)\geq4$ and $m(t_k,t_{k+1})\geq4$. If $m(t_k,t_{k-1})=3$ then $t_{k-1}\ol u\sim t_{k-1}\ol v\in\ol C$, but the only possible corresponding elements in $C$ are $s_{k-1}u\not\sim s_{k+1}v$, a contradiction, so $m(t_k,t_{k-1})\geq4$. This means $\ol x=t_{k-1}t_k\dots t_0\in\ol C$ must correspond to some simple $x=ss_k\dots s_0$. We then have
\[m(s,s_k)=M(v,x)=M(\ol v',\ol x)=m(t,t_k)\geq4\]
This means $v_1=s_kss_k\dots s_0$ is chainlike and $v_1\sim u$, hence we have $\ol v_1=t_kt_{k-1}t_k\dots t_0$. Then $u,v_1$ form a basket with $u$ and $\ol v_1$ of form (III). We have a III/II-II/III basket and can use the previous propositions.

\item Suppose $u,v$ form a III/II basket and $\ol u,\ol v$ form a II/II basket. We have
\[\qquad\quad u=s_ks_{k-1}s_k\dots s_0\qquad v=s_ks_{k+1}s_k\dots s_0\qquad \ol u=t_{k+2}tt_k\dots t_0\qquad\ol v=t_{k+2}\dots t_0\]
We have $N(u',v')=1$, so by Proposition~\ref{prop:N} in order to have the same result for $\ol u',\ol v'$ we require either $m(t_{k+1},t_k)\geq4$ or $m(t,t_k)\geq4$.
\begin{itemize}
	\item If $m(t_{k+1},t_k)\geq4$, then $\ol x=t_kt_{k+1}\dots t_0\vartriangleright\ol v'$ corresponds to $x=ss_{k+1}\dots s_0$ for some $s\in S$. Then $u'$ does not detect $x$ but $\ol u'$ detects $\ol x$, a contradiction.
	\item If $m(t,t_k)\geq4$, then $\ol x=t_ktt_k\dots t_0\vartriangleright\ol u'$ corresponds to $x=s_{k-2}s_{k-1}s_k\dots s_0$. Then $v'$ does not detect $x$ but $\ol v'$ detects $\ol x$, a contradiction.
\end{itemize}
Thus, a III/II-II/II basket is not possible.

\item Suppose $u,v$ form a II/II basket and $\ol u,\ol v$ form a I/I basket. We have
\[\qquad\quad\  u=s_{k-1}s_k\dots s_0\qquad v=s_{k-1}ss_{k-1}\dots s_0\qquad \ol u=t_{k+1}\dots t_0\qquad\ol v=t_{k+1}tt_{k-1}\dots t_0\]
We have $N(u',v')=1$, so to have the same result for $\ol u',\ol v'$ we require either $m(t,t_{k-1})\geq4$ or $m(t_k,t_{k-1})\geq4$. Assume the latter without loss of generality. Then $\ol x=t_{k-1}t_k\dots t_0\vartriangleright \ol u'$ corresponds to $x=s_{k+1}\dots s_0$ for some $s_{k+1}\in S$. Now,
\[m(t_k,t_{k+1})=M(\ol v,\ol u')=M(v,u')=m(s_k,s_{k-1})\geq4\]
Then $\ol u_1=t_kt_{k+1}\dots t_0\vartriangleright\ol u$ is chainlike. Taking $u_1=s_{k-2}s_{k-1}s_k\dots s_0$ gives a contradiction, since $w=s_{k-2}s_{k-1}ss_{k-1}\dots s_0$ would be chainlike with $w\sim u_1$ meaning that $w$ must map to $\ol w=t_kt_{k+1}tt_{k-1}\dots t_0$ which is not chainlike, so the only possibility is $u_1=s_ks_{k-1}s_k\dots s_0$ and $m(s_k,s_{k-1})\geq5$. Next, we have
\[m(s_k,s_{k+1})=M(u_1,x)=M(\ol u_1,\ol x)=m(t_k,t_{k-1})\geq4\]
This means $v_1=s_ks_{k+1}\dots s_0$ is chainlike and corresponds to $\ol v_1=t_kt_{k-1}t_k\dots t_0$, and $u_1,v_1$ form a III/II-II/III basket.
\end{enumerate}
\end{proof}

\begin{prop}\label{prop:form2}
Suppose $u=s_l\dots s_k\dots s_0$ is of form (II) and is detected by $x\vartriangleright s_l\dots s_0$ such that $x$ does not detect $w$ for $w<u$. If $u$ does not form a basket with any $v\in C$, then the form of $w$ is the same as the form of $\ol w$ for $w\leq u$.
\end{prop}

\begin{proof}
First suppose $\ol u=t_j\dots t_n\dots t_0$ is of form (II) with $n\neq k$ and $\ol x=tt_l\dots t_0$ (either simple or form (II)). If $n<k$ then $j<l$ and $m(t_l,t_{l-1})=3$, so $\ol x$ is simple. However, to have $\ol x$ detecting $\ol u$ we require $m(t,t_j)\geq3$ contradicting $\ol x$ simple. If $n>k$ then $j\geq l+2$ and $m(t,t_j)\geq3$, so $\ol x$ detects $\ol w=t_j\dots t_0$ contradicting the conditions of the proposition. This shows that if $\ol u$ is of form (II) then we have the result. If instead $\ol u$ is simple then by Lemma~\ref{lem:A} $l=k-1$ and $u$ forms a basket.

Finally, suppose $\ol u$ is of form (III). We have a lot of case-checking to do:
\begin{enumerate}
	\item If $\ell(x)\leq\ell(u)-3$, then since $\ol u''$ has the same first generator as $\ol u$, $\ol x$ also detects $\ol u''$ and we have a contradiction.
	\item If $\ell(x)=\ell(u)-2$, then we consider 4 sub-cases:
	\begin{itemize}
		\item If $\ol u=t_nt_{n-1}t_n\dots t_0$ then $\ol x\vartriangleright t_{n-1}\dots t_0$, so $\ol x$ also detects $t_{n-1}t_n\dots t_0$.
		\item If $\ol u=t_{n-1}t_nt_{n-1}t_n\dots t_0$ then $\ol x\vartriangleright t_n\dots t_0$, so $\ol x$ also detects $t_nt_{n-1}t_n\dots t_0$.
		\item If $\ol u=t_nt_{n-1}t_nt_{n-1}t_n\dots t_0$ then the only possibility for $\ol x$ is $t_{n-2}t_{n-1}t_n\dots t_0$ which doesn't detect $\ol u$.
		\item If $\ol u>t_nt_{n-1}t_nt_{n-1}t_n\dots t_0$ then there are no possibilities for $\ol x$.
	\end{itemize}	
	\item Suppose $\ell(x)=\ell(u)-1$ so $u=s_{k-1}s_k\dots s_0$, and suppose $x=s_{k-2}s_{k-1}\dots s_0$. Then $m(s_{k-1},s_{k-2})=4$ since otherwise $u$ and $v=s_{k-1}s_{k-2}s_{k-1}\dots s_0$ form a basket. We look at 2 sub-cases:
	\begin{itemize}
		\item If $\ol u=t_nt_{n-1}t_n\dots t_0$ then $\ol x=tt_n\dots t_0$ is simple. Then $m(t,t_n)=M(\ol u,\ol x)=M(u,x)=m(s_{k-1},s_{k-2})=4$, so $\ol u$ and $\ol v=t_ntt_n\dots t_0$ form a basket.
		\item If $\ol u=t_{n-1}t_nt_{n-1}t_n\dots t_0$ then $\ol x=t_{n-2}t_{n-1}t_n\dots t_0$. Let $\ol x_0=t_j\dots t_n\dots t_0$ be the maximal chainlike with $\ol x_0\geq\ol x$, $0\leq j\leq n-2$. Since $m(s_{k-1},s_{k-2})=4$ we have that $x_0$ is of form (II), and since $n=k-2$ we have $x_0=s_{j+4}\dots s_k\dots s_0$. Since there is no chainlike above $x_0$, $m(s_{j+4},s_{j+3})\geq4$. Thus if $j<n-2$ then $x_0$ is detected by $s_{j+3}s_{j+4}\dots s_0$, contradicting the first part of this proof, and if $j=n-2$ then $j+4=k$ and $s_{j+3}s_{j+4}\dots s_0$ contradicts maximality of $x_0$.
	\end{itemize}
	\item Suppose $\ell(x)=\ell(u)-1$ so $u=s_{k-1}s_k\dots s_0$, and suppose $x=ss_{k-1}\dots s_0$ is simple. Then $m(s,s_{k-1})=3$ since otherwise $u$ and $s_{k-1}ss_{k-1}\dots s_0$ would form a basket. Moreover we can assume $m(s_{k-1},s_{k-2})=3$, since otherwise $s_{k-2}s_{k-1}\dots s_0\in C$ and this reduces to the case above, and we denote $u_1=s_{k-2}s_{k-1}s_k\dots s_0\in C$. We also denote $w=s_ks_{k-1}ss_k\dots s_0$ which covers $s_{k-1}ss_k\dots s_0\in M_2(u,x)$. We look at 2 sub-cases, the first of which is exemplified in Figure~\ref{fig:form2}.
	\begin{itemize}
		\item If $\ol u=t_nt_{n-1}t_n\dots t_0$ then $\ol x=tt_n\dots t_0$ is simple. Then $\ol u_1$ begins with $t_{n-1}t_nt_{n-1}t_nt_{n-1}$ and so $m(t_n,t_{n-1})\geq6$. Since $\ol u_1$ is the only chainlike covering $\ol u$, $s_ks_{k-1}s_k\dots s_0$ cannot be chainlike and so $m(s_k,s_{k-1})=4$. Thus $w$ as defined above is semi-chainlike. Now $M_2(\ol u,\ol x)=\{tt_n\ol u',t_nt\ol u'\}$, so by Lemma~\ref{lem:semi} the only possibilities for $\ol w$ covering an element of $M_2(\ol u,\ol x)$ are:
		\begin{itemize}
			\item the element(s) of $M_3(\ol u,\ol x)$, which cover both elements of $M_2(\ol u,\ol x)$;
			\item $\ol w=t_jtt_n\ol u'=tt_jt_nt_{n-1}t_n\dots t_0$ with $j\leq n-1$, which is either not in $\ol W^{\ol J}$ or also covers $t_jt_nt_{n-1}t_n\dots t_0$;
			\item $\ol w=t_jt_nt\ol u'=t_ntt_jt_{n-1}t_n\dots t_0$ with $j\leq n-2$, which is either not in $\ol W^{\ol J}$ or also covers $t_nt_jt_{n-1}t_n\dots t_0$;
			\item $\ol w=t_{n-1}t_nt\ol u'$ which also covers $t_{n-1}t_nt_{n-1}t_n\dots t_0$ as $m(t_n,t_{n-1})\geq6$.
		\end{itemize}
		Thus $w$ has no semi-chainlike element to map to, and we have a contradiction.
		\item If $\ol u=t_{n-1}t_nt_{n-1}t_n\dots t_0$ then $\ol x=t_{n-2}t_{n-1}t_n\dots t_0$. Then $\ol u_1$ begins with $t_nt_{n-1}t_nt_{n-1}t_nt_{n-1}$ and so $m(t_n,t_{n-1})\geq7$. Since $\ol u_1$ is the only chainlike covering $\ol u$, $s_ks_{k-1}s_k\dots s_0$ cannot be chainlike and so $m(s_k,s_{k-1})=4$. Thus $w$ as defined above is semi-chainlike. Now $M_2(\ol u,\ol x)=\{tt_{n-1}\ol u',t_{n-1}t\ol u'\}$, so by Lemma~\ref{lem:semi} the only possibilities for $\ol w$ are:
		\begin{itemize}
			\item the element(s) of $M_3(\ol u,\ol x)$, which cover both elements of $M_2(\ol u,\ol x)$;
			\item $\ol w=t_jtt_{n-1}\ol u'=tt_jt_{n-1}t_nt_{n-1}t_n\dots t_0$ with $j\leq n-2$, which is either not in $\ol W^{\ol J}$ or also covers $t_jt_{n-1}t_nt_{n-1}t_n\dots t_0$;
			\item $\ol w=t_jt_{n-1}t\ol u'=t_{n-1}tt_jt_nt_{n-1}t_n\dots t_0$ with $j\leq n-3$, which is either not in $\ol W^{\ol J}$ or also covers $t_{n-1}t_jt_nt_{n-1}t_n\dots t_0$;
			\item $\ol w=t_{n-2}t_{n-1}t\ol u'$ which also covers $t_{n-2}t_{n-1}t_nt_{n-1}t_n\dots t_0$.
			\item $\ol w=t_ntt_{n-1}\ol u'$ which also covers $t_nt_{n-1}t_nt_{n-1}t_n\dots t_0$ as $m(t_n,t_{n-1})\geq7$;
			\item $\ol w=t_nt_{n-1}t\ol u'$ which also covers $t_nt_{n-1}t_nt_{n-1}t_n\dots t_0$ as $m(t_n,t_{n-1})\geq7$.
		\end{itemize}
		Thus $w$ has no semi-chainlike element to map to, and we have a contradiction.
	\end{itemize}
\end{enumerate}
Every case leads to a contradiction, so $\ol u$ is not of form (III) and the proof is complete.
\end{proof}

\vspace{-15pt}

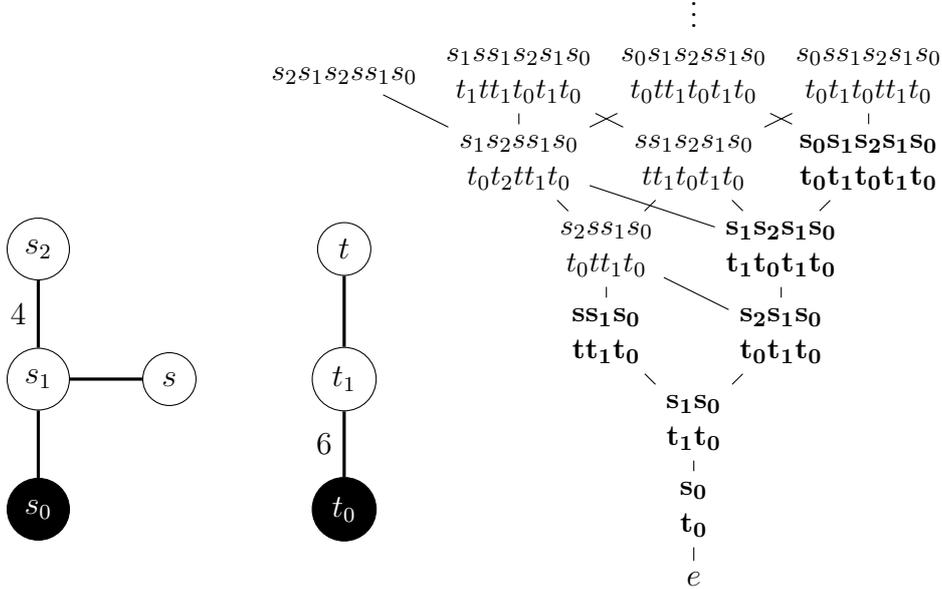
\begin{figure}[h]
\begin{center}
\begin{tikzpicture}[scale=1.15]
\begin{scope}[shift={(-2,1)}]
\begin{scope}
	\node[blacknode] (s0) at (0,0) {$s_0$};
	\node[whitenode] (s1) at (0,1.5) {$s_1$};
	\node[whitenode] (s2) at (0,3) {$s_2$};
	\node[whitenode] (s3) at (1.5,1.5) {$s$};
\end{scope}
\begin{scope}[every edge/.style={draw=black,very thick}]
    \path (s1) edge (s0) edge node[left] {$4$} (s2) edge (s3);
\end{scope}
\end{scope}
\begin{scope}[shift={(1.5,1)}]
\begin{scope}
	\node[blacknode] (t0) at (0,0) {$t_0$};
	\node[whitenode] (t1) at (0,1.5) {$t_1$};
	\node[whitenode] (t2) at (0,3) {$t$};
\end{scope}
\begin{scope}[every edge/.style={draw=black,very thick}]
    \path (t1) edge node[left] {$6$} (t0) edge (t2);
\end{scope}
\end{scope}
\begin{scope}[shift={(5.5,0)}]
\begin{scope}[every node/.style={align=center}]
	\node (e) at (0,0.2) {$e$};
	\node (0) at (0,1) {\small{$\*{s_0}$}\\\small{$\*{t_0}$}};
	\node (10) at (0,2) {\small{$\*{s_1s_0}$}\\\small{$\*{t_1t_0}$}};
	\node (310) at (-1,3) {\small{$\*{ss_1s_0}$}\\\small{$\*{tt_1t_0}$}};
	\node (210) at (1,3) {\small{$\*{s_2s_1s_0}$}\\\small{$\*{t_0t_1t_0}$}};
	\node (2310) at (-1,4) {\small{$s_2ss_1s_0$}\\\small{$t_0tt_1t_0$}};
	\node (1210) at (1,4) {\small{$\*{s_1s_2s_1s_0}$}\\\small{$\*{t_1t_0t_1t_0}$}};
	\node (12310) at (-2,5) {\small{$s_1s_2ss_1s_0$}\\\small{$t_0t_2tt_1t_0$}};
	\node (31210) at (0,5) {\small{$ss_1s_2s_1s_0$}\\\small{$tt_1t_0t_1t_0$}};
	\node (01210) at (2,5) {\small{$\*{s_0s_1s_2s_1s_0}$}\\\small{$\*{t_0t_1t_0t_1t_0}$}};
	\node (212310) at (-4,6) {\small{$s_2s_1s_2ss_1s_0$}};
	\node (131210) at (-2,6) {\small{$s_1ss_1s_2s_1s_0$}\\\small{$t_1tt_1t_0t_1t_0$}};
	\node[label=above:{$\vdots$}] (012310) at (0,6) {\small{$s_0s_1s_2ss_1s_0$}\\\small{$t_0tt_1t_0t_1t_0$}};
	\node (031210) at (2,6) {\small{$s_0ss_1s_2s_1s_0$}\\\small{$t_0t_1t_0tt_1t_0$}};
\end{scope}
\begin{scope}[every edge/.style={draw=black}]
	\path (0) edge (e) edge (10);
	\path (310) edge (10) edge (2310);
	\path (210) edge (10) edge (2310) edge (1210);
	\path (2310) edge (31210) edge (12310);
	\path (1210) edge (31210) edge (12310) edge (01210);
	\path (12310) edge (212310) edge (131210) edge (012310);
	\path (31210) edge (131210) edge (031210);
	\path (01210) edge (012310) edge (031210);
\end{scope}
\end{scope}
\end{tikzpicture}
\end{center}
\caption{Example of (4) in the proof of Proposition~\ref{prop:form2} with $k=2$, $n=1$. The two posets are drawn overlapping to show their difference, which is in the element $s_2s_1s_2ss_1s_0$ identified in the proof as $w$.}\label{fig:form2}
\end{figure}

\begin{corollary}[Stacked Label Property]\label{cor:double_label} Suppose $u=s_l\dots s_k\dots s_0$ of form (II) does not comprise a non-form-preserving basket with any $v\in C$. If either:
\begin{enumerate}[label=(\alph*)]
\item $m(s_j,s_{j-1})\geq4$ for some $j\leq l$; or
\item $l=0$ and $\ol u$ is of form (II),
\end{enumerate}
then $\ol u=t_l\dots t_k\dots t_0$ for some distinct $\{t_0,\dots,t_k\}\in\ol S$.
\end{corollary}

\begin{proof}
If $u$ is part of a form-preserving basket then by the classification of baskets we are done. For (a), assume $j$ maximal with this property. Then $x=s_{j-1}s_j\dots s_0$ is chainlike and detects $u$ and no chainlikes below $u$, so the result follows by Proposition~\ref{prop:form2}. For (b), suppose that $\ol u=t_j\dots t_n\dots t_0$ with $n>k$. Then $j>0$ and $m(t_j,t_{j-1})\geq4$ since there are no chainlikes above $u$, so $\ol x=t_{j-1}t_j\dots t_0$ is chainlike and detects $\ol u$ and no chainlikes below $\ol u$ and Proposition~\ref{prop:form2} gives a contradiction.
\end{proof}

\begin{proof}[Proof of Theorem~\ref{thm:basket_complete}]
Let $u,v$ form the non-form-preserving basket. We already have $\phi(w)$ the same form as $\ol w$ for $w\geq u',w\geq v'$ or $w\leq u''$. If $w\in C$ is not any of these, then it contains generators in the additional graph of white nodes in Figure~\ref{fig:basket_case} and $\phi(w)=w$. However, if $w$ is of form (II) then Corollary~\ref{cor:double_label} applies due to the $\infty$ label between $s_k$ and $s_{k-1}$. Repeating for the direction $\ol W^{\ol J}\to W^J$, we have $w\in C$ is of form (II) if and only if $\ol w$ is of form (II). Then $w$ is of form (III) if and only if it is not of form (II) and is greater than some form (II) element, and $w$ is simple otherwise, so the forms of all elements of $C$ and $\ol C$ are determined and Theorem~\ref{thm:graph} applies.
\end{proof}

\subsection{Labels of 5 or greater}\label{subsec:5}

With Theorem~\ref{thm:basket_complete} proven we can now assume that the isomorphism $W^J\to\ol W^{\ol J}$ contains no non-form-preserving baskets. We will call such a pair of isomorphic posets \textit{basket-free}. First we focus on form (III) elements, and classify all isomorphisms for which a form (III) chainlike maps to a form (I) or (II) chainlike. We start with edge labels greater than or equal to 6.

\begin{prop}\label{prop:form3}
Let $W^J\to\ol W^{\ol J}$ be basket-free, and suppose $u$ is of form (III) with $u=s_{k-1}s_ks_{k-1}s_k\dots s_0$ (so $m(s_k,s_{k-1})\geq6)$ and $k\geq2$. Then the form of $w\in C$ is the same as the form of $\ol w$ for $w\leq u$ or $w\geq u$.
\end{prop}

\begin{proof}
If $m(s_{k-1},s_{k-2})\geq4$ then $x_0=s_{k-2}s_{k-1}\dots s_0$ is chainlike and detects both $u$ and $u''=s_{k-1}s_k\dots s_0$, so Proposition~\ref{prop:form2} gives $\ol u''=t_{k-1}t_k\dots t_0$. Then we cannot have $\ol u=t_{k-3}\dots t_k\dots t_0$ as $\ol x_0$ would not detect $\ol u$, so we must have $\ol u=t_{k-1}t_kt_{k-1}t_k\dots t_0$ and the result follows.

If on the other hand $m(s_{k-1},s_{k-2})=3$ then $x=s_{k-2}s_{k-1}s_k\dots s_0\in C$ is chainlike and detects $u$. If $\ol u$ is of form (II), then Proposition~\ref{prop:form2} implies $u$ is of form (II), a contradiction. Now suppose $\ol u$ is simple, so $\ol u=t_{k+3}\dots t_0$ and $\ol x=tt_{k+1}\dots t_0$ is simple. Then $\ol v=t_{k+3}x$ is chainlike with $\ol v\sim\ol u$, but this means $v=s_{k-1}s_{k-2}s_{k-1}s_k\dots s_0$ which cannot be chainlike, a contradiction. Thus $\ol u$ is of form (III).

A possibility for $\ol u$ of form (III) is to have $\ol u=t_{k+1}t_kt_{k+1}\dots t_0$ and $x=tt_{k+1}\dots t_0$ simple. Since $u$ is the only chainlike covering $u'$ there are no form (II) chainlikes above $\ol u'=t_kt_{k+1}\dots t_0$ and so $m(t_k,t_{k-1})\geq4$. But then $t_{k-1}t_k\dots t_0$ detects $\ol u'$ so $u'$ is of form (II) by Proposition~\ref{prop:form2}, a contradiction. The only remaining possibility is that $\ol u=t_{k-1}t_kt_{k-1}t_k\dots t_0$, and the result follows.
\end{proof}

Now we look at labels of 5, for which we use the $N$ function.

\begin{prop}\label{prop:n5}
Let $W^J\to\ol W^{\ol J}$ be basket-free and suppose $u=s_ks_{k-1}s_k\dots s_0$ is of form (III), $m(s_k,s_{k-1})=5$ and $k\geq3$. Then the form of $w\in C$ is the same as the form of $\ol w$ for $w\leq u$ or $w\geq u$.
\end{prop}

\begin{proof}
If $m(s_{k-1},s_{k-2})\geq5$, then $u'=s_{k-1}s_k\dots s_0$ forms a basket with $s_{k-1}s_{k-2}s_{k-1}\dots s_0$. Since the isomorphism is basket-free we have $\ol u'=t_{k-1}t_k\dots t_0$ and $m(t_{k-1},t_{k-2})\geq5$, so necessarily $\ol u=t_kt_{k-1}t_k\dots t_0$ and we are done. Now consider $m(s_{k-1},s_{k-2})=4$ and let $x=s_{k-2}s_{k-1}\dots s_0$. With $u'$ detected by $x$, we have $\ol u'=t_{k-1}t_k\dots t_0$ by Proposition~\ref{prop:form2}. If $\ol u=t_{k-2}\ol u'$ then since there are no chainlikes above $u$, $m(t_{k-2},t_{k-3})=3$. But then $\ol x_1=t_{k-3}t_{k-2}\dots t_0$ detects $\ol u$, so $u$ must be of form (II) by Proposition~\ref{prop:form2}, a contradiction. Finally, suppose $m(s_{k-1},s_{k-2})=3$ and let $v=s_{k-2}s_{k-1}s_k\dots s_0$. We have $M(u,v)=2$ and $N(u,v)=5$. We have 2 cases which satisfy $N(\ol u,\ol v)=5$:
\begin{itemize}
	\item Suppose $\ol u=t_{k+2}\dots t_0,\ol v=tt_{k+1}\dots t_0$ are simple. To have $N(\ol u,\ol v)=5$ we require $m(t_{k-2},t_{k-3})\geq4$ and $m(t_{k-1},t_{k-2})=m(t_k,t_{k-1})=m(t_{k+1},t_k)=3$, so $\ol x=t_{k-3}t_{k-2}\dots t_0$ is chainlike and corresponds to $x\in C$ with $x\vartriangleright s_{k-2}\dots s_0$ for some $s$. But then $x$ detects $v$ while $\ol x$ cannot detect $\ol v$, a contradiction.
	\item Suppose $\ol u=t_{k-2}t_{k-1}t_k\dots t_0$ and $\ol v=t_kt_{k-1}t_k\dots t_0$. We have $m(s_{k-2},s_{k-3})=3$, since otherwise $x=s_{k-3}s_{k-2}\dots s_0$ does not detect $u$ but $\ol x\vartriangleright t_{k-2}\dots t_0$ detects $\ol u$. But then $v_1=s_{k-3}\dots s_k\dots s_0$ is chainlike and we then must have $\ol v_1=t_{k-1}t_kt_{k-1}t_k\dots t_0$, and we apply Proposition~\ref{prop:form3} to obtain a contradiction.
\end{itemize}
The only case left with $N=5$ is $\ol u,\ol v$ having equivalent generator expressions to $u,v$.
\end{proof}

The above two propositions cover most cases, but break down if $k$ is too small. Now we consider the $k=1$ and $k=2$ exceptional cases individually:

\begin{prop}\label{prop:i2m}
Let $W^J\to\ol W^{\ol J}$ be basket-free and suppose $u=s_1s_0s_1s_0$ is of form (III) (so $k=1$). If there is $s\in S\setminus\{s_0,s_1\}$ with an edge to $s_0$ or $s_1$, then the form of $w\in C$ is the same as the form of $\ol w$ for $w\leq u$ or $w\geq u$.
\end{prop}

\begin{prop}\label{prop:h3}
Let $W^J\to\ol W^{\ol J}$ be basket-free and suppose $u=s_2s_1s_2s_1s_0$ is of form (III) (so $k=2$) and $m(s_1,s_2)=5$. If $m(s_0,s_1)\geq4$ or there is $s\in S\setminus\{s_0,s_1,s_2\}$ with an edge to $s_0$, $s_1$ or $s_2$, then the form of $w\in C$ is the same as the form of $\ol w$ for $w\leq u$ or $w\geq u$.
\end{prop}

\begin{proof}[Proof of Proposition~\ref{prop:i2m}]
If $\ol u$ is not of form (III) then $\ol u$ is either $t_3t_2t_1t_0$ or $t_1t_2t_1t_0$. In either case $t_2t_1t_0$ is chainlike for some $t_0,t_1,t_2\in\ol S$ distinct. If there is $s\in S\setminus\{s_0,s_1\}$ with $m(s,s_0)\geq3$, then either $s$ or $ss_0$ is chainlike and so there is $t\in\ol S\setminus\{t_0,t_1\}$ with $m(t,t_0)\geq3$. But then $s$ or $ss_0$ detects $s_0s_1s_0$, so $t$ or $tt_0$ detects $t_2t_1t_0$ and $m(t,t_2)\geq3$. Thus $\ol u'\sim\ol v$ where either $\ol v=t_2tt_0$, in which case $u'$ and $v$ form a basket, or $\ol v=t_2t$, a contradiction as there is no $v\in C$ of length 2 with $v\sim u'=s_0s_1s_0$.

Now suppose there is $s\in S$ with $m(s,s_0)=2$ but $m(s,s_1)\geq3$. If $s\not\in J$, then $s_1s\in C$ and similarly $t_1t\in C$. Also, $s$ does not detect $s_0s_1s_0$ so $t$ does not detect $t_2t_1t_0$, so $m(t_2,t)=2$ and $t_2t_1t\in\ol C$. Then $\ol u'\sim\ol v=t_2t_1t$, but this implies $v\in C$ starts with $s_0$ and ends with $s$, a contradiction. If instead $s\in J$, then $x=ss_1s_0$ detects $u$. If $\ol u$ is of form (II) then we have a contradiction by Proposition~\ref{prop:form2}, and if $\ol u=t_3t_2t_1t_0$ is simple then $\ol u$ and $\ol v=t_3\ol x$ form a basket. Hence we must have $\ol u=t_1t_0t_1t_0$ and the result follows.
\end{proof}

\begin{proof}[Proof of Proposition~\ref{prop:h3}]
If $m(s_0,s_1)\geq4$, then $x=s_0s_1s_0$ detects $u'=s_1s_2s_1s_0$ and we are done by Proposition~\ref{prop:form2}. So assume $m(s_0,s_1)=3$ so that $v=s_0s_1s_2s_1s_0\in C$. If $\ol u$ is not of form (III) then we have two possible cases.

Firstly, we could have $\ol u=t_4t_3t_2t_1t_0,\ol v=t_5t_3t_2t_1t_0$ for some distinct $t_0,\dots,t_5\in\ol S$. If $s\not\in J$ then Proposition~\ref{prop:black} leads to a contradiction, so suppose $s\in J$. If $m(s,s_0)\geq3$ or $m(s,s_1)\geq3$, then we can apply Lemma~\ref{lem:A} to $v$ with $x=ss_0$ or $v'$ with $x=ss_1s_0$. If instead $m(s,s_2)\geq3$ and $m(s,s_0)=m(s,s_1)=2$ then $x=ss_2s_1s_0$ detects $u$, but then $\ol u$ and $t_4\ol x$ form a basket.

Secondly, we could have $\ol u=t_0t_1t_2t_1t_0,\ol v=t_2t_1t_2t_1t_0$. We cannot have $x=ss_2s_1s_0$ chainlike as then $x$ detects $u$ but $\ol x\vartriangleright t_2t_1t_0$ cannot detect $\ol u$. Thus there is a chainlike $x\in\{s,ss_0,ss_1s_0\}$ with the reduced expression for $\ol x$ starting with $t\in\ol S$, a generator having the same edges to $t_0,t_1,t_2$ as $s$ has to $s_0,s_1,s_2$ respectively. Furthermore since $u$ begins with $s_2$ but $\ol u$ begins with $s_0$, by seeing whether $x$ detects $u$ or $v$ we have $m(s,s_0)=m(s,s_2)=m(t,t_0)=m(t,t_2)$. We split this into a number of cases and obtain a contradiction in each.
\begin{itemize}
	\item If $s$ has an edge to $s_0,s_2$ but not $s_1$, consider $w=s_0s_1s_2s_1ss_0$ which has $w\vartriangleright v$ and $w>s$. This covers 4 elements ($v$, $s_0s_1s_2ss_0$, $s_0s_2s_1ss_0$, $s_1s_2s_1ss_0$), but the possibilities for $\ol w$ are $tt_2t_1t_2t_1t_0$, $t_2t_1tt_2t_1t_0$, $t_2t_1t_2t_1tt_0$ and $t_2t_1t_2t_1t_0t$ (if $t\not\in J$), which all cover 2 or 3 elements.
	\item If $s$ has an edge to $s_0,s_1$ and $s_2$, consider $w=s_0s_1s_2ss_1s_0$ which has $w\vartriangleright v$ and $w>s$. This covers 5 elements ($v$, $s_0s_1s_2ss_0$, $s_0s_1ss_1s_0$, $s_0s_2ss_1s_0$, $s_1s_2ss_1s_0$), but the possibilities for $\ol w$ are $tt_2t_1t_2t_1t_0$, $t_2tt_1t_2t_1t_0$, $t_2t_1tt_2t_1t_0$, $t_2t_1t_2tt_1t_0$ and $t_2t_1t_2t_1tt_0$, which all cover 2 to 4 elements.
	\item Suppose $s$ has an edge to $s_1$ but not $s_0,s_2$.
	\begin{itemize}
		\item If $s\not\in J$, consider $w=s_0s_1s_2s_1s_0s$ which has $w\vartriangleright v$ and $w>s$. This covers 4 elements ($v$, $s_0s_1s_2s_1s$, $s_0s_2s_1s_0s$, $s_1s_2s_1s_0s$) but the possibilities for $\ol w$ are $t_2tt_1t_2t_1t_0$, $t_2t_1t_2tt_1t_0$ and $t_2t_1t_2t_1t_0t$, which all cover 2 to 3 elements.
		\item If $s\in J$ and $m(s,s_1)\geq4$, consider $w=s_0s_1s_2ss_1s_0$ which has $w\vartriangleright v$ and $w>s$. This covers 3 elements ($v$, $s_0s_1ss_1s_0$, $s_1s_2ss_1s_0$), 2 of which are chainlike, and the possibilities for $\ol w$ are $t_2tt_1t_2t_1t_0$ and $t_2t_1t_2tt_1t_0$. The former covers only 2 elements, and the latter only one chainlike $\ol v$.
		\item If $s\in J$ and $m(s,s_1)=3$, consider $w=s_1s_2s_1s_2ss_1s_0$ which has $\ell(w)=7$ and $w\not>v$. This is semi-chainlike with $w'=s_2s_1s_2ss_1s_0$, but the possibilities for $\ol w$ are $t_1t_0tt_1t_2t_1t_0$ and $t_0t_1tt_1t_2t_1t_0$, which both cover more than 1 element.
	\end{itemize}
\end{itemize}
Thus we have shown $\ol u$ is of form (III), and we cannot have $\ol u=t_0t_1t_0t_1t_0$ by Proposition \ref{prop:i2m}, so we have the required result.
\end{proof}

\begin{remark}
These propositions along with the results above state that if a bw-Coxeter graph contains an edge labelled with 5 or greater, and the corresponding poset has an isomorphism to another poset not preserving the form of a chainlike of form (II) or (III) produced by that edge, then the graph must be either $(I_2(m),A_1)$ with $m\geq5$ or $(H_3,H_2)$. It is then easy to check that the only possible isomorphisms for these are:
\begin{itemize}
	\item $(H_3,H_2)\leftrightarrow(D_6,D_5)$;
	\item $(H_3,H_2)$ with an automorphism swapping the two longest chainlike elements;
	\item $(I_2(n),A_1)\leftrightarrow(A_{n-1},A_{n-2})$ for $n\geq5$;
	\item $(I_2(2n),A_1)\leftrightarrow(B_n,B_{n-1})$ for $n\geq3$.
\end{itemize}
All of these (aside from the automorphism) are covered in Section~\ref{subsec:classification}. Note that while this method reveals the $(H_3,H_2)$ automorphism, it doesn't guarantee that there are no other automorphisms, only that these unknown automorphisms preserve the forms of the chainlikes. For example, the Bruhat poset of $I_2(m)$, $m\geq3$ generated by $\{s_0,s_1\}$ has an automorphism swapping $s_0s_1$ with $s_1s_0$ and leaving other elements fixed.
\end{remark}

\subsection{Completing the classification}\label{subsec:completing}

Now we check cases with labels of $4$. Since we have covered all possibilities involving form (III) chainlikes we assume from this point forward that, in addition to $W^J\to\ol W^{\ol J}$ being basket-free, $u$ is of form (III) if and only if $\ol u$ is of form (III) for all $u\in C$. An immediate consequence is that if $u$ and $\ol u$ are both of form (II) then, by counting the number of non-form-(III) chainlikes above $u$ and $\ol u$ and applying Corollary~\ref{cor:double_label} if necessary, $w$ and $\ol w$ have the same form for all $w\leq u$ or $w\geq u$. This means we only have one case left to consider: an isomorphism which maps $u_0=s_0\dots s_k\dots s_0$ to $\ol u_0=t_{2k}\dots t_0$. If $s_0,\dots,s_k$ are the only generators in $S$ then this is the isomorphism $(B_{k+1},B_k)\to(A_{2k+1},A_{2k})$. If any $x\in C$ detects any $s_l\dots s_k\dots s_0\leq u_0$, then we can apply either Proposition~\ref{prop:black} or Lemma~\ref{lem:A}, so the only way to add generators to $S$ while maintaining the isomorphism is to connect new generators only to $s_k$. The next 2 propositions deal with this case.

\begin{prop}\label{prop:ba}
With $u_0=s_0\dots s_k\dots s_0$ and $\ol u_0=t_{2k}\dots t_0$ as above, suppose $k\geq2$. Then there cannot be any $s\in S\setminus\{s_0,\dots,s_k\}$ with $m(s,s_i)\geq3\iff i=k$, and so by the above remarks $S=\{s_0,\dots,s_k\}$.
\end{prop}

If $s\in J$ there are two branches of chainlikes above $s_k\dots s_0$, one starting with $ss_k\dots s_0$ and the other with $s_{k-1}s_k\dots s_0$. The most complex part of the proof is to verify that the map swapping these branches does not extend to an automorphism of the poset (except when $k=1$ in which $(B_3,A_2)$ has this automorphism), and this is illustrated in Figure~\ref{fig:bauto}.

\begin{proof}
If $s\not\in J$ and $t=\ol s$ then we see that $\ol u_1=t_{2k}\dots t_kt$ is chainlike and $\ol u_1\sim t_{2k}\dots t_0$, a contradiction as $u_1$ would have to start with $s_0$ and end with $s$. If instead $s\in J$, let $u=s_{k-1}s_k\dots s_0$ and $v=ss_k\dots s_0$. If $\ol v$ is simple, then $N(\ol u,\ol v)=k+2\geq4$ while $N(u,v)\in\{1,3\}$, a contradiction, so assume $\ol v$ is of form (II). We have $m(t_k,t_{k-1})=4$, and also must have $m(t_i,t_{i+1})=3$ for $i\geq k$ since otherwise $t_it_{i+1}\dots t_0$ would either map to a form (III) element in $W^J$ or not have any chainlike to map to. Denoting $s_{k+1}=s$ and $u_2=s_{k-2}s_{k-1}s_k\dots s_0$ (so $\ol u_2=t_{k+2}\dots t_0$), we have
\begin{align*}
	N_1(u,v)&=\{s_ks_{k-1}s_{k+1}s_k\dots s_0\} & N_1(\ol u,\ol v)&=\{t_kt_{k-1}t_{k+1}t_k\dots t_0\} \\
	M_2(u_2,v)&=\{s_{k-2}s_{k-1}s_{k+1}s_k\dots s_0\} & M_2(\ol u_2,\ol v)&=\{t_{k+2}t_{k-1}t_{k+1}t_k\dots t_0\}
\end{align*}
Consider $\ol w=t_{k+1}t_kt_{k+2}t_{k-1}t_{k+1}t_k\dots t_0$. We have $\ol w\in\ol W^{\ol J}$ and $\ol w$ is semi-chainlike, and moreover $\ol w'$ covers the unique elements of $N_1(\ol u,\ol v)$ and $M_2(\ol u_2,\ol v)$. Then we must have $w'=s_ks_{k-2}s_{k-1}s_{k+1}s_k\dots s_0$. By Lemma~\ref{lem:semi} the possibilities for $w$ are $s_{k-3}w'$, $s_{k-1}w'$ and $s_{k+1}w'$ which it is easy to check are all either not semi-chainlike or not in $W^J$. Thus we have a contradiction.
\end{proof}

\begin{figure}[h]
\begin{center}
\begin{tikzpicture}[scale=0.8]
\begin{scope}
	\node[blacknode] (0) at (0,0) {$0$};
	\node[whitenode] (1) at (0,2) {$1$};
	\node[whitenode] (2) at (0,4) {$2$};
	\node[whitenode] (3) at (0,6) {$3$};
	\node[whitenode] (4) at (0,8) {$4$};
\end{scope}
\begin{scope}[every edge/.style={draw=black,very thick}]
    \path (1) edge (0) edge node[left] {$4$} (2);
    \path (3) edge (2) edge (4);
\end{scope}
\end{tikzpicture}\hspace{50pt}
\begin{tikzpicture}
\begin{scope}
	\node (e) at (0,1) {$e$};
	\node (0) at (0,1.8) {$\*0$};
	\node (10) at (0,2.6) {$\*{10}$};
	\node (210) at (0,3.4) {$\*{210}$};
	\node (3210) at (-1,4.2) {$\*{3210}$};
	\node (1210) at (1,4.2) {$\*{1210}$};
	\node (43210) at (-2,5) {$\*{43210}$};
	\node (13210) at (0,5) {$13210$};
	\node (01210) at (2,5) {$\*{01210}$};
	\node (143210) at (-2,5.9) {$143210$};
	\node (213210) at (0,5.9) {$213210$};
	\node (013210) at (2,5.9) {$013210$};
	\node (2143210) at (-3,7) {$2143210$};
	\node (0143210) at (-1,7) {$0143210$};
	\node (1213210) at (1,7) {$1213210$};
	\node (0213210) at (3,7) {$0213210$};
	\node (32143210) at (-4,8.1) {$32143210$};
	\node (12143210) at (-2,8.1) {$12143210$};
	\node[label=above:{$\vdots$}] (02143210) at (0,8.1) {$02143210$};
	\node (10213210) at (2,8.1) {$10213210$};
	\node (01213210) at (4,8.1) {$01213210$};
\end{scope}
\begin{scope}[every edge/.style={draw=black}]
	\path (0) edge (e) edge (10);
	\path (210) edge (10) edge (3210) edge (1210);
	\path (3210) edge (43210) edge (13210);
	\path (1210) edge (13210) edge (01210);
	\path (43210) edge (143210);
	\path (13210) edge (143210) edge (213210) edge (013210);
	\path (01210) edge (013210);
	\path (143210) edge (2143210) edge (0143210);
	\path (213210) edge (2143210) edge (1213210) edge (0213210);
	\path (013210) edge (0143210) edge (0213210);
	\path (2143210) edge (32143210) edge (12143210) edge (02143210);
	\path (0143210) edge (02143210);
	\path (1213210) edge (12143210) edge (10213210) edge (01213210);
	\path (0213210) edge (02143210) edge (10213210) edge (01213210);
\end{scope}
\node[draw,rectangle] at (-3.5,2.5) {$\begin{aligned}
	\text{In }W^J:&\phantom{\ol W^{\ol J}}\\
	u&=1210\\
	v&=3210\\
	N_1(u,v)&=213210\\
	M_2(u_2,v)&=013210
	\end{aligned}$};
\node[draw,rectangle] at (3.5,2.5) {$\begin{aligned}
	\text{In }\ol W^{\ol J}:&\phantom{W^J}\\
	\ol u&=3210\\
	\ol v&=1210\\
	N_1(\ol u,\ol v)&=213210\\
	M_2(\ol u_2,\ol v)&=143210
	\end{aligned}$};
\end{tikzpicture}
\end{center}
\caption{Example of the proof of Proposition~\ref{prop:ba} with $k=2$. The graph on the right may be interpreted as either of the two posets depending on the choice of $u$ and $v$, as shown in the legends. Notice how $\overline w=32143210$ has no corresponding semi-chainlike element in $W^J$.}\label{fig:bauto}
\end{figure}
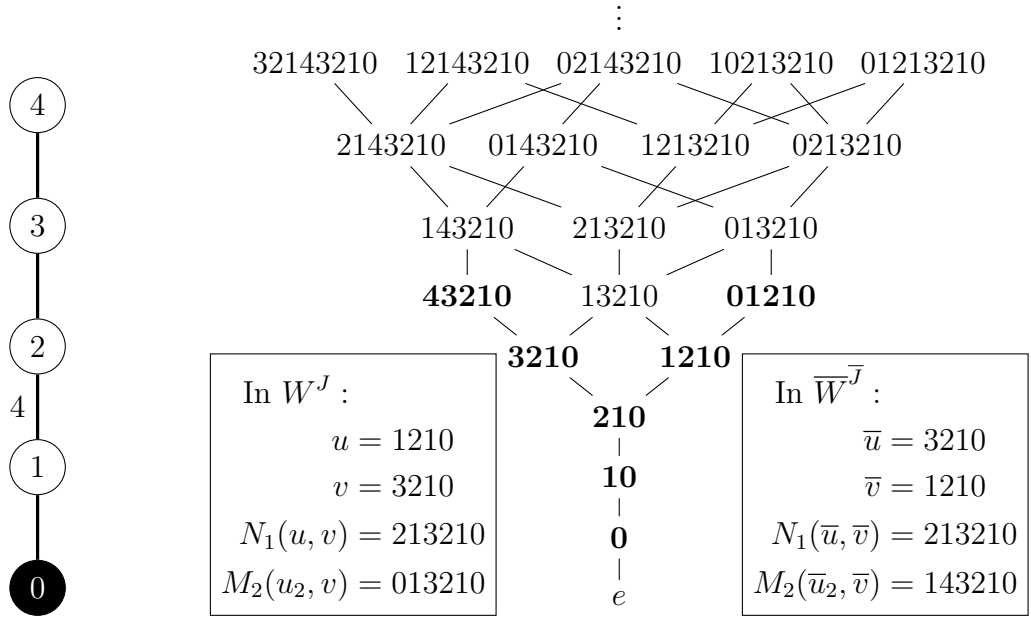

Finally, we work through the $k=1$ case separately:

\begin{prop}
If $u_0$ and $\ol u_0$ as above are given by $u_0=s_0s_1s_0$ and $\ol u_0=tt_1t_0$ (so $k=1$), then $(W,W_J)=(B_{n+1},A_n)$ for some $n\geq1$.
\end{prop}

\begin{proof}
First we show there is at most one chainlike covering $s_1s_0$ not equal to $u_0$. Suppose there are $s_a,s_b\in J$ with $v_a=s_as_1s_0$ and $v_b=s_bs_1s_0$ chainlike and distinct from $u_0$, so then $N(u_0,v_a)=N(u_0,v_b)\in\{1,3\}$. At most one of $\ol v_a,\ol v_b$ is non-simple, so assume $\ol v_a=t_at_1t_0$ is simple. If $m(t,t_1)=4$ then $t_1tt_1t_0\vartriangleright\ol u_0$ is chainlike but there is no chainlike covering $u_0$, a contradiction. Moreover, if $m(t_a,t_1)=4$ then $\ol v_1=t_1t_at_1t_0$ is chainlike and detected by $\ol u_0$ giving $v_1=s_1s_as_1s_0$, so $m(t,t_1)=M(\ol v_1,\ol u_0)=M(v_1,u_0)=m(s_1,s_0)=4$, the same contradiction. This means $m(t,t_1)=m(t_a,t_1)=3$, but then $N(\ol u_0,\ol v_a)=2$. Thus there is at most one simple chainlike $v_0$ covering $s_1s_0$.

Now suppose $s_n\dots s_0$ is a maximal simple chainlike above $s_1s_0$ with $n\geq2$, meaning $v_0=s_2s_1s_0$. We have $m(s_1,s_2)=3$ since otherwise $s_1s_2s_1s_0$ is detected by $u_0$ giving a contradiction by Lemma~\ref{lem:A}. If $m(t_0,t_1)\geq4$ then we must have $\ol v_0=t_0t_1t_0$, since otherwise $\ol u_0$, $\ol v_0$ and $t_0t_1t_0$ are distinct elements covering $t_1t_0$ and we have a contradiction as above. Since there is no chainlike covering $\ol v_0$, we must have $n=2$ in this case.

If instead $m(t_0,t_1)=3$ and $\ol v_0=t_2t_1t_0$ is simple, we have the following results:
\begin{itemize}
\item We show there are no labelled edges in the chainlikes $s_i\dots s_0$. Suppose $j\geq2$ is minimal with $m(s_j,s_{j-1})\geq4$. By Corollary~\ref{cor:double_label}, $\ol u_1=t_1\dots t_j\dots t_0$. We are assuming there are no chainlikes of form (III) above $u_1$, so since $s_0\dots s_j\dots s_0\not\in C$ we require $m(t_0,t_1)\geq4$, a contradiction.
\item We show there are no branches in the chainlikes $s_i\dots s_0$. Suppose $u,v$ are two simple chainlike elements with $u'=v'=s_i\dots s_0$, $i\geq2$. By Proposition~\ref{prop:ba} and remarks above $\ol u$ and $\ol v$ must also be simple, so to have $N(u,v)=N(\ol u,\ol v)$ we require $m(t_0,t_1)\geq4$, a contradiction.
\end{itemize}

Thus the chainlikes above $s_0$ are those of the Coxeter pair $(W,W_J)=(B_{n+1},A_n)$. All we have left to show is that there is no black node $x\in S\setminus J$ connecting to $s_i$ for some $i\geq1$. Assume $i$ is minimal, so $\ol x$ connects to $t_i$ where $t_i\dots t_0=\overline{s_i\dots s_0}$. But then either $tt_1\dots t_i\ol x$ or $t\ol x$ is chainlike giving $\ol v\in C$ with $\ol v\sim\ol u_0$, so $v$ must begin with $s_0$ and end with $x$, a contradiction.
\end{proof}

\begin{remark}
The above shows that the only remaining unchecked isomorphisms are:
\begin{itemize}
	\item $(B_n,B_{n-1})\leftrightarrow(A_{2n-1},A_{2n-2})$ for $n\geq2$ (denoting $B_1=A_1$);
	\item $(B_n,A_{n-1})\leftrightarrow(D_{n+1},A_n)$ for $n\geq3$;
	\item $(B_2,A_1)$ with an automorphism swapping the two longest chainlike elements.
\end{itemize}
All of these (aside from the automorphism) are covered in Section~\ref{subsec:classification}. This completes the classification, written up in the following theorem which expands on Theorem~\ref{thm:one}.
\end{remark}

\begin{theorem}\label{thm:final}
Suppose $(W,S),(\ol W,\ol S)$ are Coxeter systems with connected graphs and $J\subset S$, $\ol J\subset\ol S$. The Bruhat posets $W^J$ and $\ol W^{\ol J}$ are isomorphic if and only if $(W,W_J),(\ol W,\ol W_{\ol J})$ are of one of the following forms:
\begin{itemize}
	\item $(I_2(n),A_1)\leftrightarrow(A_{n-1},A_{n-2})$ for $n\geq4$;
	\item $(B_n,B_{n-1})\leftrightarrow(A_{2n-1},A_{2n-2})$ for $n\geq3$;
	\item $(B_n,B_{n-1})\leftrightarrow(I_2(2n),A_1)$ for $n\geq3$;
	\item $(B_n,A_{n-1})\leftrightarrow(D_{n+1},A_n)$ for $n\geq3$;
	\item $(H_3,H_2)\leftrightarrow(D_6,D_5)$;
	\item There is a bijection $\sigma:S\to\ol S$ with $m(s_1,s_2)=m(\sigma(s_1),\sigma(s_2))$ and $s_1\in J\iff\sigma(s_1)\in\ol J$ for all $s_1,s_2\in S$, that is the pairs $(W,W_J)$ and $(\ol W,\ol W_{\ol J})$ are isomorphic;
	\item $S=J$ and $\ol S=\ol J$, that is $W^J$ and $\ol W^{\ol J}$ are trivial.
\end{itemize}
Moreover, $W^J$ has an automorphism which is not the identity when restricted to $C(W^J)$ if and only if $(W,W_J)$ is of one of the following forms:
\begin{itemize}
	\item $(B_2,A_1)$, with the automorphism swapping the longest two chainlike elements;
	\item $(H_3,H_2)$, with the automorphism swapping the longest two chainlike elements;
	\item The graph of $(W,W_J)$ is as depicted in Figure 6;
	\item There is a non-trivial bijection $\sigma:S\to S$ with $m(s_1,s_2)=m(\sigma(s_1),\sigma(s_2))$ and $s_1\in J\iff\sigma(s_1)\in J$ for all $s_1,s_2\in S$, that is there is an automorphism of the bw-Coxeter graph of $(W,W_J)$.
\end{itemize}
\end{theorem}

\begin{proof}[Proof of Theorem~\ref{thm:final} and Theorem~\ref{thm:one}]
Both results follow from the remarks above and the results in Section~\ref{subsec:classification}.
\end{proof}

\section{Extensions and observations}\label{sec:observations}

Here we detail the main application which is extending the result in \cite{Co} for blocks of category $\mc O$, and discuss directions for further study.

\subsection{Blocks of category $\mc O$}\label{subsec:categoryproof}

Let $\mf g$ be a symmetrizable Kac-Moody algebra over $\C$ with a root space decomposition with Weyl group $\mc W$ (a crystallographic Coxeter group), Borel subalgebra $\mf b$ and Cartan subalgebra $\mf h$. We define the usual partial order on $\mf h^*$ by $\lambda\geq\mu$ if and only if $\lambda-\mu$ is a linear combination of simple roots with positive integer coefficients.

There are two definitions for the category $\mc O$ corresponding to $\mf g$ in the literature, and Theorem~\ref{thm:two} applies to both. We distinguish these following the notation of \cite{BS}:
\begin{itemize}
	\item $\mc O$ will denote the category of $\mf g$-modules $M$ that are $\mf h$-semisimple with finite-dimensional weight spaces, and such that there are finitely many weights $\lambda_1,\dots,\lambda_n$ depending on $M$ so that every weight $\lambda$ of $M$ has $\lambda\leq\lambda_i$ for some $i$.
	\item $\widehat{\mc O}$ will denote the category of $\mf g$-modules that are $\mf h$-semisimple and locally $\mf b$-finite.
\end{itemize}

Both $\mc O$ and $\widehat{\mc O}$ contain simple objects $L(\lambda)$ parametrised by $\lambda\in\mf h^*$. Since $\mf g$ is symmetrizable it admits a non-degenerate, symmetric, invariant bilinear form $(\cdot,\cdot)$ on $\mf h^*$ with an element $\rho\in\mf h^*$ such that $(\rho,\alpha)=1$ for any simple root $\alpha$. Let $\sim$ be the equivalence relation on $\mf h^*$ generated by $\lambda\sim\mu$ whenever there are $n\in\N$ and a positive root $\alpha$ such that $\lambda-\mu=n\alpha$ and $2(\lambda+\rho,\alpha)=n(\alpha,\alpha)$. The equivalence classes $\Lambda\in\mf h^*/\sim$ give a decomposition of $\mc O$ (resp. $\widehat{\mc O}$) into blocks $\mc O_\Lambda$ (resp. $\widehat{\mc O}_\Lambda$) consisting of modules $M$ with $[M:L(\lambda)]\neq0\implies\lambda\in\Lambda$. We call $\lambda\in\Lambda$ \textit{dominant} (resp. \textit{antidominant}) if $\lambda\geq\mu$ (resp. $\lambda\leq\mu$) for all $\mu\in\Lambda$, and similarly call $\mc O_\Lambda$ and $\widehat{\mc O}_\Lambda$ dominant or antidominant if $\Lambda$ contains a dominant or antidominant weight.

As in \cite{Fi} we denote by $\Stab(\lambda)$ the subgroup of $\mc W$ for which the weight $\lambda$ is fixed by the dot-action $w\cdot\lambda=w(\lambda+\rho)-\rho$, and $\mc W(\Lambda)$ the subgroup generated by reflections $s_\alpha$ where $\alpha$ is a root with $2(\lambda+\rho,\alpha)\in\Z(\alpha,\alpha)$ for some $\lambda\in\Lambda$ (in the finite-dimensional case $\mc W(\Lambda)$ and $\Stab(\lambda)$ are often written as $\mc W_{[\lambda]}$ and $\mc W_\lambda$, see e.g. \cite{Hu}). We say $\Lambda$ \textit{lies outside the critical hyperplanes} if all of these $\alpha$ above are real roots, and in this case we in fact have $\Lambda=\mc W(\Lambda)\cdot\lambda$ for $\lambda\in\Lambda$. Moreover, $\mc W(\Lambda)$ is a Coxeter group of which $\Stab(\lambda)$ is a parabolic subgroup, and the partial order on $\Lambda$ given by $u\cdot\lambda\leq v\cdot\lambda$ when $L(u\cdot\lambda)$ is a subquotient of the standard Verma module $M(v\cdot\lambda)$ is exactly the Bruhat order on $\mc W(\Lambda)/\Stab(\lambda)$ in the antidominant case, or its dual in the dominant case (see e.g. \cite{KK}, \cite[Section 5.2]{Hu}).

In \cite{Co} it is shown that for any finite-dimensional semisimple Lie algebras $\mf g$ and $\mf g'$ (i.e. Kac-Moody algebras with $\mc W$ finite), two blocks $\mc O_\Lambda,\mc O'_{\Lambda'}$ of their respective $\mc O$ categories are equivalent if and only if their corresponding posets $\mc W(\Lambda)/\Stab(\lambda)$ and $\mc W'(\Lambda')/\Stab(\lambda')$ are isomorphic. But we have now proven Theorem~\ref{thm:one} which shows there are no non-trivial isomorphisms between infinite posets, and consequently we can extend the proof of Theorem 4.2.1 in \cite{Co} to all symmetrizable Kac-Moody algebras (although not quite all blocks). First we restate a result from \cite{BS} making use of their analysis of upper finite and lower finite highest weight categories:

\begin{theorem}\label{thm:semiinfinite}
If there is a dominant weight $\lambda\in\Lambda$ then $\mc O_\Lambda$ is the full subcategory of $\widehat{\mc O}_\Lambda$ consisting of all modules $M$ such that $[M:L(\lambda)]<\infty$, and $\mc O_\Lambda$ is upper finite. If there is an antidominant weight $\lambda\in\Lambda$ then $\mc O_\Lambda$ is the full subcategory of $\widehat{\mc O}_\Lambda$ consisting of all modules of finite length, and $\mc O_\Lambda$ is lower finite. In both cases the standard and costandard objects are the Verma modules and dual Verma modules corresponding to elements of $\Lambda$.
\end{theorem}

This is an extension of \cite[Theorem 6.4]{BS}. While the original theorem is stated in the context of only affine Kac-Moody algebras and integral weights, the proof therein extends naturally to all symmetrizable Kac-Moody algebras and dominant/antidominant weights. Now we complete the proof of Theorem~\ref{thm:two}:

\begin{proof}[Proof of Theorem~\ref{thm:two}]
It is shown in \cite{Co} that finite highest weight categories with a simple-preserving duality have a unique poset structure, which in the case blocks of category $\mc O$ is exactly the Bruhat order on $\mc W(\Lambda)/\Stab(\lambda)$. This is extended to the lower and upper finite cases in \cite{Co2}, thus by Theorem~\ref{thm:semiinfinite} if $\mc O_\Lambda\simeq\mc O'_{\Lambda'}$ are equivalent and there are $\lambda\in\Lambda,\lambda'\in\Lambda'$ both dominant or both antidominant then $\mc W(\Lambda)/\Stab(\lambda)\cong\mc W'(\Lambda')/\Stab(\lambda')$.

We also have from \cite[Theorem 4.1]{Fi} that an isomorphism between Coxeter pairs $(\mc W(\Lambda),\Stab(\lambda))\cong(\mc W'(\Lambda'),\Stab(\lambda'))$ induces an equivalence $\widehat{\mc O}_\Lambda\simeq\widehat{\mc O}'_{\Lambda'}$ where $\lambda,\lambda'$ are both dominant or both antidominant (this extends the well-known result of \cite{So}), and consequently by Theorem~\ref{thm:semiinfinite} we also have an equivalence $\mc O_\Lambda\simeq\mc O'_{\Lambda'}$. So, denoting the unique dominant (or antidominant) block of $\mc O$ up to equivalence corresponding to the Coxeter pair $(W,W_J)$ by $\mc O(W,W_J)$, we have the following implications so far:
\[(W,W_J)\cong(U,U_K)\implies\mc O(W,W_J)\simeq\mc O(U,U_K)\implies W^J\cong U^K\]
To complete the proof we must verify that in every case where $(W,W_J)\not\cong(U,U_K)$ but $W^J\cong U^K$ the categories are equivalent, i.e. cases (1)-(6) in Theorem~\ref{thm:one}. For case (6) where $\mc W(\Lambda)=\Stab(\lambda)$, we have $\widehat{\mc O}_\Lambda\simeq\C\lmod$. The remaining cases involve finite groups, so these are already discussed in \cite{Co}. In particular, recalling that the Weyl group of a Kac-Moody algebra must be crystallographic:
\begin{itemize}
	\item Cases (1) and (3) involve $I_2(m)$ for which the only crystallographic cases are $B_2=I_2(4)$ and $G_2=I_2(6)$. These cases can be computed explicitly, see e.g. \cite{St}.
	\item Cases (2) and (4) are known in the literature, see e.g. \cite{ES}.
	\item Case (5) is irrelevant since $H_3$ is non-crystallographic.
\end{itemize}
This completes the proof.
\end{proof}

\subsection{Other blocks and extensions}
While Theorem~\ref{thm:two} only applies to dominant and antidominant blocks, there are some partial results towards other blocks in the literature. Blocks of affine Kac-Moody algebras for instance are organised into positive level (dominant), negative level (antidominant) or critical level, and a study of critical level blocks can be found in \cite{Fi2} and others.

Theorem~\ref{thm:two} is relevant only for crystallographic Coxeter systems since these are the ones that appear as Weyl groups of Kac-Moody algebras. Is there a way to extend the definition of $\mc O(W,W_J)$ to allow for $(W,S)$ non-crystallographic? We propose that this might be done using Soergel modules, for instance via the construction in \cite{St}. In the finite case with $\lambda\in\Lambda$ integral and dominant, a functor $\V_\lambda:\mc O_\Lambda\to C^\lambda\lmod$ fully faithful on projectives can be constructed, where $C$ is the algebra of coinvariants of the Weyl group. This is used to construct an equivalent category to $\mc O_\lambda$ with quivers, since the images of projectives under this functor are found as direct summands of modules of the form $C\otimes_{C^{s_1}}C\otimes_{C^{s_2}}\dots\otimes_{C^{s_n}}\C$ with $s_1\dots s_n$ reduced. As described in \cite{EMTW} $C$ is expressed in terms of a representation of $W$, and this need not be the root system of a Lie algebra. In fact, every finite Coxeter system has a geometric representation, so one could in theory construct an analogue of category $\mc O$ for the noncrystallographic systems $I_2(m)$, $H_3$ and $H_4$.

Assuming this constructed category is a highest weight category, the results of this paper may extend to this new case. However this is only a conjecture; in particular it remains to be checked that the isomorphisms $(I_2(n),A_1)\leftrightarrow(A_{n-1},A_{n-2})$ for $n=5$, $n\geq7$ and $(H_3,H_2)\leftrightarrow(D_6,D_5)$ give equivalences of these categories, which in the latter case would involve a lengthy computation.

\subsection{Oddities of the exceptional cases}
There are some notable commonalities of the exceptional isomorphisms (1)-(5) in Theorem~\ref{thm:one}:
\begin{itemize}
	\item All of the exceptional isomorphisms are finite, meaning there are no non-trivial isomorphisms between $W^J$ and $U^K$ if either $W$ or $U$ is infinite (it is proven in \cite{De} that any non-trivial quotient $W^J$ is finite if and only if $W$ is). 
	\item All of the exceptional isomorphisms have exactly one black node, that is $|J|=|S|-1$ and $W_J$ is a maximal proper parabolic subgroup.
\end{itemize}
These seemingly simple facts are disproportionately difficult to prove, relying on almost all of the working out in Sections~\ref{sec:chainlike} and \ref{sec:sudoku}. It would be worthwhile to find a more direct proof of these results. Moreover, while almost all the exceptional isomorphisms can be sorted into two infinite families (total orders and the $B_n\leftrightarrow D_n$ case), the case $(H_3,H_2)\leftrightarrow(D_6,D_5)$ remains a strange outlier, the significance of which is not clear.

\section{Appendix}

The common names of finite Coxeter systems are listed in the figure below.

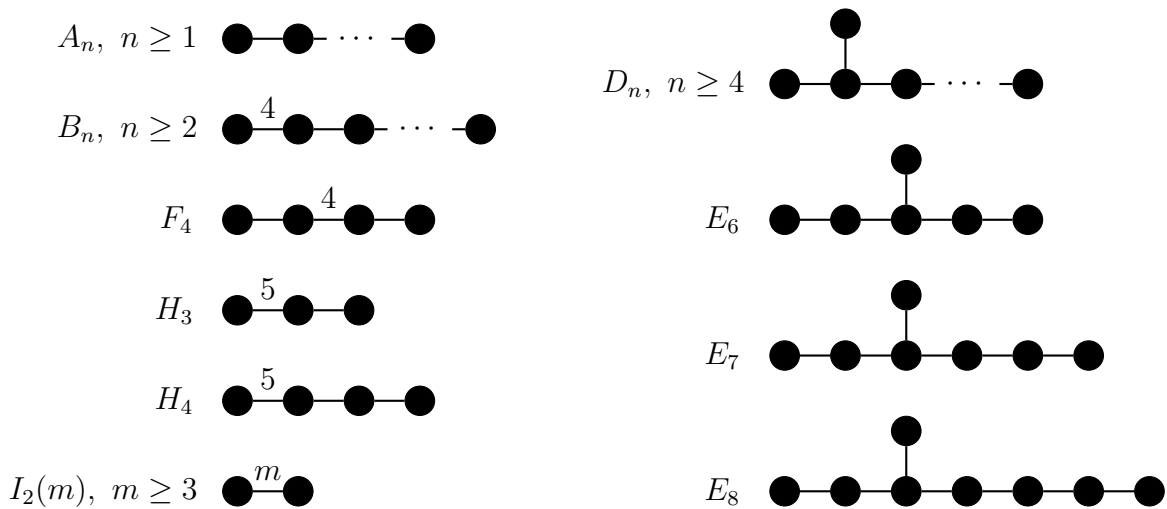
\begin{figure}[h]
\begin{tikzpicture}[scale=0.8]
\begin{scope}
	\begin{scope}[every node/.append style=miniblacknode]
		\node[label=left:{$A_n,\ n\geq1\ $}] (1) at (0,0) {}; \node (2) at (1,0) {}; \node (n) at (3,0) {};
	\end{scope}
	\node (d) at (2,0) {$\cdots$};
	\begin{scope}[every edge/.style={draw=black, thick}]
		\path (1) edge (2); \path (2) edge (d); \path (d) edge (n);
	\end{scope}
\end{scope}
\begin{scope}[shift={(0,-1.5)}]
	\begin{scope}[every node/.append style=miniblacknode]
		\node[label=left:{$B_n,\ n\geq2\ $}] (1) at (0,0) {}; \node (2) at (1,0) {}; \node (3) at (2,0) {}; \node (n) at (4,0) {};
	\end{scope}
	\node (d) at (3,0) {$\cdots$};
	\begin{scope}[every edge/.style={draw=black, thick}]
		\path (1) edge node[above] {$4$} (2); \path (2) edge (3); \path (3) edge (d); \path (d) edge (n);
	\end{scope}
\end{scope}
\begin{scope}[shift={(0,-3)}]
	\begin{scope}[every node/.append style=miniblacknode]
		\node[label=left:{$F_4\ $}] (1) at (0,0) {}; \node (2) at (1,0) {}; \node (3) at (2,0) {}; \node (4) at (3,0) {};
	\end{scope}
	\begin{scope}[every edge/.style={draw=black, thick}]
		\path (1) edge (2); \path (2) edge node[above] {$4$} (3); \path (3) edge (4);
	\end{scope}
\end{scope}
\begin{scope}[shift={(0,-4.5)}]
	\begin{scope}[every node/.append style=miniblacknode]
		\node[label=left:{$H_3\ $}] (1) at (0,0) {}; \node (2) at (1,0) {}; \node (3) at (2,0) {};
	\end{scope}
	\begin{scope}[every edge/.style={draw=black, thick}]
		\path (1) edge node[above] {$5$} (2); \path (2) edge (3);
	\end{scope}
\end{scope}
\begin{scope}[shift={(0,-6)}]
	\begin{scope}[every node/.append style=miniblacknode]
		\node[label=left:{$H_4\ $}] (1) at (0,0) {}; \node (2) at (1,0) {}; \node (3) at (2,0) {}; \node (4) at (3,0) {};
	\end{scope}
	\begin{scope}[every edge/.style={draw=black, thick}]
		\path (1) edge node[above] {$5$} (2); \path (2) edge (3); \path (3) edge (4);
	\end{scope}
\end{scope}
\begin{scope}[shift={(0,-7.5)}]
	\begin{scope}[every node/.append style=miniblacknode]
		\node[label=left:{$I_2(m),\ m\geq3\ $}] (1) at (0,0) {}; \node (2) at (1,0) {};
	\end{scope}
	\begin{scope}[every edge/.style={draw=black, thick}]
		\path (1) edge node[above] {$m$} (2);
	\end{scope}
\end{scope}
\begin{scope}[shift={(9,-0.75)}]
	\begin{scope}[every node/.append style=miniblacknode]
		\node[label=left:{$D_n,\ n\geq4\ $}] (1) at (0,0) {}; \node (2) at (1,0) {}; \node (3) at (2,0) {}; \node (0) at (1,1) {}; \node (n) at (4,0) {};
	\end{scope}
	\node (d) at (3,0) {$\cdots$};
	\begin{scope}[every edge/.style={draw=black, thick}]
		\path (1) edge (2); \path (2) edge (3) edge (0); \path (3) edge (d); \path (d) edge (n);
	\end{scope}
\end{scope}
\begin{scope}[shift={(9,-3)}]
	\begin{scope}[every node/.append style=miniblacknode]
		\node[label=left:{$E_6\ $}] (1) at (0,0) {}; \node (2) at (1,0) {}; \node (3) at (2,0) {}; \node (4) at (3,0) {}; \node (5) at (4,0) {}; \node (0) at (2,1) {};
	\end{scope}
	\begin{scope}[every edge/.style={draw=black, thick}]
		\path (1) edge (2); \path (2) edge (3); \path (3) edge (4) edge (0); \path (4) edge (5);
	\end{scope}
\end{scope}
\begin{scope}[shift={(9,-5.25)}]
	\begin{scope}[every node/.append style=miniblacknode]
		\node[label=left:{$E_7\ $}] (1) at (0,0) {}; \node (2) at (1,0) {}; \node (3) at (2,0) {}; \node (4) at (3,0) {}; \node (5) at (4,0) {}; \node (6) at (5,0) {}; \node (0) at (2,1) {};
	\end{scope}
	\begin{scope}[every edge/.style={draw=black, thick}]
		\path (1) edge (2); \path (2) edge (3); \path (3) edge (4) edge (0); \path (4) edge (5); \path (5) edge (6);
	\end{scope}
\end{scope}
\begin{scope}[shift={(9,-7.5)}]
	\begin{scope}[every node/.append style=miniblacknode]
		\node[label=left:{$E_8\ $}] (1) at (0,0) {}; \node (2) at (1,0) {}; \node (3) at (2,0) {}; \node (4) at (3,0) {}; \node (5) at (4,0) {}; \node (6) at (5,0) {}; \node (7) at (6,0) {}; \node (0) at (2,1) {};
	\end{scope}
	\begin{scope}[every edge/.style={draw=black, thick}]
		\path (1) edge (2); \path (2) edge (3); \path (3) edge (4) edge (0); \path (4) edge (5); \path (5) edge (6); \path (6) edge (7);
	\end{scope}
\end{scope}
\end{tikzpicture}
\caption{The finite irreducible Coxeter systems and their common names. In each case the index $n$ is the number of generators $|S|$. Note that $I_2(m)$ is also labelled $A_2,B_2,H_2,G_2$ for $m=3,4,5,6$ respectively.}\label{fig:finite}
\end{figure}

\end{document}